\pdfoutput=1
\documentclass[12pt]{amsart}

\usepackage{amsmath,amssymb,graphicx,subcaption} 
\usepackage[margin=1in]{geometry} 
\usepackage{enumerate}
\usepackage[pdftex]{hyperref}
\usepackage{tikz}
\usetikzlibrary{calc,petri,positioning,arrows.meta,cd,arrows,positioning,shapes.symbols}
\usepackage{float}
\usepackage[labelfont=bf]{caption}

\usepackage{multirow}
 \usepackage{mathrsfs}
 \usepackage{romannum}

\theoremstyle{plain}
\newtheorem{lem}[subsubsection]{Lemma}
\newtheorem{cor}[subsubsection]{Corollary}

\newtheorem{pro}[subsubsection]{Proposition}
\newtheorem{defn}[subsubsection]{Definition}

\newtheorem{thm}[subsubsection]{Theorem}
\theoremstyle{remark}
\newtheorem{rem}[subsubsection]{Remark}
\newtheorem{exa}[subsubsection]{Example}

\newcommand{\LM}{\text{LM}}
\newcommand{\NF}{\text{NF}}
\usepackage{algorithm}
\usepackage[noend]{algorithmic}

\begin{document}

\title[Massively Parallel Modular Methods in Algebra and Geometry]{Massively Parallel Modular Methods in Commutative Algebra and Algebraic Geometry}

\thanks{We would like to thank Fraunhofer ITWM for providing access to their computer infrastructure. HPR was supported by the DAAD In-Region program. Gef\"ordert durch die Deutsche Forschungsgemeinschaft (DFG) - Projektnummer 286237555 - TRR 195 [Funded by the Deutsche Forschungsgemeinschaft (DFG, German Research Foundation) - Project- ID 286237555 - TRR 195]. The work of JB was supported by Project B5 of SFB-TRR 195. MM acknowledges support of SFB-TRR 195. }

\keywords{Computer algebra, modular methods, Singular, distributed computing, GPI-Space, Petri nets,
computational algebraic geometry, birational geometry.} 

\author{Dirk Basson}
\address{Dirk Basson, Department of Mathematical Sciences,
Mathematics Division,
Stellenbosh University,
Private Bag X1, 7602, Matieland, South Africa}
\email{djbasson@sun.ac.za}

\author{Janko~B\"ohm}
\address{Janko~B\"ohm, Department of Mathematics, University of Kaiserslautern,  Erwin-Schr\"odinger-Str., 67663 Kaiserslautern, Germany}
\email{boehm@mathematik.uni-kl.de}

\author[Magdaleen Marais]{Magdaleen S. Marais}
\address{Magdaleen S. Marais, Department of Mathematical Sciences,
Mathematics Division,
Stellenbosh University,
Private Bag X1, 7602, Matieland, South Africa}
\email{msmarais@sun.ac.za}

\author{Mirko Rahn}
\address{Mirko Rahn, Competence Center High Performance Computing,
Fraunhofer-Institut für Techno- und Wirtschaftsmathematik ITWM,
Fraunhofer-Platz 1,
67663 Kaiserslautern, Germany}
\email{mirko.rahn@itwm.fraunhofer.de}

\author[HP Rakotoarisoa]{Hobihasina P. Rakotoarisoa}
\address{Hobihasina P. Rakotoarisoa,  Department of Mathematical Sciences,
Mathematics Division,
Stellenbosh University,
Private Bag X1, 7602, Matieland, South Africa}
\email{24283967@sun.ac.za}

\maketitle

 \pagenumbering{arabic}

\begin{abstract}

Computations over the rational numbers frequently encounter the issue of intermediate coefficient growth. Modular methods provide a solution to this problem by applying the algorithm under consideration modulo a number of primes and then lifting the modular results to the rationals. 
We present a novel, massively parallel framework for modular computations with polynomial data. Our approach introduces a data structure which is used in the modular computations and covers a broad spectrum of applications in commutative algebra and algebraic geometry. We demonstrate the framework's adaptability and effectiveness in applications such as Gröbner basis computations in characteristic zero and algorithmic methods from birational geometry. To support this application, we develop algorithms, and modular variants thereof, to compute images and domains of rational maps, as well as determining invertibility and computing inverses.

Our implementation is based on the \textsc{Singular}/\textsc{GPI-Space} framework, which uses the computer algebra system \textsc{Singular} as computational backend and user interface, while coordination and communication of parallel computations is handled by the workflow management system \textsc{GPI-Space}, which relies on Petri nets as its mathematical modeling language. Convenient installation is realized through the package manager \textsc{Spack}. Utilizing Petri nets for the coordination of the individual computational activities in the modular algorithm, our approach provides automated parallelization and balancing of the load between computation, lifting, stabilization testing, and potential verification. Via stabilization testing, our approach automatically finds with high probablity a minimal set of primes required for the successfull reconstruction. The framework employs error tolerant rational reconstruction that effectively manages the problem of bad primes, ensuring correctness and termination as long as for a fixed computation there exist only finitely many bad primes.

We present timings which illustrate that our framework has the potential for a game changing improvement of performance over previous modular and non-modular methods, considering characteristic zero Gröbner basis computations, and computations for rational maps. In particular, we illustrate that the approach scales very well with the number of processor cores used for the computation.\medskip

\noindent \textbf{Math. Subj. Class. (2020).} Primary 68W10; Secondary 68W30, 68Q85, 14Q99.

\end{abstract}

\

\section{Introduction}

In computer algebra, intermediate coefficient growth in algorithms over characteristic $0$ is a significant problem both with regard to run-time and memory usage. One way to address this issue is to resort to constant-size coefficient arithmetic. One way to do this, is to apply modular methods. For the computation of Gr\"{o}bner bases of ideals over the rational numbers, this approach was initially established in \cite{arnold2003modular}. Modular methods typically involve the following sub-steps: reducing the input modulo several primes $p$, that is mapping the input over $\mathbb Q$ to input over $\mathbb Z/p$, if well-defined; applying the algorithm under consideration with a deterministic answer over $\mathbb{Z}/p\mathbb{Z}$; lifting of the modular results via the Chinese Remainder Theorem to  $\mathbb{Z}/N\mathbb{Z}$ with $N$ the product of the primes $p$; applying rational reconstruction to lift the result to a result over $\mathbb{Q}$; and potentialy verifying the result to obtain a mathematically proven answer. In the case of Gröbner bases this can be done using Buchberger's criterion to verify that the result is a Gröbner basis, together with reducing the original generators to ensure that the result is a Gröbner basis of the given ideal. The verification step, which may in some cases be very expensive, is necessary, since $N$ may be too small or some of the used primes $p$ may be \emph{bad primes}, that is primes modulo which the result over the rationals does not reduce to the result computed over $\mathbb Z/p$.  Arnold noticed in the case of computing reduced Gröbner bases for ideals that a prime $p$ is good, that is not bad, if  the lead monomials over $\mathbb Z/p$ agree with the result over $\mathbb Q$. However, the result over $\mathbb Q$ is not known a priori.  Arnold's observation was used in \cite{idrees2011parallelization} to turn modular methods for Gröbner bases into an actual algorithm and implementation by introducing a majority vote over the sets of lead monomials for the different primes, as well as a  stabilization test under adding a result modulo a prime not considered yet. This results in a randomized algorithm, which delivers a correct result with high probability. In general, if we lift an ideal or module (produced by a deterministic algorithm) where a generating set is not known a priori (which is the usual setting if an algorithm is determining a new mathematical object), there can exist bad primes which cannot be detected directly, or eventually be eliminated by a majority vote. This, in particular, this may lead to the computation not terminating. The error tolerant rational reconstruction described in \cite{Boehm2012} solves the problem and ensures termination if for a fixed input there exist only finitely many bad primes for the computation, and thus the good primes eventually dominate the bad primes.

In this paper, we develop a generic massively parallel framework for modular computations with polynomial data. We introduce a data structure for the input and output data of the framework, which covers all relevant use cases in commutative algebra. 

One use case for which the framework was build, in particular, is algorithms from birational geometry. The framework might, for example, be of use for an algorithmic approach for the minimal model program, see \cite{lazic2023programming}. We specifically develop algorithms for handling  the computation of images of rational maps of algebraic varieties, determining invertibility of rational maps, computation of inverses, and domains of rational maps. 

We use the language of Petri nets to develop our parallel algorithms following the idea of separation of coordination and computation by Gelernter \cite{gelernter1992coordination}. Our implementation is based on the \textsc{Singular}/\textsc{GPI-Space} framework \cite{boehm2018massively,singgspc}, which uses the computer algebra system \textsc{Singular} for polynomial computations \cite{Singular} as computational backend and user interface, while coordination and communication of parallel computations is handled by the workflow management system \textsc{GPI-Space} \cite{gpispace}, which relies on Petri nets as its mathematical modeling language. Utilizing Petri nets for the coordination of the individual computational activities in the modular algorithm, our approach provides automated parallelization and balancing of the load between the different components of the generic modular approach: reduction of the input modulo primes, computation of results in finite characteristic, lifting using the Chinese remainder theorem and error tolerant rational reconstruction, stabilization testing to provide an indication that the set of good primes is large enough (resulting in a randomized algorithm which, in practice, gives reasonably trustworthy results), as well as a potential verification (to obtain mathematically proven results). The Petri nets developed as part of our framework address frequently occuring algorithmic patterns which are of relevance far beyond modular methods.

We present timings which illustrate that our framework shows significant performance improvements over previous methods in characteristic zero Gröbner basis computations and computations for rational maps.
We observe that the modular framework works equally well for multiple cases with varying properties, for example characteristic zero Gr\"{o}bner basis computations where the positive characteristic Gr\"{o}bner basis computations are the bottle neck, or where the lifting or the stabilization test is the bottle neck. Similar observations apply for images of rational maps. 

In Section \ref{sec:1}, we first recall the key idea of rational reconstruction via the Farey map. We also outline error tolerant rational reconstruction. We develop a  generic modular framework for deterministic algorithms in commutative algebra. In particular, a new data structure is introduced as well as all the corresponding tools for modular algorithms. We demonstrate the use of the framework by applying it to the case of Gr\"{o}bner bases of ideals and subquotients of modules over a polynomial ring. In Section \ref{sec:5} this framework will be the starting point of our parallel one.

In Section \ref{sec3}, we outline algorithms to compute for a rational map of projective varieties graph, image, and domain, as well as a set of representatives covering the domain. We develop an effective criterion for birationality,
which is based on  \cite{Simis2004}, and give an algorithm to compute the inverse.

In Section \ref{sec4}, we explain the benefit of the separation of the coordination layer and computation layer in the implementation of parallel algorithms. We recall the notion of Petri nets which can be considered as a coordination model for parallel tasks. Finally, we give a short overview how the workflow management system  \textsc{GPI-Space} relies on Petri nets to describe asynchronous activities.

In Section \ref{sec:5}, we provide a model for modular algorithms in terms of Petri nets. This section contains exhaustive details of all the subnets and transitions that appear in the model. The exactness and termination of the framework is explained in the last section. Our approach is implemented in the open source software framework \textsc{gspc-modular}. For the code, documentation and installation instructions, see  \cite{modular}. 

In Section \ref{sec:6}, we time the use of generic massively parallel modular algorithm for the computation of Gr\"{o}bner basis and the computation of the image of rational maps.

\section{Modular methods in commutative algebra}\label{sec:1}

\subsection{Rational reconstruction}\label{rational reconstruction}
In problems over $\mathbb Q$, with big intermediate results, rational reconstruction in conjunction with the Chinese Remainder Theorem is often used. To improve the performance further, computations can be done in parallel. For this we need an algorithm solving the problem under consideration that is also applicable over finite fields, and of which the result is deterministic. The approach is as follows: Map the input data to $\mathbb Z/N$, where $N=p_1\cdots p_n$ is a product of primes. Apply the Algorithm over $\mathbb Z/p_i$ using parallel methods. Lift these results, using the Chinese remainder Theorem, to a result over $\mathbb Z/N$. Then reconstruct the result via rational reconstruction in $\mathbb Q$.

In order to map $x =\frac{a}{b}\in\mathbb Q$ to $\mathbb Z/N$, given that  $\text{gcd}(b,N) = 1$, we define the residue class for $x$ modulo $N$ as $$x_N = \overline{a} \cdot {\overline{b}}^{-1} \in \mathbb{Z}/N \mathbb{Z}.$$ 

For the rest of the section we will discuss how to reconstruct results in $\mathbb Q$ from results in $\mathbb Z/N$. One way to recover a rational number from its class modulo $N$ is to use  the  Farey map:
$$
\begin{array}{ccc}
\varphi_N :  \left\{\frac{a}{b} \in \mathbb{Q} \, \Big| \, \begin{array}{c}

\text{gcd}(a,b)=1\\
\text{gcd}(b,N)=1
\end{array}\,
|a|,|b| \leq \sqrt{(N-1)/2} \right\} &\longrightarrow &\mathbb{Z}/N \mathbb{Z}  \\
\frac{a}{b}&\longmapsto& \overline{a} \cdot {\overline{b}}^{-1}
\end{array}
$$
which is injective.

Algoritm 2 in \cite{Boehm2012}  gives an efficient way to reconstruct preimages of the Farey map (see also \cite{COLLINS1995287} and \cite{Kornerup1983}). Note that for our approach the restricted domain of the Farey map is no problem, since we can choose $N$ as large as we want. This approach has the following flaw. There may exist primes $p$ for which a result of a given algorithm over $\mathbb Q$ mapped via the Farey map to $\mathbb Z/p$, does not map to the result computed over $\mathbb Z/p$. We call such primes {\bf bad}. In such a case our approach may give a wrong result. There are typically very few bad primes, they can mostly be detected a priori and, if the primes are chosen randomly, hardly create any practical difficulties. Still there exist no theoretic argument to eliminate all possible bad primes for an algorithm and the result obtained by using the above method of rational reconstruction may not be correct. See Example \ref{exm:recostruction} for an illustration that a prime can be bad without any apparent reason, like dividing numerators or denominators of the relevant rational numbers.

This problem was solved by an error tolerant method of reconstruction presented in \cite{Boehm2012}. 
For this method, consider the subset $C_N \subset \mathbb{Z}/N \mathbb{Z}$ defined by $\overline{r} \in C_N$, if there exists $u, v , q \in \mathbb{Z}$, such that
\begin{itemize}
	\item[(i)] $u  \geq 0, v \neq 0$ and $\text{gcd}(u,v) = 1$,
	\item[(ii)] $q \geq 1$ and $q$ divides N,
	\item[(iii)] $u^2 + v^2 < \frac{N}{q^2}$ and $u \equiv vr \pmod{\frac{N}{q}}$.
\end{itemize}

In Lemma 4.2 in \cite{Boehm2012} it is proven that the rational number $\frac{u}{v}=\frac{qu}{qv}$ is uniquely determined by (iii). That implies that all vectors with $u^2+v^2<N$ in the following lattice represent the same rational number.
Therefore the map 
$$\psi_N : C_N \to \mathbb{Q},\ \psi(\overline{r}) =
\frac{u}{v}$$ is well-defined. Consider the lattice 
$$\Lambda =  \Lambda_{N,r} = \langle (N,0), (r,1)\rangle \subset 
\mathbb{Z}^2.$$

\begin{lem}(Lemma 4.2, \cite{Boehm2012})
	All $(x,y) \in \Lambda $ with $x^2 +y^2 < N $ are collinear. That is $\frac{x}{y}$ 
	defines a unique rational number.\label{lemma:1}
\end{lem}

\begin{rem}
We have the inclusion $\text{im}(\varphi_N) \subset C_N$. Indeed, for each $\frac{u}{v}\in\mathbb Q$ in the domain of $\varphi_N$, set $q=1$.
\end{rem}

The following theorem states that if $N$ is large enough $\overline{r}\in\mathbb Z/N$ will have a preimage in~$\mathbb Q$. Furthermore, if there exists a preimage,
 then, in particular, any of the shortest vectors in the lattice $\Lambda$ will represent the preimage.

 Now, let $N',M\ge 2$ be integers, with $N=N'M$ and $\gcd(N',M)=1$. Let $a\ge 0$, $b\neq 0$ be integers such that $\gcd(b,N')=1$ and $a\equiv bs\mod N'$, with $0\le s\le N'-1$. In other words, $(\frac{a}{b})_{N'}=s$. Let $0\le r\le N$, be such that $\overline{r}=\overline{s}\in\mathbb Z/N'$. Let $0\le t\le M-1$ be such that $\overline{r}=\overline{t}\in\mathbb Z/M$. Furthermore suppose that $\psi_M(t)\neq\frac{a}{b}$. For practical purposes, we may think of $N'$ as the product of good primes and $M$ as the product of bad primes. It is proven in Lemma 4.3 in \cite{Boehm2012} that $\psi_N(r)=\frac{a}{b}$ regardless of the fact that $\psi_M(r)=\psi_M(t)$ gives a wrong result, provided that $M<<N'$.

\begin{thm} (Lemma 4.3, \cite{Boehm2012})\label{Lemma4.3}
We use the above notation. Suppose that $(a^2+b^2)M<N'$. Then, for all $(x,y)\in\Lambda$, with $x^2+y^2<N,$ we have $\frac{x}{y}=\frac{a}{b}$. Furthermore, if $\gcd(a, b) = 1$ and $(x, y)$ is a shortest nonzero vector in $\Lambda$,
we also have $gcd(x, y)|M$.

\end{thm}

We therefore can adopt the following error tolerant algorithm to reconstruct a rational number $x$ from $x_N\in Z/N$.

	\begin{algorithm}[ht]
		\begin{algorithmic}[1]
			\REQUIRE a product of distinct primes $N=p_1,\cdot\ldots\cdot p_n$ and a tuple $(\overline{r_1},\ldots,\overline{r_n})\in\mathbb Z/p_1\times\cdots \times\mathbb Z/p_n$.
   			\ENSURE a rational reconstruction of a representative of $(\overline{r_1},\ldots,\overline{r_n})$ in $\mathbb Z/N$.
			\STATE use the Chinese remainder isomorphism to lift the tuple $(r_1,\ldots,r_n)\in \mathbb Z/p_1\times\cdots \times\mathbb Z/p_n$ to $\overline{r}\in\mathbb Z/N$.
            \STATE compute a shortest vector $(a,b)$ in $\Lambda_{N,r}$.
			\STATE verify the result as \texttt{true} or \texttt{false}.
            \IF{the result is \texttt{true}}
            \RETURN{$\frac{a}{b}$}
            \ELSE
            \RETURN{\texttt{false}}
            \ENDIF
		\end{algorithmic}
  \caption{ \cite{Boehm2012} Rational reconstruction}
		\label{algo:reconstruct}
	\end{algorithm}

\begin{rem}
If Algorithm \ref{algo:reconstruct} gives $\texttt{false}$ as output, we can repeat the procedure by adding a new prime $p_{n+1}\neq p_1,\ldots, p_n$ to the product of primes $N=p_1\cdots p_n$, computing a new representation of the number that needs to be lifted in $\mathbb Z/p_1\times\cdots \times\mathbb Z/p_n$, and applying the algorithm again. If, eventually, the good primes outnumber the bad primes the Algorithm will give a correct answer.
\end{rem}

\begin{exa}
\label{exm:recostruction}
	To illustrate Algorithm \ref{algo:reconstruct}, we consider the example to reconstruct $x=\frac{8}{7}$ using the primes $3, 5 , 11$ and $103$. From the Chinese remainder isomorphism
	$$\chi : \mathbb{Z}/3 \mathbb{Z} \times \mathbb{Z}/5 \mathbb{Z} \times \mathbb{Z}/11 \mathbb{Z} \times \mathbb{Z}/103 \mathbb{Z}\longrightarrow \mathbb{Z}/16995 \mathbb{Z}.$$
	The residue class of $x $ in $\mathbb{Z}/3 \mathbb{Z} \times \mathbb{Z}/5 \mathbb{Z} \times \mathbb{Z}/11 \mathbb{Z} \times \mathbb{Z}/103 \mathbb{Z}$ is
	\begin{align*}
	(\overline{8}\cdot \overline{1}, \overline{8} \cdot \overline{3}, \overline{8} \cdot \overline{8}, \overline{8} \cdot \overline{59}) &= (\overline{2}, \overline{4}, \overline{9},\overline{60})
	\end{align*}
	and $$\chi(\overline{2},\overline{4},\overline{9},\overline{60})=\overline{2429}.$$
	The error tolerant algorithm reconstructs $\frac{8}{7}$ from $\overline{2429}$ with the modulus $16995$.
	
	Now, assume that we made a mistake instead of getting $(\overline{2}, \overline{4}, \overline{9}, \overline{60})$ we have $(\overline{2}, \overline{0}, \overline{9}, \overline{60})$.\newline
	Then $$\chi((\overline{2}, \overline{0}, \overline{9}, \overline{60})) =
	\overline{16025}$$  and we reconstruct $\frac{8}{7}$ using the error tolerant
	algorithm.  Note that the Farey preimage does not return a reconstruction at this stage.
	
	If we change the  congruence modulo the prime  $103$ to $\overline{61}$. We get $$\chi({\overline{2},\overline{4},\overline{9},\overline{61}}) = \overline{3254}.$$
	Then we reconstruct the wrong result $-\frac{17}{47}$. In this case, the prime
	$103$ is a bad prime. 
	
	Generally, it is impossible to detect if a prime is bad so at the end of the algorithm we need to check the correctness of the result.
	\end{exa}
\subsection{Generic framework for modular algorithms in commutative algebra}
\label{sub:genericFramework}In Section \ref{rational reconstruction} we gave an approach to solve problems given by a deterministic algorithm over the rational numbers, that is also applicable over $\mathbb Z/p$, via the Chinese Remainder Theorem. In this section we extend this approach to algorithms over tuples of sets of vectors in a free module over $\mathbb Q$, that is also applicable over tuples of sets of vectors in a free module over $\mathbb Z/p$. This includes algorithms over important structures in commutative algebra, for example, polynomial rings, ideals, modules, quotient rings etc. After we have given a generic algorithm for this approach we will conclude the section by giving some important examples.

We use the following notation:
 $$R_{0}=\mathbb{Q}[X_{1},\ldots,X_{n}],\  R_{N}=\mathbb{Z}/N[X_{1},\ldots,X_{n}]$$ and $$R\left\langle m_{1},\ldots,m_{s}\right\rangle =\mathcal{M}(R^{m_{1}})\times\ldots\times \mathcal{M}(R^{m_{s}}),$$ with $m_i \geq 1 $  an integer, for $i=1,\ldots,s$, where $$\mathcal M (R^n) = \{ A\subset R^n \mid \left| A\right| < \infty \text{ and } 0\not\in A\}.$$

Consider a deterministic algorithm $\mathcal A$ which takes as input an element $P\in
R_{0}\left\langle m_{1},\ldots,m_{s}\right\rangle $ (that is, a tuple of sets
of vectors in a free module over $\mathbb Q$) and computes an element $$Q(0)=\prod_{i=1}^{t}Q_i\in R_{0}\left\langle
n_{1},\ldots,n_{t}\right\rangle $$ such that each $Q_{i}$, for $i=1,\ldots,t$, is
a reduced Gr\"{o}bner basis with respect to the global monomial ordering $>_{i}$ on
$R_0^{n_{i}}$. The algorithm under consideration also needs to be applicable over $R_{p}$ for $p$ a prime. We denote $\mathcal A$ applied over $R_p$ as $\mathcal A_p$ 
\begin{defn}
For a
set of vectors $A\subset R^{m}$ we define the set of lead monomials as
\[
\operatorname*{LM}(A)=\left\{  L(a)\mid a\in A\right\}
\]
where $L(a)$ is the lead monomial of $a$ with respect to a monomial ordering $>$ on $R^{m}$.
\end{defn}

In order to use modular methods via the Chinese Remainder Theorem we need to map our input in $R_0\langle m_1,\ldots,m_s\rangle$ to $R_N\langle m_1,\ldots,m_s\rangle$, where $N$ is an integer. For that purpose we need the reduction map.

\begin{defn} (Reduction modulo $N$)
For a polynomial $f=\sum_ic_iv_i\in R_{0}$, where the $c_i's$ are the coefficients in $\mathbb Q$ and the $v_i's$ are the monomials, we define the reduction $f_{N}\in R_{N}$ modulo
$N$ as $f_N=\sum_i\overline{c_i}v_i$, where $\overline c_i$ is the residue class of $c_i$ modulo $N$.\\
For%
\[
P=(\left\{  f_{1,1},\ldots,f_{1,w_{1}}\right\}  ,\ldots,\left\{
f_{s,1},\ldots,f_{s,w_{s}}\right\}  )\in R_{0}\left\langle m_{1},\ldots
,m_{s}\right\rangle
\]
we define%
\[
P_{N}=(\left\{  (f_{1,1})_{N},\ldots,(f_{1,w_{1}})_{N}\right\}  ,\ldots
,\left\{  (f_{s,1})_{N},\ldots,(f_{s,w_{s}})_{N}\right\}  )
\]
whenever all denominators of coefficients (in their coprime representation) are coprime to $N$, and consider the \textbf{reduction map}%
\[%
\begin{tabular}
[c]{llll}%
$\phi_{N}:$ & $D_{N}$ & $\rightarrow$ & $R_{N}\left\langle m_{1},\ldots
,m_{s}\right\rangle $\\
& $P$ & $\mapsto$ & $P_{N}$%
\end{tabular}
\]
with
\[
D_{N}=\left\{  P\in R_{0}\left\langle m_{1},\ldots,m_{s}\right\rangle \mid
P_{N}\text{ well-defined}\right\}  \text{.}%
\]
\end{defn}

Using the notation in Section \ref{rational reconstruction}, we assume $N=p_1\cdot\ldots\cdot p_r$ is a product of distinct primes. Following our approach: after we have reduced the given input data modulo $N$, we apply our given deterministic algorithm over $R_{p_j}$ with output $Q_{p_j}=(Q_{p_j})_1\times\ldots\times(Q_{p_j})_t\in R_{p_j}\langle n_1,\ldots,n_t\rangle$, for each $j$. To lift the modular results we need to lift the sets $(Q_{p_j})_i$ for each $i=1,\ldots,t$. In order to do that the sets need to be compatible, that is it needs to be ordered so that, for a fixed $i$, the $k$'th vectors of each of the sets is liftable as an element of an Cartesian product by the Chinese Remainder Theorem. Knowing that for good primes $p_j$, each of the sets $(Q_{p_j})_i$ is a reduced Gr\"{o}bner basis, for a fixed $i$, the sets $(Q_{p_j})_i$ should have the same lead ideal for all $j$. Hence, for all good primes $p_j$ ordering the elements in each set according to their lead monomials, using $>_i$, will assure compatibility. The flattening map is used for that purpose. 

\begin{defn}
For a polynomial ring $R$, we define the \textbf{flattening map}%
\[
\omega:U\rightarrow R^{\infty}%
\]
with%
\begin{multline*}
U= \Bigl\{  (\left\{  f_{1,1},\ldots,f_{1,w_{1}}\right\}  ,\ldots,\left\{
f_{s,1},\ldots,f_{s,w_{s}}\right\}  )\in R\left\langle m_{1},\ldots
,m_{s}\right\rangle \mid \\
 L(f_{i,1}),\ldots,L(f_{i,w_{i}})\text{ pairwise
distinct }\forall i \Bigr\}
\end{multline*}
by mapping%
\[
(\left\{  f_{1,1},\ldots,f_{1,w_{1}}\right\}  ,\ldots,\left\{  f_{s,1}%
,\ldots,f_{s,w_{s}}\right\}  )\mapsto\operatorname*{st}(f_{1,\sigma^{1}%
(1)},\ldots,f_{1,\sigma^{1}(w_{1})},\ldots,f_{s,\sigma^{s}(1)},\ldots
,f_{s,\sigma^{s}(w_{s})})
\]
where $\sigma^{i}\in S_{w_{i}}$ is the permutation, which sorts $f_{i,1}%
,\ldots,f_{i,w_{i}}$ according to their lead monomial, and
\[%
\begin{tabular}
[c]{llll}%
$\operatorname*{st}:$ & $R^{a_{1}}\times\ldots\times R^{a_{l}}$ &
$\rightarrow$ & $R^{a_{1}+\ldots+a_{l}}$\\
& $\left(  \left(
\begin{array}
[c]{c}%
g_{1,1}\\
\vdots\\
g_{a_{1},1}%
\end{array}
\right)  ,\ldots,\left(
\begin{array}
[c]{c}%
g_{l,l}\\
\vdots\\
g_{a_{l},l}%
\end{array}
\right)  \right)  $ & $\mapsto$ & $\left(
\begin{array}
[c]{c}%
g_{1,1}\\
\vdots\\
g_{a_{1},1}\\
\vdots\\
g_{a_{l},l}%
\end{array}
\right)  $%
\end{tabular}
\]
For integers $N_{1}, \dots, N_{r}$, and%
\[
T_r=R_{N_{1}}\left\langle m_{1},\ldots,m_{s}\right\rangle
\times\ldots\times R_{N_r}\left\langle m_1,\ldots,m_s\right\rangle\]
we define the \textbf{compatible flattening map} as%
\[%
\begin{tabular}
[c]{llll}%
$\tau:$ & $E_r$ & $\rightarrow$ & $R_{N_{1}}^{\infty}\times\ldots\times R_{N_{r}}^{\infty}$\\
& $(a_1,\ldots,a_r)$ & $\mapsto$ & $(\omega(a_1),\ldots,\omega(a_r))$%
\end{tabular}
\]
where%
\[
E_r=\left\{  (a_1,\ldots,a_r)\in T\mid\operatorname*{LM}(a_{1,i})=\ldots=\operatorname*{LM}%
(a_{r,i})\text{ },\forall i=1,\ldots,s\right\}
\]
\end{defn}

Since there may be bad primes in the set $\{p_j\mid j=1,\ldots,r\}$, the product $Q_{p_1}\times\ldots\times Q_{p_r}$ may not be in $E_r$. To rectify the product such that the different $Q_{p_j}$'s are compatible, we can use a \begin{bf}majority vote\end{bf}, that is we keep a biggest set of primes such that $Q_{p_1}\times\ldots\times Q_{p_k}\in E_k$ and replace the old product by the new one. To ensure a correct result, after applying the Chinese Remainder Theorem and lifting the result to $R_0\langle m_1,\ldots, m_s\rangle$, the set $\mathbb P=\{p_1,\ldots,p_k\}$ of primes should be sufficient, as defined below:

\begin{defn}
Let $\mathbb P$ be a finite set of primes. Let $\mathcal A$ be a deterministic algorithm as described at the beginning of the section, with input $P\in R_0\langle m_1,\ldots,m_s\rangle$ and output $Q(0)=\prod_{i=1}^{t}Q_i\in R_0\langle m_1,\ldots,m_s\rangle$. Let $N'$ be the product of all good primes and $M$ be the product of all bad primes. The set $\mathbb P$ is said to be sufficiently large if $N'>(a^2+b^2)M$ for all the  coefficients $\frac{a}{b}$ of each of the $Q_i$'s.
\end{defn}

Now in the first iteration of the loop the number of bad primes may be more than the number of good primes. Using the majority vote as described above, we may keep a set of bad primes. Since the good primes are discarded, it is possible that only bad primes are accumulated in each loop. Hence we will never reach a sufficiently large set of primes, which implies that the algorithm is not error tolerant. To avoid such a situation, we do a weighted cardinallity count. We weigh the vote of the primes in such a way that the sum of the votes of the newly added primes in the while-loop weights strictly more than the sum of the votes of the primes in previous loops. We call such a voting system a {\bf weighted majority vote}.    

After, for a fixed $i$, sorting and stacking the vectors in the sets $(Q_{p_j})_i$ into one vector, such that our result is liftable to $R_N^{\infty}$, we apply the Chinese Remainder Theorem, coefficient by coefficient, for each compatible element of the cartesian product $R_{p_1}^\infty\times\ldots\times R_{p_k}^\infty$ of stacked vectors. 

\begin{defn}
The \textbf{Chinese remainder map}%
\[
\operatorname*{chrem}:\tau(E)\rightarrow R_{N_{1}\cdot\ldots\cdot N_{r}}^{\infty}%
\]
is defined by applying the Chinese remainder theorem
coefficient-by-coefficient. 
\end{defn}

Finally, we define the Farey map in this context:
\begin{defn}
The \textbf{Farey map} is
\[%
\Psi_{N}:F\rightarrow R_{0}^{\infty}%%
\]
with domain%
\[
F=\left\{  X\in R_{N}^{\infty}\mid\text{each coefficient in }X\text{ is in }%
C_{N}\right\},
\]
defined by applying the map $\psi_{N}:C_{N}\rightarrow\mathbb{Q}$
defined in Section \ref{rational reconstruction} to each coefficient appearing in $X$. 
\end{defn}

Finally, we apply%
\[
\omega^{-1}:\omega(U)\rightarrow U\subset R_0\left\langle m_{1},\ldots
,m_{s}\right\rangle
\]
in case the Farey lift $Q'$ is in $\omega(U)$, that is, the lead monomials of each of the sets of components in $Q'$ representing a reduced Gr\"{o}bner basis, stay pairwise destinct. 

Since there is no proof that the set of primes are sufficiently large a priori and the verification test may be computationally expensive, it  makes sense to do another test, the {\bf pTest} in positive characteristic:
Randomly choose a prime number $p
\not \in \mathbb P$ such that $p$ does not divide any coefficient in $P$ or $\omega^{-1}(Q')=Q$. Return {\bf true} if $Q_p=\mathcal A_p(P_p)$. 

In Algorithm \ref{generic} we give an algorithmic framework to outline the above approach.
	\begin{algorithm}[ht]
		\begin{algorithmic}[1]
			\REQUIRE $P\in
                     R_{0}\left\langle m_{1},\ldots,m_{s}\right\rangle $, $P_i\in R_0^{m_i}$, a deterministic algorithm $\mathcal A$ over $R_0=\mathbb Q[X_1,\ldots,X_n]$ that is also applicable over $R_p=\mathbb Z/p[X_1,\ldots,X_n]$, and computes an element $Q(0)=\prod_{i=1}^{t}Q_i\in R_{0}\left\langle
n_{1},\ldots,n_{t}\right\rangle $ such that each $Q_{i}$, for $i=1,\ldots,t$, is
a reduced Gr\"{o}bner basis with respect to the global monomial ordering $>_{i}$ on
$R_0^{n_{i}}$, as well as, a verification test.
			\ENSURE $Q=\mathcal A(P)$.
			\STATE choose a finite random set $\mathbb P=\{p_1,\ldots, p_r\}$ of distinct primes $\mathcal{P}$, and let $N=p_1\cdot\ldots\cdot p_r$. 
            \WHILE{true}
            \STATE set result $=$ true.
			\STATE compute $\Phi_N(P)=P_N$, where $\Phi_N$ is the reduction map modulo $N$.
			\STATE form the products $Q_{p_j}=(Q_{p_j})_1\times\ldots\times (Q_{p_j})_t$ by applying $\mathcal A_{p_j}$, $j=1,\ldots,r$ to $P_{p_j}$.
            \STATE do a 
            \begin{bf}weighted majority vote\end{bf} on the $p_j'$s, finding a biggest set of $p_j$'s, such that $LM(Q_{p_j})$ is the same for all the $p_j'$s in the set, in such a way that the sum of the votes of primes generated before the current while-loop weigh strictly less than the sum of the votes of primes generated in the current loop. Let $\{p_1,\ldots,p_k\}$ be a set of primes that won the vote.
            \STATE let $S=Q_{p_1}\times\ldots\times Q_{p_k}$
            \STATE let $S'=\text{chrem}\circ\tau(S)$
            \IF{$S'$ is in the domain of the Farey map}
            \STATE $Q'=(q_1,\ldots,q_\nu)=\Psi(S')$
                \IF{for all $i\neq j$,  $L(q_i)\neq L(q_j)$}
                    \STATE $Q=\omega^{-1}(Q')$
                    \IF{$Q$ pass the {\bf PTest}}
                        \IF{$Q$ pass the verification test}
                        \RETURN $Q$
                        \ENDIF
                    \ENDIF
                \ENDIF
            \ENDIF
            \STATE let $\{p_{k+1},\ldots,p_{k+r}\}$ be a set of  primes, such that none of the primes are previously chosen, and let $N=p_1,\ldots,p_{k+r}$.
            \STATE $r=r+k$.
            \ENDWHILE
		\end{algorithmic}
		\caption{Generic modular framework for deterministic algorithms in commutative algebra}
  \label{generic}
	\end{algorithm}

We now discuss some examples to which Algorithm \ref{generic} is applicable.

For computational purposes, we represent an ideal (resp. module) by a set of generators. To ensure uniqueness, we represent the ideal (resp. module) by a reduced Gr\"{o}bner basis. To compute a reduced Gr\"{o}bner basis of an ideal $I\in R_0$ (resp. module $M\in R_0\langle m_1\rangle$) we can use any algorithm $\mathcal A_G$ computing a reduced Gr\"{o}bner basis of an ideal (resp. module) over $R_0$, with input any set of generators of the ideal (resp. module), that is also applicable over $\mathbb R/p$, $p$ a prime. To parallelize $\mathcal{A_G}$ we can apply Algorithm \ref{generic} to $I$ (or $M$) and $\mathcal A_G$. As verification test we can check whether the output $G$ is a reduced Gr\"{o}bner basis generating $I$ (resp. $M$). For Cartesian products of ideals or modules a Gr\"{o}bner basis of each of the components can be computed, similarly, simultaneously.
This implies that we can combine the algorithm $\mathcal A_G$ with any deterministic algorithm with input $P\in R_0\langle m_1,\ldots,m_s\rangle$ over $R_0$, that is also applicable over $R_p$, with output in $R_0\langle n_1,\ldots,n_t\rangle$, to create an algorithm to which Algorithm \ref{generic} is applicable.

Lastly, we discuss how we can parallelize algorithms with input that is Cartesian products of subquotients of modules over $R_0$ via Algorithm \ref{generic}. We can represent a subquotient of a module over $R_0$ by a module of generators $S$ and a module of relations $T$. Clearly $T\subset S$.  
If we have a deterministic algorithm over $R_0$, that is also applicable over $R_p$, $p$ a prime, with output a subquotient $V_0=S_0/T_0$ over $R_0$, and output $V_p=S_p/T_p$ over $R_p$, we need to find a way to represent  $V_p$ such that for different primes the output is compatible to be lifted by the Chinese Remainder Theorem. Algorithm \ref{generic} requires that the output of the given algorithm $\mathcal{A}$ computes an element $Q(0)=\prod_{i=1}^{t}Q_i\in R_{0}\left\langle
n_{1},\ldots,n_{t}\right\rangle $ such that each $Q_{i}$, for $i=1,\ldots,t$, is
a reduced Gr\"{o}bner basis with respect to the global monomial ordering $>_{i}$ on
$R_0^{n_{i}}$. Since reduced Gr\"{o}bner bases are unique, it is easy to determine which elements of the Gr\"{o}bner bases over the different primes, in case the primes are good primes, should be lifted together, by looking at the lead monomials of the elements. 
Taking a reduced Gr\"{o}bner basis $G_{T_0}=\{g_1,\ldots,g_n
\}$ for $T_0$, $T_0$ can be represented uniquely. Since $T_0\subset S_0$, we can extent $G_{T_0}$ to a Gr\"{o}bner basis for $G_{S_0}=\{g_1,\ldots,g_n,g_{n+1},\ldots,g_m\}$. Let $\overline{g}_i=\NF(g_i \mid \{\overline{g}_1,\ldots \overline{g}_{i-1}\})$, where $\NF(f \mid G)$ is the reduced normal form of $f$ with respect to $G$. Then $g_1=\overline{g}_1,\ldots, g_n=\overline{g}_n$ and $\{\overline{g}_1,\ldots,\overline{g}_m\}$ is a reduced Gr\"{o}bner basis of $S_0$, and hence unique. Hence, representing a product of subquotients $\prod_{i=1}^{t}S_i/T_i$ by $Q(0)=\prod_{i=1}^{2t}G_i$, where $G_i$ is the reduced Gr\"{o}bner basis of $S_{\frac{i+1}{2}}$ for $i$ uneven, and  $G_i$ is the reduced Gr\"{o}bner basis of $T_{\frac{i}{2}}$ for $i$ even, we can apply Algorithm \ref{generic} to a deterministic algorithm over $R_0$, with input in $R_0\langle m_1,\ldots,m_s\rangle$, that is applicable over $R_p$, $p$ a prime, computing Cartesian products of subquotients over $R_0$.

\section{Rational Maps Toolbox} \label{sec3}
One of the standard approaches of modern mathematics to understand mathematical
objects is to study maps between the objects. It is, in particular, useful to study maps which conserve
interesting properties of the objects under consideration. For example, in group theory, we study
group homomorphisms which preserve the group operations. This section will focus
on rational maps between projective varieties and especially describe their image, domain and inverse in case if it is birational.
\newline
As is practise in Projective algebraic geometry, we work in a projective space $\mathbb{P}_K^n$ over a
algebraically closed field $K$. A bit of terminology, we call \textbf{quasi-projective variety} an open subset of a projective variety. 
We represent an element of
$\mathbb{P}_K^n$ by homogeneous coordinates $(p_0:{\dots}: p_n)$. The results of this section can be found in \cite{boehm2017} and \cite{Simis2004}. We provide new proofs and relation between the two resources. The first section recalls the backgrounds about rational maps. The second section presents a algorithm to compute the image of a rational map. The third section talks about a test of birationality as well as a algorithm for computing the inverse of birational map. The last section gives  an algorithm to determine the domain of a rational map.

\subsection{Basic}
We start by giving basic definitions related to rational map. Then, we show how rational maps are represented in practice.

\begin{defn}[Morphisms and Rational maps]   \ 
	
	\label{defn:rational}
\begin{itemize}
  \item[\textup{(i)}] Let $X \subset \mathbb{P}_K^m$ be a quasi-projective variety. Then, a function $f:X \rightarrow K $ is called a \textbf{regular function} at a point $p \in \mathbb{P}_K^m$ if there is an open neighbourhood $U \subset X$ of $p$ and homogeneous polynomials $g,h \in K[t_0,\dots, t_m]$ of the same degree such that, for all $q \in U$, we have 
  $$f(q) = \frac{g(q)}{h(q)} \text{ and } h(q) \neq 0.$$
  If $f$ is regular at each point of $X$, we say that $f$ is regular on $X.$ We denotes by $\mathscr{O}(X)$ the $K$-algebra of regular functions on $X$ and $\mathscr{O}_X(U)$ the $K$-algebra of regular function $U \rightarrow K$ on  an open subset $U $ of $X$. 
  \item[\textup{(ii)}] A \textbf{morphism of quasi-projective varieties}  $$f : X \subset \mathbb{P}_K^n  \rightarrow Y \subset \mathbb{P}_K^m $$ is a continuous map such that for every open subset $U $ of $Y$ and regular function $\varphi \in \mathscr{O}_Y(U) $,  the pull-back function $$ f^* \varphi := \varphi \circ f : f^{-1} (U) \rightarrow K$$ is a regular function in $\mathscr{O}_X(f^{-1}(U)).$ 
  
	\item[\textup{(iii)}] Let $X,Y$ be irreducible projective varieties. A \textbf {rational map}  $\Phi$ from $X$ to $Y$, denoted by
	$$ \Phi : X \dashrightarrow Y $$ is an equivalence class of morphisms from a non-empty open subset $U \subset X$ to $Y$. Two such morphisms $U \rightarrow Y$ and $V \rightarrow Y$ are equivalent if they agree on $U \cap V.$
\end{itemize}
\end{defn}
\begin{rem}
 The separateness of a projective variety $X \subset \mathbb{P}^m_K$, that is the diagonal morphism $ \Delta : X \rightarrow X \times X $ is a closed immersion ( $\Delta(X) \subset X \times X$ is  closed and $X \rightarrow \Delta(X)$ is an homeomorphism) , is an important property  for the results in this section, since the uniqueness of the extension  of morphisms is satisfied in this setting as is shown in the following proposition.
\end{rem}

\begin{pro}\cite{Liu2006}
	Let $X, Y$ be projective varieties. Let $f$ and $g$ be two morphisms  from $X$ to $Y$. If there is a dense  open  subset $U \subset X$ such that $f_{|U} = g_{|U}$, then the two morphisms $f$ and $g$ are the same. \label{pro:uniqueness}
\end{pro}
\begin{proof}
	See \cite{Liu2006} Proposition 3.11.
\end{proof}
\begin{rem}
	Let  $\Phi :X \dashrightarrow Y$ be a rational map, then Proposition \ref{pro:uniqueness} leads to the  existence of an unique morphism $f: U \rightarrow Y$ in the equivalence class of the rational map $\Phi$, where the open set $U \subset X$ is maximal for the inclusion. That is, if $V \rightarrow X$ is another element in the class, we have $V \subset U$ . Furthermore, every morphism $g:V \rightarrow Y$ in the equivalence class of $\Phi$, satisfy $f_{|V} = g$. Indeed, all the morphisms in the class glue together to a  morphism from $U$ to $Y$. The uniqueness is a consequence of  Proposition \ref{pro:uniqueness}. We call such an open subset $U$ the \textbf{domain of definition} $D(\Phi)$ of $\Phi$. \label{rem:domain}
\end{rem}
\begin{pro}
	Let $X \subset \mathbb{P}_K^n$ be a projective variety and let $f_0, \dots, f_m \in K[t_0,\dots,t_n]$ be homogeneous polynomials of the same degree. Consider the open subset $U = X \setminus V(f_0, \dots, f_m) $ of $X$. Then the map
	$$ f : U \rightarrow \mathbb{P}_K^m, \, t \mapsto (f_0(t): \cdots : f_m(t) )$$
	defines a morphism. \label{prop:homogeneouspoly}
\end{pro}

\begin{proof}
	First, note that the map $f$ is well-defined set-theoretically. Let ${(V_i)}_{0 \leq i \leq m}$ be the affine open cover of $\mathbb{P}_K^m$ with $$V_i = \{(x_0: \cdots : x_m) \in \mathbb{P}_K^m \, : \, x_i \neq 0 \}.$$ Then the open subsets  $U_i = f^{-1} (V_i) = \{ t \in U \, : \, f_i(t) \neq 0\} \subset U$ cover $U$. It follows that in the affine coordinate of $V_i$, the map $f_{|U_i}$ is defined by the tuple $(\frac{f_0}{f_i} , \dots , \hat{\frac{f_i}{f_i}} , \cdots , \frac{f_m}{f_i})$. Since each quotient of polynomials $\frac{f_j}{f_i}$ for $j=0, \dots, m$, with $j \neq i$, is regular function over $U_i$, it follows that $f_{| U_i}$ is a morphism. Hence, using the gluing property of a morphism, $f$ is a morphism.
\end{proof}
\begin{rem}\label{rem:homogeneous2}
	It is important to note that not every morphism $U \subset \mathbb{P}_K^n \rightarrow \mathbb{P}_K^m$ is induced by homogeneous polynomials of the same degree. Instead we have the following:\newline
 If $f:U \subset \mathbb{P}_K^n \rightarrow \mathbb{P}_K^m$ is a  morphism, $f$ is locally defined by homogeneous polynomials of the same degree. That is, for every element $a \in U$, there is an open subset $U_a \subset U$  and homogeneous polynomials $f_0, \dots, f_m \in K[t_0,\dots, t_n]$ of the same degree, such that $a \in U_a $ and $f(t) = (f_0(t):\cdots:f_m(t))$ for all  $t \in U_a$. Indeed, using the affine open cover ${(V_i)}_{0 \leq i \leq m}$ of $\mathbb{P}_K^m$ that we have seen  in the proof of Proposition \ref{prop:homogeneouspoly}, we get a morphism $f_{|U_i} : U_i \rightarrow V_i$ with $U_i = f^{-1}(V_i)$ for $0 \leq i \leq m$. Since the open subset $U_i$ cover $U$, by symmetry we can assume that $a \in U_0.$ As $V_0 \subset \mathbb{A}_K^m$ is affine, the morphism $f_{|U_0}$ has the form
	$$ f_{|U_0} : U_0 \rightarrow V_0, \, t \mapsto (\varphi_1(t),\dots,\varphi_m(t) )$$
	where the $\varphi_i$'s are regular functions in $\mathscr{O}_U(U_0)$ for $1 \leq i \leq m.$ By definition, for each $1 \leq i \leq m$, there exist an open subset $X_i \subset U$ and homogeneous polynomials of the same degree $ f_i, g_i$ such that $a \in X_i$ and $\varphi_i(t) = \frac{f_i(t)}{g_i(t)}$ for all $t \in X_i.$ \\ Set $U_a = \bigcap_{i=1}^m X_i$. Then $$f_{|U_a} : U_a \rightarrow V_0, \, t \mapsto \left(\frac{f_1(t)}{g_1(t)}, \dots, \frac{f_m(t)}{g_m(t)} \right)$$ in the affine coordinate of $V_0$. In the homogeneous coordinate of $\mathbb{P}_K^m$, we have 
	$$ f_{|U_a} : U_a \rightarrow \mathbb{P}_K^m,\, t \mapsto \left(1: \frac{f_1(t)}{g_1(t)}: \cdots: \frac{f_m(t)}{g_m(t)}  \right)$$
	Let $g = \prod_{1}^{m} g_i$ and $\hat{g_i} = \frac{g}{g_i}$. It follows that 
	$$ f_{|U_a} \rightarrow \mathbb{P}_K^m, \, t \mapsto \left(g(t):f_1(t) \cdot \hat{g_1}(t):\cdots: f_m(t) \cdot \hat{g_m}(t)  \right).$$
	This result allows us to restrict our study of rational map to only studying morphisms induced by homogeneous polynomials as in Proposition \ref{prop:homogeneouspoly}. 
 \end{rem}

\begin{rem}
	 From now, by Remark \ref{rem:homogeneous2},  we define a \textbf{rational map} of irreducible projective varieties $ \Phi : X \subset \mathbb{P}_K^m \dashrightarrow Y \subset \mathbb{P}_K^n$ by given one representative in the class induced by homogeneous polynomials of the same degree  in $R = K[t_0,\dots,t_m]/ I(X)$, and we write
	$$ \Phi : X \dashrightarrow Y, \, t \mapsto (\overline{f_0}(t): \cdots : \overline{f_n}(t)). $$
	This is well defined because $\overline{f_i}(t)$ for  $0\leq i \leq n$  is well defined for all $t\in X$. We denote the representative of the rational by the tuple $(\overline{f_0} , \dots , \overline{f_n}) \in R^{n+1}.$
		\end{rem}

  \begin{defn}[Graph and image of rational map] \, \\
	Let $X=V(I) \subset \mathbb{P}_K^n$ and $Y \subset \mathbb{P}_K^m$ be irreducible projectives varieties. Let $\Phi : X \dashrightarrow Y  $ a rational map. We define the \textbf{graph} of $\Phi$ to be the subset 
		$$ \Gamma(\Phi) = \left\{(t,x),\, t \in D(\Phi) \text{ and } x=\Phi(t) \right\} \subset \mathbb{P}_K^m \times  \mathbb{P}_K^n,$$
		and the \textbf{image} of $\Phi$ to be the closure in $Y$ of the set-theoretical image of the associated morphism $D(\Phi) \rightarrow Y$ in Proposition \ref{rem:domain}. 
\end{defn}

\begin{exa}
    Take $K = \mathbb{C}$ and consider the rational map defined by 
 $$\Phi : \mathbb{P}_K^1 \dashrightarrow \mathbb{P}_K^2 , \, t \mapsto (t_0^2+t_1^2:t_0^2-t_1^2:2 t_0t_1).$$
Clearly, the vanishing locus $V(t_0^2+t_1^2,t_0^2-t_1^2,2 t_0t_1) = \emptyset \subset \mathbb{P}_K^1$, so $D(\Phi) = \mathbb{P}_K^1$ and $\Phi$ is a morphism.\newline
 We will see later that 
 $$\text{im}(\Phi) = V (x_1^2+x_2^2-x_0^2) \subset \mathbb{P}_K^2.$$
\end{exa}

\begin{exa}[Veronese map]
    Consider the rational rational map defined by
 $$\begin{array}{lcll}
  \Phi : & \mathbb{P}_\mathbb{C}^2 & \longrightarrow & \mathbb{P}_\mathbb{C}^5\\
  & (t_0:t_1:t_2) & \longmapsto & (t_0^2:t_1^2:t_2^2: 2 t_0 t_1:2 t_0 t_2:2 t_1 t_2).
 \end{array}
$$
Clearly, the rational map $\Phi$ is a morphism. That is the domain $D(\Phi) = \mathbb{P}^2_\mathbb{C}$.
The image of $\Phi$ is called the Veronese surface $S \subset \mathbb{P}^5_\mathbb{C}.$
\end{exa}

\begin{rem}
    We note that even rational map behaves similar to an ordinary map, the composition of rational maps not always defined. Let $\Phi : X \dashrightarrow Y$ and $\Psi : Y \dashrightarrow Z$ two rational maps of irreducible projective varieties. It may occur that  the image of $\Phi$ lies out of  any dense open subset of $Y$, where $\Psi$ is defined. In this case, the composition $\Psi \circ \Phi $ is not defined even as a rational map. The composition is defined as an equivalence class of morphisms $\Psi_{\restriction_V} \circ \Phi_{\restriction_U}$ over the dense open subset  $\Phi^{-1}_{\restriction U}(V)$, where $U \subset X$ and $V \subset Y$ are dense open subsets such that the restriction $\Phi_{\restriction_U}$ (respectively $\Psi_{\restriction_V})$ is a morphism in the equivalent class of $\Phi$ (respectively $\Psi$). This identification makes sense if $\Phi^{-1}_{\restriction_U}(V) \neq \emptyset$ exists. It is the case if    the 
     rational map $\Phi$ is \textbf{dominant}, that is  $\overline{\operatorname{im}(\Phi)} = Y$,. 
\end{rem}

In birational geometry, rational maps are extensively used to classify projective varieties by birational equivalence.

\begin{defn}
Let  $\Phi : X \dashrightarrow Y$ be a dominant rational map between irreducible  projective varieties. The rational map $\Psi$ is called \textbf{birational} if there exists a rational maps $\Psi : Y \dashrightarrow X $ such that $\Phi \circ \Psi = \operatorname{Id}_Y$ and $\Psi \circ \Phi = \operatorname{Id}_X$ as rational maps. We then say $X$ and $Y$ are birational.
\end{defn}
Examples of rational maps that are birational are given later.
Before describing the graph of a rational map explicitly by saturation, we first dive into elimination theory. Quotients and saturations play an important role in elimination theory. We will recall them.
\begin{defn}[Quotients and saturations of ideals]
 Let $I, J \subset K[t_1,\dots,t_n] $ be ideals.
 \begin{itemize}
  \item[(i)] The \textbf{ quotient} of $I$ by $J$ is the ideal 
  $$I : J = \{f \in K[t_1,\dots , t_n] \, | \, f g \in I, \, \, \forall g \in J \}.$$
  \item[(ii)] The \textbf{saturation} of $I$ by $J$  is the ideal 
  $$I : J^{\infty} = \{f \in K[t_1,\dots , t_n] \, | \, \exists N \in \mathbb{N} :  f g^N \in I, \, \, \forall g \in J \} .$$
 \end{itemize}
\end{defn}

The two definitions are related by the following proposition.
\begin{pro}\cite{David2007}
 If $I$, $J \subset K[t_1,\dots, t_n]$ are ideals. Then
 \begin{itemize}
  \item[(i)] $I : J^{\infty} = I:J^N$ for all sufficiently large $N$.
  \item[(ii)] $\sqrt{I:J^\infty} = \sqrt{I}:J$.
 \end{itemize}
\label{pro:iteration}
\end{pro}

\begin{rem}
\label{rem:iteration}
 The Proposition \ref{pro:iteration} leads to an algorithm  computing saturations by iterating the ideal quotients until it stabilizes since we have the equality
 $$ (I:J):L=I:(J\cdot L)$$
 for all ideals $I,J$ and $L$ in $K[t_1,\cdots, t_n]$.
\end{rem}
A geometric interpretation of ideal quotients and saturations are  given by the following propositions.
\begin{pro}\cite{David2007}
 If $I,J \subset K[t_1,\dots,t_n]$ are ideals, then
$\overline{V(I)\backslash V(J)} \subset V(I:J)$
where $V$ is the affine vanishing locus.\label{pro:quotient}
\end{pro}

\begin{exa}\cite{David2007}
\label{exa:quotient}
 Let $I=\langle x^2 (y-1) \rangle$ and $J= \langle x\rangle$ ideals in $\mathbb{C}[x,y]$. A simple geometric interpretation shows that $\overline{V(I)\backslash V(J)} = V(y-1)$. However, the ideal quotient is 
 \begin{align*}
  I:J & = \left\{f \in  \mathbb{C}[x,y] \,; \, fg \in I \, \, \forall g \in J  \right\}\\
  & = \{ f \in \mathbb{C}[x,y] ; fx\in I \}\\
  & = \{ f \in \mathbb{C}[x,y] \,\,; \, \, \exists A \in   \mathbb{C}[x,y], \, fx =A x^2(y-1) \}\\
  & = \{ f \in \mathbb{C}[x,y] \,\,;\,\, \exists A \in \mathbb{C}[x,y], \, f = A x(y-1)  \}\\
  & = \langle x(y-1) \rangle.
 \end{align*}
\end{exa}
This example shows that the vanishing locus $V(I:J)$ does not match with 
$\overline{V(I)\backslash V(J)}$. In order to obtain an equality, we need  
ideal other than the quotient $I:J$.
\begin{pro}\cite{David2007}\\
 If $I,J$ are ideals in $K[t_1,\dots, t_n]$ and $K$ is algebraically closed fields, then $$V(I:J^\infty) = \overline{V(I)\backslash V(J)}.$$
\end{pro}

The following proposition  and theorem lead us to an  algorithm for computing 
quotients and saturations.
\begin{pro} \label{pro:sum}\cite{David2007}
 Let $I$ and $J_1, \dots J_r$ be ideals in $K[t_1,\dots, t_n]$ then 
 \begin{itemize}
  \item[\textup{(i)}] $I:(\sum_{i=1}^{r} J_i) = \bigcap_{i=1}^{r} (I:J_i)$,
  \item[\textup{(ii)}] $I:{(\sum
  _{i=1}^{r} J_i)}^\infty =\bigcap_{i=1}^{r} {(I:J_i^\infty)} .$
 \end{itemize}
\end{pro}
\begin{thm}\cite{David2007}
 Let $I \subset K[t_1,\dots, t_n]$ an ideal  and $g \neq 0 \in K[t_1,\dots,t_n]$. Then
 \begin{itemize}
  \item[\textup{(i)}] If $I\cap \langle g \rangle = \langle g f_1,gf_2,\dots , g f_s \rangle$ where $f_i \in K[t_1,\dots,t_n]$, we have $I:\langle g \rangle = \langle f_1,\dots, f_s \rangle.$
  \item[\textup{(ii)}] If $I = \langle h_1, \dots, h_s \rangle$ and $\overline{I}= \langle h_1,\dots, h_s,1-yg \rangle $, we obtain $I:{\langle g\rangle}^\infty = \overline{I} \cap K[t_1,\dots, t_n]$. 
 \end{itemize} \label{thm:saturation}
\end{thm}

\begin{rem}
 Let $I , J \subset K[t_1,\dots, t_n]$ ideals such that $J = \langle g_1,\dots,g_s\rangle $.
 \begin{itemize}
  \item[\textup{(i)}] In order to compute $I:J$, Proposition \ref{pro:sum} implies $I:J = 
\bigcap_{i=1}^s (I:\langle g_i \rangle)$ and we can get  $I:\langle g_i \rangle $ via  Theorem \ref{thm:saturation} (i). The command 
\texttt{quotient(I:J)}  in Singular computes $I:J.$
  \item[\textup{(ii)}] Using the same argument as  in (i),  Theorem \ref{thm:saturation} (ii) yields an algorithm for computing saturations. Note that Singular uses the iterative  quotient algorithm mentioned in Remark \ref{rem:iteration} through the command \texttt{sat(I,J)} in the library \texttt{elim.lib} and gives as output the ideal $I:J^\infty$ and the smallest  $N \in \mathbb{N}$ such that $I:J^\infty =I:J^N $.
 \end{itemize}\label{algo:saturation}
\end{rem}
\begin{exa}
 Consider Example \ref{exa:quotient}, Singular gives the outputs $I:J = \langle x^2 y -x^2 \rangle$ and the saturation $I:J^\infty = \langle y-1\rangle$ with stability index $N=2$, that is $I:J^\infty = I:J^2.$
\end{exa}

In view of describing the image of rational map between projective variety, we introduce an additional tool in the  algebra-geometry correspondence. This is elimination of ideal which corresponds to projection in geometry.\newline
First we recall how it appears in the affine case in the following theorem. 
\begin{thm}\cite{David2007}
 Let $I\subset K[t_1,\dots,t_n]$ ideal with $K$ an algebraically closed field. Set $A = V(I) \subset K^n$ the vanishing locus of $I$ and let $0\leq k <n$. Consider 
 $$\pi_k:K^n \longrightarrow K^{n-k}, (a_1,\dots,a_n) \longmapsto (a_{k+1}, \dots, a_n)$$
 be the projection onto the last $k$ components. Then 
 $$ \overline{\pi_k(A)} = V(I_k)$$
 where $I_k = I \cap K[t_{k+1},\dots,t_n]$ is the elimination ideal of $I$.
\end{thm}
In the projective case, things are more complicated as the following example will show.
\begin{exa}
\label{exa:projection}
Let $I = \langle x+yz,x+xz \rangle \subset \mathbb{C}[x,y,z]$ an ideal and $X = V(I) 
\subset \mathbb{P}^1_\mathbb{C} \times \mathbb{C}$. This makes sense because $I$ 
is homogeneous in the variables $x$ and $y$. Then, it is straightforward to show 
that $$X = \{((0:1),0), ((1:1),-1),((1:1),-1)  \} \subset 
\mathbb{P}^1_\mathbb{C} \times \mathbb{C}.$$ 
So the image of $X$ under the projection $$ \pi : \mathbb{P}^1_\mathbb{C} \times 
\mathbb{C} \longrightarrow K, \, \, ((a_1:a_2),a_3) \longmapsto a_3$$
is $\pi(X) = \{0,-1\}$. If we try to use the affine elimination, we get 
$I \cap \mathbb{C}[z] = \{0\}$  which means that the image is all $\mathbb{C}$. Indeed, in affine point of view $V(I) \subset \mathbb{C}^2 \times \mathbb{C}$ contains the points $(0,0,t)$ for every $t \in \mathbb{C}$ but the point $((0:0),t)$ does not exist in $\mathbb{P}^1_\mathbb{C} \times \mathbb{C}.$\newline
So the affine elimination does not capture the fact that we are in projective space. However, we can extend it by removing points which are not defined in the projective space.
\end{exa}

\begin{thm}[Main theorem of elimination theory]\cite{greuel2008singular}
    Let $Y$ be any quasi-projective variety and $$\pi : \mathbb{P}_K^n \times Y \longrightarrow Y$$ the projection on the second factor. Then $\pi$ is a closed map.
\end{thm}
\begin{proof}
    See \cite{greuel2008singular} Theorem A.7.6 .
\end{proof}

\begin{defn}
 Let $I \in K[t_0,\dots,t_m,x_1,\dots,x_n]$ be an ideal generated by homogeneous polynomials in the variables $t_0,\dots, t_m$. The projective elimination ideal $I$ of  is the ideal
 $$ \hat{I} = (I:{\langle t_0,\dots,t_m \rangle}^\infty) \cap 
K[x_1,x_2,\dots,x_n].$$
\end{defn}
\begin{exa}
 In Example \ref{exa:projection}, we have $$\hat{I}= (I:{\langle x , y \rangle}^\infty) \cap \mathbb{C}[z] = \langle y(y+1)\rangle$$ and $\pi(A) = V(\hat{I})$. The following theorem justifies this observation.
\end{exa}
\begin{thm}\cite{David2007}
 Consider an algebraically closed field $K$. Let $I = \langle f_1,\dots, f_p \rangle$ be an ideal in $K[t_0,\dots,t_m,x_1,\dots,x_n]$ generated by homogeneous polynomials in the variables $t_0,\dots, t_m$. We define 
 $$ V(I) = \left\{(t,x) \in \mathbb{P}^m_K \times K^n  \,\, \Big| \,\, 
 \begin{array}{c}
 f(t,x) = 0 \\ \text{ for all } f \in I \text{ homogeneous in $t_i$}
 \end{array}
 \right\}$$ and 
 $$
 \begin{array}{lcll}
     \pi: & \mathbb{P}^m_K \times K^n & \longrightarrow  & K^n \\
     &((t_0:{\dots} : t_m),(x_1,\dots,x_n)) &  \longmapsto & (x_1,\dots,x_n).
 \end{array}
 $$
 Then, we have
$$ \pi(V(I)) = V(\hat{I}).$$ \label{thm:main}
\end{thm}

\begin{cor}\cite{David2007}
\label{cor:main}
 If $K$ is algebraically closed field, and $I$ a bihomogeneous ideal in $K[t_0,\dots,t_m,x_0,\dots,x_n]$   which means homogeneous both in the variables $t_0,\dots,t_m$ and the variables $x_1,\dots,x_n$. We can define
 $$ V(I) = \left\{(t,x) \in \mathbb{P}_K^m \times \mathbb{P}_K^n \,\, \Big| \,\, 
\begin{array}{c}
f(t,x) = 0 \\
\text{ for all bihomogeneous } f \in I 
\end{array}
\right\}$$ and $$ 
\begin{array}{lcll}
    \pi :& 
\mathbb{P}_K^m \times \mathbb{P}_K^n &\longrightarrow & \mathbb{P}_K^n \\ 
&((t_0:{\dots}:t_m),(x_0:{\dots}:x_n) ) &\longmapsto &(x_0:{\dots}:x_n).
\end{array}
$$
 Then $$ \pi(V(I)) = V(\hat{I}).$$
\end{cor}
\begin{proof}
 First of all we need to show that the ideal $\hat{I} \subset K[x_0,\dots,x_n]$ is homogeneous. Note that  the ideal $I$ is a homogeneous ideal of $K[t_0,\dots,t_m,x_0,\dots,x_n]$. We recall that $$\hat{I} = (I:\langle t_0,\dots,t_m \rangle) \cap K[x_0,\dots,x_n].$$
 Let $f \in \hat{I}$ and write $f = \sum_i f_i$ as sum of homogeneous parts. There are $s_j \in \mathbb{N}$ such that $ f \cdot t_j^{s_j} \in I$ for all $j= 0,\dots, m.$ Then $\sum_i f_i  \cdot t_j^{s_j} \in I$. The expression $\sum_i f_i t_j^{s_j} $ is the decomposition of $f t_j^{s_j}$ as sum of homogeneous parts. \newline
 Since $I$ is a homogeneous ideal, each summand $f_i t_j^{s_j} \in I $ for all $j = 0,\dots,m.$ This means that each homogeneous summand  $f_i \in \hat{I}$ for all $f \in \hat{I}$. This implies that $\hat{I} \subset K[x_0,\dots,x_n]$ is homogeneous. Then $V(\hat{I})\subset \mathbb{P}_K^n $ is well defined.\newline
 Now we observe that the  ideal $I$ also defines a vanishing locus in $\mathbb{P}_K^m \times K^{n+1}$ because it is homogeneous in $t_0, \dots, t_m.$ Theorem \ref{thm:main} gives the cone $$ \pi_a(V(I)) = V(\hat{I}) \subset  K^{n+1}$$ with $\pi_a : \mathbb{P}_K^m \times K^{n+1} \rightarrow K^{n+1}$ the projection on the second factor.\\
It follows that the  corresponding projective variety coincides with $\pi(V(I))$, so
 $$ \pi(V(I)) = V(\hat{I}) \subset \mathbb{P}_K^n.$$
\end{proof} 

\begin{rem}
 The algorithm to compute the elimination ideal $\hat{I}$ needs an elimination ordering for $t_0,\dots,t_m$ on 
 the monomials of $K[t_0,\dots,t_m,x_0,\dots,x_n]$. In our examples, we will always use the product ordering  $>$  of $>_\texttt{dp}$ on the monomials $t^{\alpha}$ in the variables $t_0,\dots,t_m$ and $>_{\text{dp}}$ on the monomials $x^\beta$ in the variables $x_0,\dots,x_n$ defined by 
 $$ t^\alpha x^\beta > t^{\alpha '} x^{\beta'} \Leftrightarrow t^\alpha > t^{\alpha '} \text{ or } (t^\alpha = t^{\alpha '} \text{ and } x^\beta > x^{\beta'}). $$
\end{rem}

Now, we describe the graph of a rational map using saturation.

\begin{pro}
\label{prop:graph}
	Let a rational map $$
 \begin{array}{lcll}
 \Phi :& X = V(I) \subset
	\mathbb{P}_K^m &\dashrightarrow &Y=V(J) \subset \mathbb{P}_K^n \\
 &t &\mapsto &(\overline{f_0}(t):{\dots}: \overline{f_n}(t))
 \end{array}
 $$ with $X$ is irreducible,  $\overline{f}_i \in 
	K[t_0,\dots,t_m]/I $ are homogeneous of the same degree, and $I \subset K[t_0,\dots, t_m]$ is a prime ideal . Then $$\overline{\Gamma(\Phi)} = V\left(\left( I+J
	\right) : {\langle f_0,\dots,f_n \rangle}^\infty \right) \subset \mathbb{P}_K^m \times \mathbb{P}_K^n$$ 
 with 
 $$J = \left\langle \operatorname{minors}_2\begin{pmatrix}
                                 x_0  & \cdots & x_n\\
                                 f_0  & \cdots & f_n
                                \end{pmatrix}
 \right\rangle \subset K[t_0,\dots,t_m,x_0,\dots,x_n] .$$
\end{pro}

\begin{proof}
	We recall that the graph of the rational map $\Phi$ is defined by
	$$ \Gamma(\Phi) = \left\{(t,x),\, t \in D(\Phi) \text{ and } x=\Phi(t) \right\} \subset \mathbb{P}_K^m \times  \mathbb{P}_K^n.$$

 First, let's show that $\Gamma(\Phi) \subset V( (I+J): {\langle f_0,\dots, f_n \rangle}^\infty )$ . Let $f\in \left( I+J
	\right) : {\langle f_0,\dots,f_n \rangle}^\infty$ and $(\overline{t_0}, \dots, \overline{t_m},\overline{x_0},\dots,\overline{x_n}) \in \Gamma(\Phi)$. We have to show that $f(\overline{t_0}, \dots, \overline{t_m},\overline{x_0},\dots,\overline{x_n}) = 0.$\\
	If there is $0\leq i \leq n$ such that $f_i(\overline{t_0},\dots,\overline{t_m}) \neq 0$. There is $N \in \mathbb{N}$ such that $f f_i^N \in I+J.$
	This leads to 
 $$
 f(\overline{t_0}, \dots, \overline{t_m},\overline{x_0},\dots,\overline{x_n}) f_i^N(\overline{t_0},\dots,\overline{t_m})  =        0   
 \text{ and } 
 f(\overline{t_0}, \dots, \overline{t_m},\overline{x_0},\dots,\overline{x_n}) = 0 $$

	It follows that 
 $$ f(\overline{t},\overline{f_0}(\overline{t}),\dots,\overline{f_n}(\overline{t})) = 0  \text{ for all } \overline{t} \in X \setminus V(f_0, \dots, f_n).$$
 Now, if $f_i(\overline{t_0},\dots, \overline{t_m}) = 0$ for all $0 \leq  i \leq n.$ Then, there exists another representative $(\overline{g_0},\dots,\overline{g_n})$,  with $g_i \in K[t_0,\dots,t_m]$ homogeneous polynomials of the same degree , of the rational map $\Phi$ such that $\Phi(\overline{t}) = (\overline{g_0}(\overline{t}): \dots : \overline{g_n}(\overline{t}))$. \\
 Set $g(t) = f(t,g_0(t),\dots , g_n(t)).$ Then $g(\overline{t}) = 0 $ for all $$ \overline{t} \in U = \left(X \setminus V(f_0,\dots,f_n) \right) \cap  \left(X \setminus V(g_0,\dots,g_n) \right) \subset X.$$
 This means that $U \subset V(I + \langle g \rangle) \subset X.$ On the other hand, we have $X = U \cup (X\setminus U)$. This implies that $X = V(I + \langle g \rangle) \cup (X \setminus U).$\\
 As $X$ is an irreducible variety, we must have $V(I + \langle g \rangle) = X$ or $X\setminus U = X.$ By definition, $X \setminus U  \neq X$ because $U \neq \emptyset$. Then $$V(I + \langle g \rangle) = X = V(I)$$. It follows that $g \in I$ and $g(\overline{t}) = 0 \text{ for all } \overline{t} \in X$. \\
    It remains to show that the set $V((I+J) : {\langle  f_0, \dots, f_n\rangle}^\infty )$ is the smallest closet subset containing the graph $\Gamma(\Phi)$ in $\mathbb{P}^m \times \mathbb{P}^n.$ Let $L \subset K[t_0,\dots, t_m, x_0, \dots, x_n]$ a bihomogeneous ideal such that $\Gamma(\Phi) \subset V(L) \subset \mathbb{P}^m \times \mathbb{P}^n.$ 
	Let $f \in L$ and bihomogeneous. For all $\overline{t} = (\overline{t_0}:\dots: \overline{t_m}) \in V(I)$, we have $$f(\overline{t_0},\dots,\overline{t_m},f_0(\overline{t}), \dots, f_n(\overline{t})) = 0$$. This means that $f(t_0,\dots,t_m,f_0,\dots,f_n) \in I.$\\
	By Taylor's Theorem in several variables, there are $\alpha_i \in K[t_0,\dots,t_m,x_0,\dots,x_n,\lambda]$  such that
	\begin{align*}
	f(t,x) & = f(t_0,\dots,t_m,\lambda f_0 +x_0 - \lambda f_0, \dots,\lambda f_n + x_n - \lambda f_n)\\
	& = f(t_0,\dots,t_m,\lambda f_0,\dots,\lambda f_n)+\sum_{i=0}^{i=n}\alpha_i (x_i - \lambda f_i)\\
	& = \lambda^N f(t_0,\dots,t_m,f_0,\dots,f_n)+\sum_{i=0}^{i=n}\alpha_i (x_i - \lambda f_i)
	\end{align*}
  with $f(t,x) = f(t_0,\dots,t_m,x_0,\dots,x_n).$
	By substituting $\lambda = \frac{x_j}{f_j}$ with $0\leq j \leq n$ , we get
	\begin{align*}
	f(t,x) & = \left(\frac{x_j}{f_j}\right)^N f(t_0,\dots,t_m,f_0,\dots,f_n)+\sum_{i=0}^{i=n}{\alpha_i}_{|\lambda=\frac{x_j}{f_j}} (x_i - \frac{x_j}{f_j} f_i)
	\end{align*}
	It follows that there is $M \in \mathbb{N}$ sufficiently large such that 
	\begin{align*}
	f(t_0,\dots,t_m,x_0,\dots,x_n) f_j^M &\in I+J
	\end{align*}
	which means that $f \in \left( I+J
	\right) : {\langle f_0,\dots,f_n \rangle}^\infty$. Thus $$V((I+J):{\langle f_0,\dots,f_n \rangle}^\infty) \subset V(L)$$

\end{proof}

Combining Proposition \ref{prop:graph} and elimination, we are able to describe the image of rational map in  term of elimination.

\subsection{Modular method for computing the image of rational map}
This section aims to show the process of computing the image of a rational map. The main idea is to eliminate the source variables from the graph. Then, the resulting ideal defines the closure of the image a rational map as illustrated by the following theorem, which is stated in \cite{boehm2017}.
\begin{thm}
Let $K$ be an algebraically closed field, $I \subset K[t_0,\dots,t_m]$ a homogeneous prime ideal and
$$\begin{array}{lcll}
\Phi : & X = V(I) \subset \mathbb{P}_K^m & \dashrightarrow & \mathbb{P}_K^n\\
& t = (t_0:{\dots}:t_m) & \longmapsto & (\overline{f_0}(t):{\dots}: \overline{f_n}(t))
\end{array}
$$
a rational map with $\overline{f_i} \in K[t_1,\dots,t_m]/I$ homogeneous of  the
same degree. Then
$$\overline{\operatorname{Im}(\Phi)} = V((J: {\langle f_0,\dots,f_n \rangle}^\infty)\cap K[x_0,\dots,x_n]) \subset \mathbb{P}_K^n$$
where
$$J = \left\langle I,  \operatorname{minors}_2\begin{pmatrix}
x_0  & \cdots & x_n\\
f_0  & \cdots & f_n
\end{pmatrix}
\right\rangle \subset K[t_0,\dots,t_m,x_0,\dots,x_n] .$$
Moreover, if the domain of definition $D(\Phi) = X$, then the image $\operatorname{Im}(\Phi)$ is Zariski closed.
\end{thm}
\begin{proof}
    By definition, $\operatorname{Im}(\Phi) = \pi(\Gamma(\Phi))$ with $\pi : \mathbb{P}_K^m \times \mathbb{P}_K^n \rightarrow \mathbb{P}_K^n$ the projection on the second factor. Since $\pi$ is continuous, we have $$\overline{\operatorname{Im}(\Phi)} = \overline{\pi(\Gamma(\Phi))} = \pi(\overline{\Gamma(\Phi)}). $$
    Then, by Proposition \ref{prop:graph} and Corollary \ref{cor:main} , we get
    \begin{align*}
            \overline{\operatorname{Im}(\Phi)}  & = \pi(\overline{\Gamma(\Phi)}) \\
            & = \pi\left(V(J:{\langle f_0, \dots, f_n \rangle}^\infty)\right) \\
            & = V\left((L: {\langle x_0, \dots, x_n \rangle}^\infty ) \cap K[x_0,\dots, x_n] \right)
    \end{align*}
    with $L = J:{\langle f_0, \dots, f_n \rangle}^\infty \subset K[t_0,\dots,t_m,x_0,\dots, x_n].$\\
    It remains to show that  
    $V\left((L: {\langle x_0, \dots, x_n \rangle}^\infty ) \cap K[x_0,\dots, x_n] \right) = V(L \cap K[x_0,\dots,x_n]).$
    As $L \cap K[x_0,\dots,x_n] \subset (L: {\langle x_0, \dots, x_n \rangle}^\infty) \cap K[x_0,\dots, x_n]$, we have the inclusion
$$V\left((L: {\langle x_0, \dots, x_n \rangle}^\infty ) \cap K[x_0,\dots, x_n] \right)  \subset  V(L \cap K[x_0,\dots,x_n]).$$
For the other inclusion, let $(\overline{x_0},\dots, \overline{x_n}) \in  V(L \cap K[x_0,\dots,x_n])$ and $ f \in (L: {\langle x_0, \dots, x_n \rangle}^\infty ) \cap K[x_0,\dots, x_n]$. There exists $ 1\leq i \leq n$ such that $\overline{x_i} \neq 0$. On the another, there is $N \in \mathbb{N}$ such that $f x_i^N \in L$. Since $f \in K[x_0,\dots, x_n]$, we have  $f x_i^n \in L \cap K[x_0,\dots, x_n]$.
Then, it follows that $f(\overline{x_0}, \dots, \overline{x_n}) \overline{x_i}^N = 0 $ and $f(\overline{x_0},\dots,\overline{x_n}) = 0.$\\
Now if $D(\Phi) = X$. This means that $f$ is a morphism. Then, the morphism $(\operatorname{Id}_X,f) : X \rightarrow X \times Y, \, \overline{t} \mapsto (\overline{t}, f(\overline{t}))$ is a closed immersion. It follows that the  graph  $\Gamma(\Phi)$ of the rational map $\Phi$ are closed and $\text{Im}(\Phi)= \pi(\Gamma(\Phi))$ is closed by Corollary \ref{cor:main}.

\end{proof}
The algorithm to compute the image of rational maps is presented in Algorithm \ref{algo:image}.

	\begin{algorithm}[ht]
		\caption{\cite{boehm2017} Image of a rational map }
  \label{algo:image}
		\begin{algorithmic}[1]
			\REQUIRE Rational map
			$$\Phi :  X = V(I) \subset
			\mathbb{P}_K^m \dashrightarrow \mathbb{P}_K^n , \, t \mapsto
			(\overline{f_0}(t):{\dots}: \overline{f_n}(t))$$ with $I \subset K[t_0,\dots,t_m]$  a prime ideal and $\overline{f}_i \in 
			K[t_0,\dots,t_m]/I $ 
			\ENSURE Ideal $S \subset K[x_0,\dots,x_n]$ defining the closure of the image of the rational map $\Phi.$
            \STATE Compute the ideal $$J = \left\langle I,\text{minors}_2\begin{pmatrix}
                                 x_0  & \cdots & x_n\\
                                 f_0  & \cdots & f_n
                                \end{pmatrix}
 \right\rangle \subset \mathbb{Q}[t_0,\dots,t_m,x_0,\dots,x_n].$$
            \STATE Compute $S = (J: {\langle f_0,\dots,f_n \rangle}^\infty)\cap K[x_0,\dots,x_n]$.
			\RETURN $S$.
		\end{algorithmic}
	\end{algorithm}

\begin{exa}
\label{exa:stereographic}
    Consider  the rational map
 $$\begin{array}{lcll}
  \Phi : &  \mathbb{P}_\mathbb{C}^1 & \dashrightarrow & \mathbb{P}_\mathbb{C}^2\\
  & (t_0:t_1) & \longmapsto & (\overline{{t^2_0}+ t_1^2} : \overline{t^2_0 - t^2_1 }: \overline{ 2 t_0 t_1}).
 \end{array}
$$
Set $J = \left\langle \text{minors}_2\begin{pmatrix}
                                 x_0  &  x_1 & x_2\\
                                 t^2_0+t^2_1   & t^2_0 - t^2_1 & 2 t_0 t_1
                                \end{pmatrix}
 \right\rangle   \subset K[t_0,t_1,x_0,x_1,x_2]$ . The ideal is generated by they following bihomogeneous polynomials
 \begin{align*}-t_0^2 x_2+2t_0t_1x_1+t_1^2x_2, \\
-t_0^2x_2+2t_0t_1x_0-t_1^2x_2,\\
-t_0^2x_0+t_0^2x_1+t_1^2x_0+t_1^2x_1\end{align*}
 Using the command \texttt{sat} in Singular, the ideal of the graph $$ G = J:{\langle t^2_0 + t^2_1, t^2_0 - t^2_1, 2 t_0 t_1 \rangle}^\infty$$ is generated by
 \begin{align*}x_0^2-x_1^2-x_2^2,\\
t_0x_2-t_1x_0-t_1x_1,\\
t_0x_0-t_0x_1-t_1x_2 \end{align*}
 using as ordering  the product ordering of \texttt{dp} on $t_0, \dots, t_m$ and \texttt{dp} on $x_0,\dots, x_n$.
 Then, the ideal $$ \left(J:{\langle t^2_0 + t^2_1, t^2_0 - t^2_1, 2 t_0 t_1\rangle}^\infty\right) \cap K[x_0,x_1] \subset K[x_0,x_1,x_2] $$ is generated by the polynomial $ x_0^2-x_1^2-x_2^2.$
 It follows that 
 $$\overline{\operatorname{Im}(\Phi)} = V(\langle x_0^2-x_1^2-x_2^2 \rangle) \subset \mathbb{P}^2_\mathbb{C}.$$
 Clearly, the rational map $\Phi$ is a morphism. Then
 $$\operatorname{Im}(\Phi) = V(\langle x_0^2-x_1^2-x_2^2 \rangle) \subset \mathbb{P}^2_\mathbb{C}.$$
\end{exa}

\subsection{Algorithm for computing the inverse of rational map}
In this section, we discuss  the test of birationality of a rational map given in \cite{Simis2004} and show how the  ideal of the graph in Proposition \ref{prop:graph} can be used to know if a rational map is birational. The approach uses a lot of terms from universal algebra, so we first recall them as well as the key theorems needed. In the following, $M$ is a $A$-module and $T(M)$ denotes its tensor algebra, see for more details and proofs \cite{Bourbaki1-3} Chapter \Romannum{3}. \\
Consider the direct sum of $A$-modules $\bigoplus_{n \geq 0} T^n(M)$, where $T^n(M)$ denotes $\bigotimes^{n} M$ the $n$-th tensor power of $M$ with $T^0(M) = A$ and $T^1(M) = M.$ Then, for $p,q > 0 $, $x_i \in M$ and $\alpha \in A$, consider the multiplication

 \begin{equation}
\label{tensorAlgebra}
\left\{ \begin{array}{r@{}l}
    \left(x_1 \otimes \cdots \otimes x_p \right) \cdot (x_{p+1} \otimes \cdots \otimes x_{p+q}) &{}= x_1 \otimes \cdots \otimes x_{p+q} \\
    \alpha \cdot (x_1 \otimes \cdots \otimes x_p ) &{}= \alpha (x_1 \otimes \cdots \otimes x_p). 
\end{array} \right.
\end{equation}

It is clear that this multiplication defines a graded $A$-algebra structure on $\bigoplus_{n \geq 0} T^n(M)$.  The algebra $\bigoplus_{n \geq 0} T^n(M)$ with the multiplication defined in \ref{tensorAlgebra} is called \textbf{tensor algebra} of $M$.

\begin{pro}\cite{Bourbaki1-3}
	Let $E$  be a $A$-algebra and $f: M \to E$ a $A$-linear map. There exists an unique $A$-algebra homomorphism $g:T(M) \to E$ such that $f = g \circ i$, where the map $i$ is the canonical injection $M \to T(M)$. 
\end{pro}
\begin{pro}\cite{Bourbaki1-3}
	Let $A$ a ring, $M$ and $N$ two $A$-modules and 
	$$ u:M \to N $$
	an  $A$-linear mapping. There exists an unique $A$-algebra homomorphism 
	$$T(u): T(M) \to T(N)$$ such that the diagram
	$$
	\begin{tikzcd}
	M \arrow{r}{u} \arrow[swap]{d}{i_M} & N \arrow{d}{i_N} \\%
	T(M) \arrow{r}{T(u)}& T(N)
	\end{tikzcd}
	$$ is commutative. Moreover, the homomorphism $T(u)$ is a graded algebra homomorphism.
\end{pro}

\begin{pro}\cite{Bourbaki1-3}
	If $u:M \to N$ is a surjective $A$-linear mapping , the homomorphism $T(u) : T(M) \to T(N)$ is surjective and its kernel is the two-sided  ideal  of $T(M)$ generated by  $ker(u) \subset M \subset T(M)$. It follows that we have a $A$-algebra epimorphism 
\end{pro}
\begin{defn}[Symmetric algebra]
	Let $A$ be a ring and $M$ an $A$-module. The \textbf{symmetric algebra} of $M$, denoted by $S(M)$ is  the quotient algebra over $A$ of the tensor algebra $T(M)$ by the two-sided ideal generated by $x \otimes y - y \otimes x$ of $T(M)$, where  $x$ and $y$ run through $M$.\\
 If $R$ a ring and $I \subset R$  an ideal, the symmetric algebra of the ideal $I$, denoted by $\mathcal{S}_R(I)$, is the symmetric algebra of $I$ as an $R$-module. 
\end{defn}
\begin{pro}\cite{Bourbaki1-3}
	Let $E$ be an $A$-algebra and $f:M\to E$ an $A$-linear mapping such that $$ f(x) f(y) = f(y)f(x) \, \text{ for all } x,y \in M.$$
	There exists a unique $A$-algebra homomorphism $g: S(M) \to E $ such that $f=g \circ i$, where $i$ is the canonical injection $M \to S(M)$.
\end{pro}

\begin{pro}\cite{Bourbaki1-3}
	Let $A$ be a ring, $M$ and $N$ two $A$-modules and 
	$$ u:M \to N $$
	an  $A$-linear mapping. There exists a unique $A$-algebra homomorphism 
	$$S(u): S(M) \to S(N)$$ such that the diagram
	$$
	\begin{tikzcd}
	M \arrow{r}{u} \arrow[swap]{d}{i_M} & N \arrow{d}{i_N} \\%
	S(M) \arrow{r}{S(u)}& S(N)
	\end{tikzcd}
	$$ is commutative. Moreover, the homomorphism $S(u)$ is a graded algebra homomorphism.
	\label{symmetric_algebra}
\end{pro}
\begin{pro}\cite{Bourbaki1-3}
	If $u:M \to N$ is a surjective $A$-linear mapping, the homomorphism $S(u) : S(M) \to S(N)$ is surjective and its kernel is the   ideal  of $S(M)$ generated by  the submodule $ker(u) \subset M \subset S(M)$.\label{surjection}
\end{pro}

\begin{defn}[Rees Algebra]
    Let $I$ be an ideal of a ring $R$. Consider the subring of the polynomial ring $R[t]$ generated by the elements $\sum_{i=0}^n c_i t^i$ such that $c_i \in I^i$ (we set $c_0 \in I^0 = R$ and $c_1 \in I^1 = I$), we call it the \textbf{Rees Algebra} of the ideal $I$, and  we denote it by $\mathcal{R}_R(I).$ If the ideal $I$ is finitely generated by elements $a_1 , \dots, a_m$, then  $\mathcal{R}_R(I) = R[a_1 t, \dots, a_m t].$ 
\end{defn}

The following remarks show how symmetric algebra and Rees algebra are represented in practice.
\begin{rem}
	Let $A$ be a ring and $I \subset A$ an ideal. From the injection $i : I\to R$ and Proposition \ref{symmetric_algebra}, we obtain an $A$-algebra homomorphism $S(i) : S(I) \to S(R) = R[t]$, $a_1 \otimes \cdots \otimes a_n \mapsto a_1 \cdots a_n t^n.$ It is clear that $\operatorname{Im}(S(i)) = \mathcal{R}_R(I)$, the Rees algebra of the ideal $I$. Suppose that  $I=\langle a_1, \dots , a_m \rangle $. The canonical epimorphism $R^m \to I$ induces a surjection  $R[t_1,\dots,t_m] = S(R^m) \to \mathcal{S}_R(I)$ by Proposition \ref{surjection}, whose kernel is the ideal $\mathfrak{p} \subset K[t_1,\dots,t_m]$ generated by the linear forms $\sum_{i=1}^m c_i t_i$ such that $\sum_{i=0}^m c_i a_i = 0$ by Proposition \ref{surjection} also. It follows that $$ \mathcal{S}_R(I) \simeq R[t_1,\dots,t_m]/\mathfrak{p}.$$ \label{ree} 
 The ideal $\mathfrak{p}$ is called \textbf{the defining ideal of the symmetric algebra} $\mathcal{S}_R(I).$
\end{rem}
\begin{rem}
	On the other hand, we have the canonical epimorphism  $R[t_1,\dots,t_m] \to \mathcal{R}_R(I), \, t_i \mapsto a_i t$ whose the kernel  is the ideal $\mathfrak{q} \subset K[t_1,\dots,t_m]$ generated by  the  homogeneous polynomials $F \in R[t_1,\dots,t_m]$ such that $F(a_1,\dots,a_m) = 0$. This leads to 
	$$ \mathcal{R}_R(I) \simeq R[t_1,\dots,t_m]/\mathfrak{q}. $$ 
	\label{ree1} 
\end{rem}

\begin{rem}
	Let $I,J \subset R[t_1,\dots,t_m]$ be ideals with generating systems
	$I=\langle a_1,\dots,a_m \rangle $ and $J = \langle b_1, \dots, b_m\rangle.$ If we have the isomorphism
	$$ \mathcal{R}_R(I) \simeq \mathcal{R}_R(J), \text{ by $a_i t  \mapsto b_i t$}$$ 
	then by the previous remark, we get 
	$$ R[t_1,\dots,t_m]/\mathfrak{q}_I \simeq R[t_1,\dots,t_m]/\mathfrak{q}_J.$$
	This isomorphism means that $ \mathfrak{q}_I = \mathfrak{q}_J$. We conclude that the conditions $F(a_1,\dots,a_m) =0 $ and $F(b_1,\dots,b_m) = 0 $  are equivalent for all $F \in R[t_1,\dots,t_m]$ homogeneous polynomial. By Remark~\ref{ree}, if follows that 
	$$ S(I) \simeq S(J).$$
\end{rem}

The following result will be used in the proof of the main criterion of birationality.
\begin{pro} \label{prop:power}
	Let $\varGamma = ker (\mathcal{S}_R (I)  \rightarrow \mathcal{R}_R (I)  )$, then there exists $s \in \mathbb{N}$ such that $I^s \varGamma = 0$.
\end{pro}
\begin{proof}\cite{micali1964algebres}
	Remarks \ref{ree} and \ref{ree1} show that 
	$ \varGamma \simeq \mathfrak{q} / \mathfrak{p}$. We need  to show that there is $s \in \mathbb{N} $ such that $I^s \mathfrak{q} \subset \mathfrak{p}$. Assume that $I = \langle a_1, \dots a_n \rangle  \subset R$  where $0 \neq a_i \in R$. 
	
	Let $f \in \mathfrak{q}$ homogeneous. We will show by induction on the degree of 
	$f $ that there is $s \in \mathbb{N}$ such that $I^s f \in \mathfrak{p}.$\\
	Suppose that the degree of the polynomial $f$ is $1$, it follows that $f \in \mathfrak{p}$ and any $s \geq 0 $ will satisfy requirement. Now, suppose that for each homogeneous polynomial in $\mathfrak{q}$ of degree less or equal to $q-1$, the statement is true. Assume $f$ has degree $q$. Write 
	$$ f = t_1 f_1(t_1,\dots,t_n) + t_2 f_2(t_2,\dots , t_n) + \cdots + t_n f_n(t_n)$$
	with $f_i \in R$ homogeneous of degree $q-1$. Consider the polynomial 
	$$g =  t_1 f_1(a_1,\dots,a_n) + t_2 f_2(a_2,\dots , t_n) + \cdots + t_n f_n(a_n).$$
	It is clear that $g $ is a homogeneous polynomial of degree $1$ in $\mathfrak{q}$, so $g$ is in $\mathfrak{p}$. On the other hand, 	
	$$ {a_n}^{q-1}f - {t_n}^{q-1} g= t_1 g_1(t_1,\dots,t_n) + t_{n-1} g_{n-1}(t_{n-1},t_n),$$
	where $g_i \in R$ homogeneous of degree $q-1$. By assumption, there are $s_i \in \mathbb{N}$ such that $I^{s_i} g_i \in \mathfrak{p}$. Set $$v = \operatorname{max}(\{s_1,\dots,s_{n-1}, n+1 \})$$. It follows that $I^v {a_n}^{q-1} f \in \mathfrak{p}$. 
	By varying $1 \leq i \leq n-1$, there is a $ n < p \in \mathbb{N}$  such that $I^p {a_i}^{q-1} f \in \mathfrak{p}$.
	This  implies that $I^{s} f \in \mathfrak{p}$ with $s= p(q-1)$
\end{proof}

Defining the inverse of a rational map, if it exists, is equivalent to give a representative of that rational map defined by homogeneous polynomials of the same degree. The following proposition shows how to compute all the representatives of a rational map.

\begin{pro}\cite{Simis2004} \label{prop:domain_module}
Let $K$ be an algebraically closed field, $I \subset K[t_0,\dots,t_m]$ a homogeneous prime ideal and
$$\begin{array}{lcll}
\Phi : & X = V(I) \subset \mathbb{P}_K^m & \dashrightarrow & \mathbb{P}_K^n\\
& t = (t_0:{\dots}:t_m) & \longmapsto & (\overline{f_0}(t):{\dots}: \overline{f_n}(t))
\end{array}
$$
a rational map with $\overline{f_i} \in K[t_1,\dots,t_m]/I$ homogeneous of  the
same degree.\newline
    Set $R = K[t_0,\dots,t_m]/I$   and $J = \langle \overline{f_0},\dots, \overline{f_n} \rangle \subset R$ a homogeneous ideal.
    Consider the minimal free graded presentation of $J$
    \begin{equation}
    \label{eq:presentation}
        \begin{tikzcd} 
	\displaystyle\bigoplus\limits_s R(-d_s) \arrow{r}{\varphi} & R^{m+1}(-d) \arrow{r}{(\overline{f_0} \, \cdots \, \overline{f_n)}} & J\arrow{r} & 0.
	\end{tikzcd}
    \end{equation}
 Then, the set of all the representatives of $\Phi$ is in one-to-one correspondence  with the set of homogeneous vectors in the $R$-module $\operatorname{ker}(\varphi^t).$
\end{pro}
\begin{proof}
    Let $(\overline{f'_0}, \dots, \overline{f'_n})$ be another representative of the rational map $\Phi$. We have 
    $$ \text{minors}_2\begin{pmatrix}
\overline{f_0}  & \cdots & \overline{f_n}\\
\overline{f'_0}  & \cdots & \overline{f'_n}
\end{pmatrix} = 0.$$
This means that the two vectors $(\overline{f_0},\dots,\overline{f_n})$ and $(\overline{f'_0},\dots,\overline{f'_n})$ in $\text{Frac}\left(R^{n+1}\right)$ are  proportional by a homogeneous factor of $R$.  This leads to a correspondence between a representative of the rational map $\Phi$ and a homogeneous element $q \in \text{Frac}(R)$ such that $qJ \subset R$. The later defines the fractional ideal $R:_{\text{Frac}(R)} J \simeq \text{Hom}(J,R) $. By dualizing the exact sequence \eqref{eq:presentation} with respect to $R$, we obtain the exact sequence
 \[
 \begin{tikzcd} 
	0 \arrow{r} & \text{Hom}(J,R) \arrow{r}& \text{Hom}(R^{m+1}(-d),R) \arrow{d}{\varphi^t}  \\
 & &\text{Hom}(\displaystyle\bigoplus\limits_s R(-d_s),R).
	\end{tikzcd} \]
 The exactness on $\text{Hom}(R^{m+1}(-d),R)$ yields $\text{Hom}(J,R) \simeq \text{Ker}(\varphi^t)$
\end{proof}

The next proposition states a way to check if two rational map are inverse to each other by comparing Rees Algebras.

\begin{pro}\cite{Simis2004} \label{pro:birational}
Let $X \subset \mathbb{P}^m_K$ and $Y \subset \mathbb{P}^n_K$ two irreducible projective varieties. Let $\Phi : X \dashrightarrow \mathbb{P}^n_K$ and $\Psi : Y \dashrightarrow \mathbb{P}^m_K$ be rational maps with images $Y$ and $X$, respectively. Let $(\overline{f_0},\dots,\overline{f_n} )$ and $(\overline{g_0},\dots, \overline{g_m})$ be  representatives of $\Phi$ and $\Psi$, respectively. Let $R$ and $S$ denote the respective homogeneous coordinate rings of $X$ and $Y.$  Then, the following are equivalent:
\begin{itemize}
    \item [(i)]$\Phi$ and $\Psi$ are inverse to each other.
    \item[(ii)] The identity map of $K[t_0,\dots,t_m,x_0,\dots,x_n]/\left(I(X),I(Y)\right) )$ induces a bigraded isomorphism 
    $$ \mathcal{R}_R \left( \frac{\langle f_0,\dots, f_n, I(X) \rangle}{I(X)}  \right)  \simeq \mathcal{R}_S \left( \frac{\langle g_0,\dots, g_m, I(Y) \rangle}{I(Y)}  \right)$$ of Rees algebras.
\end{itemize}
\end{pro}
\begin{proof}
    See \cite{Simis2004} Proposition 2.1.
\end{proof}

\begin{rem}\cite{Simis2004} \label{rem:jacobian}
    In general, we can use the bigraded $K$-algebra $R[x_0,\dots,x_n] $ with $R = (K[t_0,\dots,t_m]/I(X)$ to express the Rees algebra $\mathcal{R}_R(\overline{f_0},\dots, \overline{f_n})$ as a residue algebra, for a rational map $X \subset \mathbb{P}^m_K \dashrightarrow \mathbb{P}^n_K$ with a representative $(\overline{f_0},\dots, \overline{f_n}).$ Indeed, this is done by using the $R$-algebra homomorphism $$R[x_0,\dots,x_n] \rightarrow  \mathcal{R}_R(\overline{f_0},\dots, \overline{f_n}),$$ mapping $x_i \text{ to } \overline{f_i}.$ Denote by $\mathcal{I}_f$ the kernel of this homomorphism, we call it the \textbf{defining ideal of } $\mathcal{R}_R(\overline{f_0},\dots, \overline{f_n})$. \\ Now, if $(\overline{f'_0},\dots,\overline{f'_n} )$ is another representative of the rational map, then $\mathcal{I}_f = \mathcal{I}_{f'}$. This can be explained by the equivalence of the two conditions $F(\overline{f_0},\dots,\overline{f_n}) = 0 $ and $F(\overline{f'_0},\dots,\overline{f'_n}) = 0$ for all $F \in R[x_0,\dots,x_n].$ The defining ideal associated to a rational map is independent of the representative of the rational map. We just denote it by $\mathcal{I}$ if there is no confusion.\\
    The ideal $\mathcal{I}$ is bihomogeneous. Thus, we can have  minimal generators,  which are constituted by bihomogeneous polynomials with various bidegree $(r,s)$ with $r,s \geq 1$. By Remark \ref{ree}, those of bidegree $(r,1)$ with $r \geq 1$ generates the symmetric algebra $S_R(\overline{f_0},\dots,\overline{f_n}).$\\
    Now, consider those generators of bidegree $(1,s)$ with $s \geq 1$. They generate an ideal of the form $\text{minors}_1( \psi \cdot {(t_0 \cdots t_m)}^t )$ where $\psi$ is a matrix with $m+1$ columns and entries in $K[x_0,\dots, x_n].$ The matrix $\psi$ can be seen as the Jacobian matrix respecting to the variables $t_0,\dots,t_n$ of the form obtained from the bidegree $(1,s)$ elements lifted to $K[t_0,\dots,t_m,x_0,\dots, x_n]$. The rows of $\psi$ are homogeneous vectors of ${(K[x_0,\dots,x_n])}^{n+1}.$ We mostly consider the matrix $\psi$ as matrix with entries over the homogeneous coordinate   ring $S= K[x_0,\dots,x_n]/I(Y)$ of the image $Y \subset \mathbb{P}^{n+1}_K.$
\end{rem}

\begin{defn}[Weak Jacobian Dual] \cite{Simis2004}
 For a rational map $\Phi$, the matrix $\psi$ defined above in Remark \ref{rem:jacobian} considered as a matrix over the coordinate ring of the image is  called a \textbf{weak Jacobian dual} matrix $\psi$ of  the rational map $\Phi$.
\end{defn}

\begin{rem}\cite{Simis2004}
    Clearly, a weak Jacobian dual matrix depends on the representative of the rational map. So, it is not uniquely defined. Nevertheless, for a fix bidegree $(1,s)$, the number of generators of bidegree $(1,s)$ in a minimal bihomogeneous generators of a defining ideal $\mathcal{I}$ are invariant. It is equal to the dimension of the spanned $K$-vector subspace of all the bidegree $(1,s)$ elements of $\mathcal{I}$.  That is, the size of the weak Jacobian dual is invariant in spite of being not unique.
\end{rem}

Now, we state the main criterion of birationality.
\begin{thm} \cite{Simis2004} \label{thm:main_criterion} Let $X \subset \mathbb{P}^m $ be an irreducible projective variety with the vanishing homogeneous prime ideal $I \subset K[t_0,\dots, t_m]$ and $F:X \mapsto \mathbb{P}^n $ a rational map with image $Y \subset \mathbb{P}^n$. Let $R= K[t_0,\dots,t_m]/I(X)$ and $S = K[x_0,\dots,x_n]/I(Y)$ the respective homogeneous coordinate rings of $X$ and $Y$. Then, the following statements are equivalent :
	\begin{itemize}
		\item[\textup{(i)}] The rational map $F$ is birational onto $Y$.
		\item[\textup{(ii)}] $\operatorname{dim}(R) = \operatorname{dim} (S)$, $F$ admits a weak Jacobian dual matrix $\psi$ such that $\operatorname{rank}_{\operatorname{Frac}(S)}(\psi) = m $ and $\operatorname{Im}(\psi^t) = {\left(\operatorname{Im}(\psi)\right)}^* .$
	\end{itemize}
Moreover, if we have the condition $(ii)$, then $\operatorname{ker}_S(\psi)$ is the $S$-module of representatives of the inverse map of $F$.
\end{thm}

To understand the subsequent arguments it is essential to  recall the proof of the theorem from \cite{Simis2004}, and add a few more details.

\begin{proof}
	$(i) \Rightarrow (ii)$ Assume that $F$ is birational onto $Y$. Then, it is clear that $\operatorname{dim}(R) = \operatorname{dim}(S).$ Let $(\overline{f_0},\dots,\overline{f_n})$ and $(\overline{g_0},\dots,\overline{g_m})$  representatives of $F$ and $F^{-1}$, respectively. Then, we have the isomorphism of Rees algebras$$\mathcal{R}_R(\overline{f_0},\dots,\overline{f_n}) \simeq \mathcal{R}_S(\overline{g_0},\dots,\overline{g_m})$$ by Proposition \ref{pro:birational}, and they have the same defining ideal $\mathcal{I}.$
 Let $\Phi$ the graded representation matrix of the ideal $\langle \overline{g_0},\dots, \overline{g_m} \rangle$

  \begin{equation}
    \label{eq:presentation_inverse}
        \begin{tikzcd} 
	\displaystyle\bigoplus\limits_s S(-d_s) \arrow{r}{\Phi} & S^{m+1}(-d) \arrow{r}{(\overline{g_0} \, \cdots \, \overline{g_m})} & \langle \overline{g_0} ,\dots, \overline{g_m}  \rangle\arrow{d} \\
 &&0.
	\end{tikzcd}
    \end{equation}
    It follows that $\text{minor}_1\left(( t_0 \cdots t_m)  \cdot \Phi \right) \subset \mathcal{I}$ as $\text{minor}_1\left(( t_0 \cdots t_m)  \cdot \Phi \right)$ is the defining ideal of the symmetric algebra of $\langle \overline{g_0},\dots, \overline{g_m} \rangle$. On the another hand, any bidegree $(1,s)$, for $s \geq 1$ , generators in a minimal homogeneous generators of $\mathcal{I}$ is a syzygy between the  $\overline{g_i}.$ Then, $\psi := \Phi^t$ is a weak Jacobian dual matrix of the rational map $F$. Clearly,  by the exactness on $S^{m+1}(-d)$ of the exact sequence \eqref{eq:presentation_inverse} $$ \text{rank}_{\text{Frac}(S)}(\psi) = \text{dim} \left(\text{ker}(\overline{g_0} \cdots \overline{g_m} )\right) = m.$$
    Let \begin{equation}
    \label{eq:exact}
        \begin{tikzcd} 
	 S^q \arrow{r}{\rho} & S^{m+1} \arrow{r}{\psi = \Phi^t} & \text{Im}(\psi) \arrow{r} & 0.
	\end{tikzcd}
    \end{equation} be exact. Dualizing the sequence \eqref{eq:exact} with respect to $S$ yields
    \begin{equation}
    \label{eq:exact1}
        \begin{tikzcd} 
	  0 \arrow{r} &{\left(\text{Im}(\psi)\right)}^* \arrow{r}{\Phi} & {\left(S^{m+1}\right)}^* \arrow{r}{\rho^t} & {\left(S^q\right)}^*.
	\end{tikzcd}
    \end{equation} 
    It follows that ${\left(\text{Im}(\psi)\right)}^* = \text{ker}(\rho^t).$ As $${(\overline{g_0}  \cdots  \overline{g_m})}^t \in \text{ker}(\Phi^t) = \text{Im}(\rho)$$ by Proposition \ref{prop:domain_module} and $\rho$ has rank one over $\text{Frac}(S)$, then  the tuple $(\overline{g_0}, \dots, \overline{g_m})$ and $\text{Im}(\rho^t)$ have the same first syzygies. The first sygygies  module  between the $\overline{g_i}$ is $\text{Im}(\Phi) = \text{Im}(\psi^t)$ and that of $\text{Im}(\rho^t)$ is $\text{ker}(\rho^t) = {\left(\text{Im}(\psi)\right)}^* $. Finally $\text{Im}(\psi^t)={\left(\text{Im}(\psi)\right)}^*$ 
    
    $\textup{(i)} \Leftarrow \textup{(ii)}$ Consider the exact sequence 
     \begin{equation}
    \label{eq:exact2}
        \begin{tikzcd} 
	 S^q \arrow{r}{\rho} & S^{m+1} \arrow{r}{\psi } & \text{Im}(\psi) \arrow{r} & 0.
	\end{tikzcd}
    \end{equation}
    By dualizing the exact sequence  \eqref{eq:exact2} and using the hypothesis ${\left(\text{Im}(\psi)\right)}^* = \text{Im}(\psi^t) $, we get the exact sequence.

    \begin{equation}
    \label{eq:exact2bis}
        \begin{tikzcd} 
	  0 \arrow{r} & \text{Im}(\psi^t) \arrow{r}{\Phi} & {\left(S^{m+1}\right)}^* \arrow{r}{\rho^t} & {\left(S^q\right)}^*.
	\end{tikzcd}
    \end{equation} 
Let $g = (\overline{g_0} \cdots \overline{g_m} )$ be any homogeneous $S$-combination of the rows of $\rho^t$
    As previously, $g$ and $\text{Im}(\rho^t)$ have the same first sygygies module. Since the latter has rank one and is equal to $\text{ker}(\rho^t) = \text{Im} (\phi^t)$, we obtain.
    $$ \mathcal{S}_S(\langle \overline{g_0} , \dots, \overline{g_m} \rangle) \simeq S[t_0,\dots, t_m]/ \text{minor}_1\left( (t_0 \cdots t_m ) \cdot \psi^t \right).$$
    Clearly, $\text{minor}_1\left( (t_0 \cdots t_m ) \cdot \psi^t \right)$ is in the defining ideal of the Rees algebra $\mathcal{R}_R(\langle \overline{f_0} , \dots, \overline{f_n} \rangle)$, for a representative $(\overline{f_0} , \dots, \overline{f_n})$ of the rational map $F.$ Then we have a canonical epimorphism $$\pi : \mathcal{S}_S(\langle \overline{g_0} , \dots, \overline{g_m} \rangle) \rightarrow  \mathcal{R}_R(\langle \overline{f_0} , \dots, \overline{f_n} \rangle)$$ induced by the the identity map of $K[t_0,\dots, t_m, x_0, \dots, x_n].$ On the another hand, the Rees algebra$\mathcal{R}_S(\langle \overline{g_0} , \dots, \overline{g_m} \rangle)$ is isomorphic to $\mathcal{S}_S(\langle \overline{g_0} , \dots, \overline{g_m} \rangle)$ modulo its $S$-torsion $\mathcal{T}$, see Remarks \ref{ree} and \ref{ree1}. Also, there is $l \geq 1 $ such that 
    \begin{equation}
    \label{eq:torsion}
    {\langle \overline{g_0} , \dots, \overline{g_m} \rangle}^l \mathcal{T} = 0
    \end{equation}
    by Proposition \ref{prop:power}. By identifying, $$ \mathcal{R}_R(\langle \overline{f_0} , \dots, \overline{f_n} \rangle) \simeq R[\overline{f_0}t, \dots, \overline{f_n}t] \subset R[t],$$ the morphism $\pi$ maps $\overline{x_i} \in S $ to $\overline{f_i} t$.  Applying $\pi$ to the torsion equation \eqref{eq:torsion} yields 
    \begin{align*}
    {\langle \overline{g_0}(\overline{f_0}t, \dots, \overline{f_n}t) , \dots, \overline{g_m}(\overline{f_0}t, \dots, \overline{f_n}t) \rangle}^l \pi(\mathcal{T}) &= 0  \\
      {\langle \overline{g_0}(\overline{f_0}, \dots, \overline{f_n}) , \dots, \overline{g_m}(\overline{f_0}, \dots, \overline{f_n}) \rangle}^l \pi(\mathcal{T}) &= 0 .
    \end{align*}
    However, $\langle \overline{g_0}(\overline{f_0}, \dots, \overline{f_n}) , \dots, \overline{g_m}(\overline{f_0}, \dots, \overline{f_n}) \rangle \neq 0$ because $$K[x_0,\dots,x_n]/I(Y) \simeq R[\overline{f_0},\dots, \overline{f_n}],$$ and at least one $\overline{g_i} \neq 0.$ Necessarily, we have $\phi(\mathcal{T}) = 0 $ and  an epimorphism
    $$ \mathcal{R}_S(\langle \overline{g_0} , \dots, \overline{g_m} \rangle) \rightarrow \mathcal{R}_R(\langle \overline{f_0} , \dots, \overline{f_n} \rangle).$$
    The two Rees algebras have the same dimension. Since 
    \begin{align*}
        \text{dim}(\mathcal{R}_S(\langle \overline{g_0} , \dots, \overline{g_m} \rangle)) & = \text{dim}(S) +1 \\
        & = \text{dim}(R) +1 \\
        & = \text{dim}(\mathcal{R}_R(\langle \overline{f_0} , \dots, \overline{f_n} \rangle)).
    \end{align*}
    So they are domains and have the same dimension. Then, the map is an isomorphism and it is induced by the identity map of $K[t_0,\dots,t_m,x_0,\dots, x_n].$ By Proposition \ref{pro:birational}, the rational map $F$ is birational onto its image $Y.$

    The last statement follows immediately from Proposition \ref{prop:domain_module}.

 \end{proof}
\begin{rem} \label{rem:criterion}
    Using the notation of Theorem \ref{thm:main_criterion}, it is not clear how we can check the condition $\text{Im}(\psi^t) = {\left(\text{Im}(\psi)\right)}^*$  in condition (ii) of Theorem \ref{thm:main_criterion}. However, the proof of the theorem shows that the condition is equivalent to $\text{Im}(\psi^t) = \text{ker}(\rho^t)$.

    On the other hand, we did not mention how we can compute a weak Jacobian dual of a rational map $F$. This is equivalent to find a defining ideal $\mathcal{I}$ of the Rees algebra $\mathcal{R}_R(\langle \overline{f_0} , \dots, \overline{f_n} \rangle)$, with $(\overline{f_0} , \dots, \overline{f_n} )$ a representation of the rational map. Remark \ref{ree1} said that the the ideal $\mathcal{I}$ is defined by
    
    $$ \mathcal{I} = \left\{ f \in R[x_0,\dots, x_n] \, | \, f(\overline{f_0},\dots, \overline{f_n)} = 0  \right\}.$$
    If we write  the Rees algebra $\mathcal{R}_R(\langle \overline{f_0} , \dots, \overline{f_n} \rangle)$ in terms of residue algebra of the the $K$-algebra $K[t_0,\dots, t_m,x_0,\dots, x_n]$, then  the condition $ f(\overline{f_0},\dots, \overline{f_n)} = 0 $ is replaced by $$f(t_0,\dots,t_m,f_0,\dots, f_n) \in I(X)$$ for $f \in K[t_0,\dots, t_m,x_0,\dots, x_n]$. The latter condition is equivalent to saying that $f$ is in the ideal of graph of the rational map. To conclude, a weak Jacobian dual can be computed via the ideal of the graph  of a rational map in Proposition \ref{prop:graph}. 
\end{rem}

The algorithm for determining invertibility and, if this is the case, computing the inverse of a rational map is presented in Algorithm \ref{alg:inverse}. We illustrate the algorithm and Theorem \ref{thm:main_criterion} in Example \ref{ex inverse}.

\begin{algorithm}[ht]
		\caption{Inverse of rational map}
  \label{alg:inverse}
		\begin{algorithmic}[1]
			\REQUIRE A rational map
			$$\Phi : X = V(I) \subset
			\mathbb{P}_K^m \dashrightarrow \mathbb{P}_K^n , \, t \mapsto
			(\overline{f_0}(t):{\dots}: \overline{f_n}(t))$$ with $X$  irreducible and $\overline{f}_i \in 
			K[t_0,\dots,t_m]/I(X) $. Denote by $R$ the homogeneous coordinate ring of $X$.
			\ENSURE If it is birational, the $S$-module of representative of the inverse of the rational map $\Phi$, where $S$ is the homogeneous coordinate of the image of $\Phi$. If not the case, return \textbf{false}.
			\STATE Set the ideal $$ J = \left\langle \text{minors}_2\begin{pmatrix}
				x_0  & \cdots & x_n\\
				f_0  & \cdots & f_n
			\end{pmatrix}\right\rangle \subset K[t_0,\dots,t_m,x_0,\dots,x_n]$$
			\STATE Compute  a Gr\"obner basis respecting an elimination ordering for $t_0,\dots,t_m$ of the ideal of the graph $$G=  (I+J): \left\langle f_0,\dots,f_n \right\rangle^\infty \subset  K[t_0,\dots,t_m,x_0,\dots,x_n].$$
             \STATE Compute the ideal of the image $$G \cap K[x_0,\dots, x_n] \subset K[x_0,\dots,x_n].$$ and set $S$ to be the homogeneous coordinate ring of the image.
			\STATE Define a weak Jacobian dual matrix $\psi \text{ over } S $ such that the entries of $ \psi \cdot ({t_0 \cdots t_m)}^t$ are the bidegree $(1,s)$ generators of the ideal $G$, for $s\geq 1$, .
            \STATE Compute a matrix $\rho$  over $S$ such that  $\text{im}(\rho) = \text{ker} (\psi) $
			\IF{ \text{dim}$(R)$ $=$ dim$(S)$  }
            \IF{ $\text{rank}_{\text{Frac}(S) } (\psi)$}
            \IF{$\text{im}(\psi^t) = \text{ker}(\rho^t)$ } 
            \STATE Compute a matrix $\beta$ such that $ \text{im}(\beta) =\text{ker}(\psi) \subset S^{m+1}$ over $S$.
            \RETURN $\text{im}(\beta).$
            \ENDIF
            \ENDIF
            \ENDIF
            \RETURN \textbf{false}.
		\end{algorithmic}
	\end{algorithm}

\begin{exa}\label{ex inverse}
    Reconsider  the rational map from Example \ref{exa:stereographic}
 $$\begin{array}{lcll}
  \Phi : &  \mathbb{P}_\mathbb{C}^1 & \dashrightarrow & \mathbb{P}_\mathbb{C}^2\\
  & (t_0:t_1) & \longmapsto & (\overline{{t^2_0}+ t_1^2} : \overline{t^2_0 - t^2_1 }: \overline{ 2 t_0 t_1}).
 \end{array}
$$
Set $J = \left\langle \text{minors}_2\begin{pmatrix}
                                 x_0  &  x_1 & x_2\\
                                 t^2_0+t^2_1   & t^2_0 - t^2_1 & 2 t_0 t_1
                                \end{pmatrix}
 \right\rangle  \subset K[t_0,t_1,x_0,x_1,x_2].  $ The ideal $J$ is generated by the polynomials \begin{align*}-t_0^2 x_2+2t_0t_1x_1+t_1^2x_2, \\
-t_0^2x_2+2t_0t_1x_0-t_1^2x_2,\\
-t_0^2x_0+t_0^2x_1+t_1^2x_0+t_1^2x_1
 \end{align*}
 The ideal of the graph $G \in K[t_0,\dots,t_m,x_0,\dots,x_n]$ is generated by
 \begin{align*}  x_0^2-x_1^2-x_2^2, \\
t_0x_2-t_1x_0-t_1x_1,\\
t_0x_0-t_0x_1-t_1x_2
 \end{align*}
 and the ${\text{im}(\Phi)} = V(\langle x_0^2-x_1^2-x_2^2 \rangle) \subset \mathbb{P}^2_\mathbb{C}$
 Then, a weak Jacobian dual $\psi$ can be compute by extracting the bidegree $(1,s)$ generators from $G$. That are the two last generators of $G$ and 
 $$\psi = \begin{pmatrix}
         \overline{x_2} & - \overline{x_0}-\overline{x_1} \\
            \overline{x_0}-\overline{x_1}   &  -\overline{x_2}
                                \end{pmatrix} \in \mathcal{M}_{2,2}(\mathbb{C} [x_0,x_1,x_2]/ \langle x_0^2-x_1^2 - x_2^2 \rangle  .)$$

                                The rank of $\psi$ over the fraction field of the homogeneous coordinate ring of the image $\mathbb{C} [x_0,x_1,x_2]/ \langle x_0^2-x_1^2 - x_2^2 \rangle$ is one because the second row of $\psi$ is obtained by multiplying the first one by $\overline{x_2}$. Clearly, $\text{dim}(\mathbb{P}^1_\mathbb{C}) =\text{dim}(\text{im}(\Phi)) = 1. $\\
Using the command \texttt{syz} we can compute 
$$ \rho = \begin{pmatrix}
    \overline{x_0} + \overline{x}_1 & \overline{x}_2 \\
    \overline{x}_2 & \overline{x_0}-\overline{x}_1
\end{pmatrix}$$ defined by the exact sequence 
$$
        \begin{tikzcd} 
	 S^2 \arrow{r}{\rho} & S^{2} \arrow{r}{\psi } & \text{Im}(\psi) \arrow{r} & 0,
	\end{tikzcd}
    $$
    where $ S $ is the homogeneous coordinate ring of the image. Then $\text{ker}(\rho^t)$ is generated by the submodule image of the matrix $$
    \begin{pmatrix}
        \overline{x_0} -\overline{x_1} & - \overline{x_2} \\
        -\overline{x_2} & \overline{x_0} + \overline{x_1}. 
    \end{pmatrix}
    $$ in $S^2$. We can see that the condition $\text{Im}(\psi^t) = \text{ker}(\rho^t) $ is satisfied. Thus, the rational  $\Phi$ is birational by  Theorem \ref{thm:main_criterion} and the $S$-module of representative of the inverse is given by $\text{ker}(\psi)$. Using for example the procedure \texttt{syz} from in \textsc{Singular}, we determine $\text{ker}(\psi)$, which is the submodule image of the matrix $$
    \begin{pmatrix}
        \overline{x_0} + \overline{x_1} & \overline{x_2} \\
        \overline{x_2} & \overline{x_0} - \overline{x_1}
    \end{pmatrix}$$ in $S^2.$

Then, the inverse of $\Phi$ is defined by
     $$\begin{array}{lcll}
  \Phi^{-1} : &  V(\langle x_0^2-x_1^2-x_2^2 \rangle) \subset \mathbb{P}_\mathbb{C}^2 & \dashrightarrow & \mathbb{P}_\mathbb{C}^1\\
  & (x_0:x_1:x_2) & \longmapsto & (\overline{x_0} + \overline{x_1}: \overline{x_2}).
 \end{array}
$$
\end{exa}

\subsection{Algorithm for computing the domain}
We present an algorithm to compute the module of all the  representative of a rational map, which then allows us to find the domain of the rational map. Proposition \ref{prop:domain_module} already gives an algorithm to compute it. However, we show in this section that it is possible to compute it also from the ideal of the graph in Proposition \ref{prop:graph}. See the corresponding algorithm in Algorithm \ref{algo:domain}.

	\begin{algorithm}[ht]
		\caption{Representatives and domain of rational map}
            \label{algo:domain}
		\begin{algorithmic}[1]
			\REQUIRE A rational map
			$$\Phi : X = V(I) \subset
			\mathbb{P}_K^m \dashrightarrow \mathbb{P}_K^n , \, t \mapsto
			(\overline{f_0}(t):{\dots}: \overline{f_n}(t))$$ with $X$  irreducible and $\overline{f}_i \in 
			K[t_0,\dots,t_m]/I(X) $. Denote by $R$ the homogeneous coordinate ring of $X$
			\ENSURE  The $R$-module of representatives of the rational map $\Phi$, as well as the domain..
			\STATE Set the ideal $$ J = \left\langle \text{minors}_2\begin{pmatrix}
				x_0  & \cdots & x_n\\
				f_0  & \cdots & f_n
			\end{pmatrix}\right\rangle \subset K[t_0,\dots,t_m,x_0,\dots,x_n]$$
			\STATE Compute  a Gr\"obner basis respecting an elimination ordering for $t_0,\dots,t_m$ of the ideal of the graph $$G=  (I+J): \left\langle f_0,\dots,f_n \right\rangle^\infty \subset  K[t_0,\dots,t_m,x_0,\dots,x_n].$$
			\STATE Define a  matrix $\alpha \text{ over } R $ such that the entries of $ \alpha \cdot ({x_0 \cdots x_n)}^t$ are the bidegree $(r,1)$ generators of the ideal $G$, for $r\geq 1$.
            \STATE Compute a matrix $\beta$  over $R$ such that  $\text{im}(\beta) = \text{ker} (\alpha) .$
            \STATE Write $\beta = [ \beta_1 \ \beta_2 \ \cdots \ \beta_n ]$.
            \RETURN $\text{im}(\beta)$, and 
\[
\text{Domain}(\beta) = X\setminus V\left(\cap_{i=1}^{n} \langle \beta_i \rangle\right)
\] where $\langle \beta_i \rangle$ is the ideal generated by the entries of the column $\beta_i$ of $\beta$.
		\end{algorithmic}
	\end{algorithm}

\begin{proof}(Algorithm \ref{algo:domain}) 
Let $ h_1,\dots,h_k \in K[t_0,\dots,t_m]$  the elements of $G$ which lie in the subring $K[t_0,\dots,t_m]$ and $l_0,\dots,l_r \in K[t_0,\dots,t_m,x_0,\dots,x_n]$ the elements  of bidegree $(r,1)$ in the Gr\"obner basis of $G$. Denote by $E= \text{ker}(\alpha)$ the submodule of the free module $R^{n+1}$. \\
Let $(\overline{g_0},\dots, \overline{g_n}) \in E$. This means that $$l_i(t_0,\dots,t_m,g_0,\dots,g_n)=l_{i_{|x_q=g_q}} \in I(X)$$ for all $0\leq i\leq v$. On the other hand, as $x_i f_j -x_j f_i \in J \subset G $ and the set  of $G$ is a Gr\"{o}bner basis, we have 
\begin{equation}\label{eq:domain}
x_i f_j -x_j f_i= \sum_{u=1}^{k} \alpha_{u} h_ u + \sum_{u=0}^v \beta_u l_u \text{ for all } 0\leq i,j\leq n \end{equation}
This is a standard representation of $x_i f_j -x_j f_i$ respecting the normal form $NF(-,G).$ By doing the substitution $x_q=g_q$ for all $0 \leq q \leq n$ in  equation \eqref{eq:domain}, we get 
\begin{align*}
g_i f_j -g_j f_i &= \sum_{u=1}^{k} \alpha_{u_{|x_q=g_q}} h_ u + \sum_{u=0}^v \beta_{u_{|x_q=g_q}} l_{u_{|x_q=g_q}} \in I(X) 
\end{align*} $\text{ for all } 0\leq i,j \leq n.$
because $h_u, l_{u_{|x_q=g_q}} \in I(X)$. This means that
$$  g_i f_j - g_j f_i \equiv 0 \pmod{I(X)} \text{ for all } 0\leq i,j \leq n.$$
It follows that the tuple $(\overline{g_0},\dots,\overline{g_n})$ is a representation of the rational map $\Phi$. Then $ \cup_{0\leq i \leq s } U_i \subset D(\Phi)$.

For the other inclusion, let $(\overline{g_0},\dots,\overline{g_n})$ be  a representation of the rational map $\Phi$. This means that 
\begin{equation}
\label{eq:2}
	g_i f_j - g_j f_i \in I(X) \text{ for all } 0\leq i,j \leq n, \, i \neq j.
\end{equation}
We have to show that $(\overline{g_0}:\dots:\overline{g_n}) \in E.$ That is $l_{u_{|x_q=g_q}} \in I(X)$ for all $0\leq u \leq r.$\\
Fix $0\leq u \leq r.$ As $l_u \in G$, there exist $N \in \mathbb{N}$ and $0\leq p \leq n$ such that $f_p \notin I(X)$ and $l_u f_p^N \in I+J.$\\
If $I=\langle m_1,\dots, m_v \rangle$, there are $\alpha_i , \beta_{ij} \in K[t_0,\dots,t_m,x_0,\dots,x_n]$ such that
$$
	l_u f_p^N = \sum_{i=1}^v \alpha_i m_i + \sum_{\substack{0\leq i,j \leq n \\ i\neq j}} \beta_{ij} (x_if_j-x_j f_i.)
$$
By substitution, $x_q = g_q$ for all $0 \leq q \leq n$, we get 
$$
l_{u_{|x_q=g_q}} f_p^N = \sum_{i=1}^v \alpha_{i_{|x_q=g_q}} m_i + \sum_{\substack{0\leq i,j \leq n \\ i\neq j}} \beta_{ij_{|x_q=g_q}} (g_if_j-g_j f_i) \in I(X)
$$
\text{ by }  equation \eqref{eq:2}.
As $I(X)$ is a prime ideal and $f_p^N \notin I(X)$, we have $l_{u_{|x_q=g_q}} \in I(X)$.\\
Suppose that $E = \langle (\overline{g}_{00},\dots,\overline{g}_{0n}), \dots,(\overline{g}_{s0},\dots,\overline{g}_{sn}) \rangle$. Thus
\begin{align*}
	\cap_{i=0}^{s} V(I+\langle g_{i0},\dots , g_{in}\rangle) \subset V(I+ \langle g_0, \dots, g_n\rangle)
\end{align*}
as $(\overline{g}_0,\dots,\overline{g}_n) \in E$
Then,
\begin{align*}
	X \setminus V(I+ \langle g_0, \dots, g_n\rangle) \subset X \setminus 	\cap_{i=0}^{s} V(I+\langle g_{i0},\dots , g_{in}\rangle) = \cup_{i=0}^s U_i
\end{align*} for all representatives $(\overline{g}_0:\dots:\overline{g}_n)$ of the rational map $\Phi$.\\
It follows that $D(\Phi) \subset \cup_{i=0}^s U_i.$
\end{proof}

\begin{rem}
The ideal $G$ is the defining ideal the Rees algebra $\mathcal{R}_R(\langle \overline{f_0}, \dots, \overline{f_n}  \rangle)$ by Remark~\ref{rem:criterion}. Then, using Remark \ref{ree} and \ref{ree1}, the elements of bidegree $(r,1)$ of $G$ generate the defining ideal of symmetric algebra $\mathcal{S}_R(\langle \overline{f_0} , \dots, \overline{f_n} \rangle)$. It follows that those elements are relations between the $\overline{f_i}$ and Proposition \ref{prop:domain_module} proves  Algorithm \ref{algo:domain}.
\end{rem}

\begin{exa}
    Consider the rational map $$\begin{array}{lcll}
  \Phi : &  V(\langle t_0^2-t_1^2-t_2^2 \rangle) \subset \mathbb{P}_\mathbb{C}^2 & \dashrightarrow & \mathbb{P}_\mathbb{C}^1\\
  & (t_0:t_1:t_2) & \longmapsto & (\overline{t_0} + \overline{t_1}: \overline{t_2}).
 \end{array}
$$
Let apply Algorithm $\ref{algo:domain}$ to find the domain of the rational map $\Phi$.\\
Using \textsc{Singular}, the ideal of the graph $G$ is generated by  the polynomials
$$
\begin{array}{l}
     2t_1 x_0 x_1-t_2 x_0^2+t_2 x_1^2, \\
t_0 x_1+t_1 x_1-t_2 x_0, \\
t_0 x_0-t_1 x_0-t_2 x_1, \\
t_0^2-t_1^2-t_2^2.
\end{array}$$

We see that only the second and third generators of $G$ have bidegree $(r,1)$ for $r \geq 1$.  It follows that 
$$ \alpha = \begin{pmatrix}
   -\overline{t_2} &  \overline{t_0}+\overline{t_1} \\
\overline{t_0}-\overline{t_1} & -\overline{t_2}   
\end{pmatrix} \in \mathcal{M}_{22}(R),$$
where $R = K[t_0,t_1,t_2]/\langle t_0^2 -t_1^2 -t_2^2 \rangle. $ \\ Using the procedure \texttt{syz} from \textsc{Singular}, we get
$$
\beta = \begin{pmatrix}
    \overline{t_0}+\overline{t_1} & \overline{t_2} \\  
\overline{t_2} &   \overline{t_0}-\overline{t_1}
\end{pmatrix} \in \mathcal{M}_{22}(R),
$$
such that $\text{im}(\beta) = \ker({\alpha}).$\\
This means that the rational map $\Phi$ has two main representatives which, are $(\overline{t_0}+ \overline{t_1}, \overline{t_2} ) \in R^2$ and $(\overline{t_2}, \overline{t_0} - \overline{t_1}) \in R^2.$ The others representatives are just $R$-linear combinations of these two. \\
Only the point $(1 : -1 : 0) \in V\left( \langle t_0^2-t_1^2- t_2^2 \rangle \right)$  does not have an image if we use the representative $(\overline{t_0}+ \overline{t_1}, \overline{t_2} )$. However, it does in the second representative  $(\overline{t_2}, \overline{t_0} - \overline{t_1})$. To conclude, the rational map $\Phi$ is morphism. That is $D(\Phi) = V\left( \langle t_0^2-t_1^2- t_2^2 \rangle \right).$
\end{exa}

\section{Petri nets and GPI-Space}\label{sec4}

Development of parallel algorithms, their efficient implementation, and exploring their use is a fundamental challenge in computer algebra. Separating the coordination and the computation in a given algorithmic problem has significant benefits when modeling parallel computations, see \cite{gelernter1992coordination}. An asynchronous ensemble is built from a description of each of its computational activities and the connections between them which model their logical dependence and communication. In the language of Petri nets, this is done by describing transitions implementing the computational activities, and linking them via places, which can hold pieces of data (called tokens), into a bipartite graph by adding directed edges.

Due to the unpredictability of the time and memory consumption of key algorithmic tools, like Buchberger's algorithm for computing Gröbner bases, parallelization is a fundamental challenge in computational commutative algebra and algebraic geometry. As demonstrated in multiple application areas, the idea of separating computation and coordination has proven to be a game-changer in this sector of computer algebra.

The actual use of the idea of computation and coordination depends on an efficient and reliable implementation. This task goes beyond what computer algebra systems can deliver on their own. We thus make use of a proven workflow management system which can directly execute a variant of a Petri net.

\subsection{\textsc{Singular}/\textsc{GPI-Space} Framework}

The tasked-based workflow management system \textsc{GPI-Space} for parallel applications, which is developed by the Fraunhofer Institute for Industrial Mathematics (ITWM) as open-source software, realizes the coordination layer in our approach and is combined with the computer algebra system \textsc{Singular} for highly efficient polynomial computations as the computation layer. Further computational tools can be integrated as well, but are not of relevance for our problems in birational geometry. For algorithms with coordination modelled in terms of a Petri net, \textsc{GPI-Space} offers automatic parallelization and balancing of work-loads. For practical usability, \textsc{GPI-Space} allows transitions to take time and tokens to be data structures (corresponding to so-called timed and colored Petri nets). The system can be applied on a wide range of computational infrastructure, from a personal computer, a compute server, to an HPC cluster. In our framework, we use \textsc{Singular} not only as backend, but also as user frontend. We note that our framework allows for the use of other frontends, as well.

The \textsc{Singular}/\textsc{GPI-Space} framework has been effectively applied in various domains, such as computing the smoothness of algebraic varieties, resolution of singularities, computation of GIT-fans in geometric invariant theory, determination of tropicalizations, and in high energy physics for multivariate partial fraction decomposition. For some of these applications, see \cite{boehm2018massively, boehmissac, bendle2020parallel, boehmPFD}. Generic massively parallel wait-all and wait-first workflows have been implemented in the framework. In this paper, we develop a generic massively parallel modular framework, relying on the approach discussed in Section \ref{rational reconstruction}.

In order to use the \textsc{Singular}/\textsc{GPI-Space} framework, we have to model our workflow as a Petri net. We start out with a short outline of the concept of Petri nets.

\subsection{Petri nets} 
The theory of Petri nets was introduced by Carl Adam Petri in his Ph.D thesis in 1962 to model distributed  systems through a mathematical representation. This representation describes the structure as well the dynamic behavior of the system. We will use an approach using graph theory to formally define Petri net. See, for example, \cite{peterson1981} for more details.
\begin{defn}
	A Petri net is a bipartite directed graph $G=(V,A)$, where $V$ is a finite set of vertices and $A \subset V \times V$ is a finite set of directed edges. The vertices are divided into two types, called places and transitions, denoted by $P$ and $T$, respectively. Then, $V$ is the disjoint unions $V=P \cup T$ we require that $A \subset (P \times T \cup T \times P)$.\\
	In  our visualization of a Petri net, places are depicted as circles and transitions are represented by squares.
\end{defn}
\begin{exa}\label{fig petri} 
The following Petri net, consists out of two places, $P_1$ and $P_2$, one transition $T$, connected by three edges.
\
	\begin{center}
		\begin{tikzpicture}
		\node[place,label=below:$P_1$] (i) at (0,0) {};
		\node[transition,minimum width=8mm, minimum height = 8mm,label=below:$T$]  (t)  at (2,0) {};
		\node[place,label=below:$P_2$] (o) at (4.4,0) {};
		\draw[thick] (i) edge[post] (t);
		\draw[thick] (t) edge[post,bend left=20] (o);
		\draw[thick] (o) edge[post,bend left=20] (t);
		\end{tikzpicture}
	\end{center}
\end{exa}
\begin{defn}
	A marking $\mu$ of a Petri net $G$ is function $\mu : P  \rightarrow \mathbb{N}$. For $p \in P$, we say that the place $p$ holds  $\mu(p)$ tokens under $\mu$.
	Assigning a marking to a Petri net defines a marked Petri net $G=(V,A,\mu)$.
	Tokens are visually represented by dots in the circles which represent the places of a Petri net.
	\begin{exa}\label{exa:1} Assigning to the Petri net in Example \ref{fig petri} the marking $\mu$ defined by $\mu(P_1) = 4$ and  $\mu(P_2) = 1. $, we obtain
		\begin{center}
			\begin{tikzpicture}
			\node[place,label=below:$P_1$,tokens=4] (i) at (0,0) {};
			\node[transition,minimum width=8mm, minimum height = 8mm,label=below:$T$]  (t)  at (2,0) {};
			\node[place,label=below:$P_2$,tokens=1] (o) at (4.4,0) {};
			\draw[thick] (i) edge[post] (t);
			\draw[thick] (t) edge[post,bend left=20] (o);
			\draw[thick] (o) edge[post,bend left=20] (t);
			\end{tikzpicture}
		\end{center}
	\end{exa}
\end{defn}
The same  Petri can have many \emph{states} depending on which marking is placed on it. So it is natural to think of a way to relate marking  between them. This leads to the following definitions.

\begin{defn}
	Let $G=(V,A,\mu)$ a marked Petri net. A transition $t\in T$ is enabled by the marking $\mu$ if for $p \in P$ such that $(p,t) \in A$, we have 
	$$\mu(p) \geq 1.$$\end{defn}

\begin{rem}
	Let $t \in T$, a place $p \in P$ such that $(p,t) \in A$ is called an input place of the transition $t$ and if $(t,p) \in A$, the place $p$ is said to be  an output place of the transition $t$. So, a transition is enabled if each of its input places has at least one tokens.
\end{rem}
\begin{defn}
	A transition fires by removing one token from each of its  input places and depositing one token into each of its output places.
\end{defn}
\begin{exa}
	In Example	\ref{exa:1}, firing the transition $T$ results in the Petri net
	\begin{center}
		\begin{tikzpicture}
		\node[place,label=below:$P_1$,tokens=3] (i) at (0,0) {};
		\node[transition,minimum width=8mm, minimum height = 8mm,label=below:$T$]  (t)  at (2,0) {};
		\node[place,label=below:$P_2$,tokens=1] (o) at (4.4,0) {};
		\draw[thick] (i) edge[post] (t);
		\draw[thick] (t) edge[post,bend left=20] (o);
		\draw[thick] (o) edge[post,bend left=20] (t);
		\end{tikzpicture}
	\end{center}
	with the new marking $\mu'$ defined by $\mu'(P_1)=3$ and $\mu'(P_2) = 1$. It is easy to verify that from the original marking, $T$ can fire four times, but not more.
\end{exa}
\begin{pro}\label{pro:firing}
	Let  $G=\left(V,A,\mu \right)$ be a marked Petri net. Firing a transition $t$ results in a new marking $\mu'$ defined  by
	$$\mu'(p) = \mu(p)-\left|\left\{(p,t) \right\} \cap A \right|+ \left| \left\{ (t,p)\right\} \cap A \right| \, \text{ for all } p \in P.$$
\end{pro}
\begin{defn}
	Given a marked Petri net $G=(V,A,\mu)$, the firing of transition $t \in T$ gives a new marking $\mu'$ as defined in Proposition \ref{pro:firing}. We say that $\mu'$ \emph{is immediately reachable from $\mu$} and we denote it visually by $\mu \stackrel{t}{\longrightarrow} \mu'.	$
	\\In general, $\mu'$ is \emph{reachable from $\mu$} if there is a sequence of transitions $(t_1,t_2,\dots,t_n)$ and a sequence of marking $(\mu_1,\mu_2,\dots,\mu_n)$  such that $$\mu \stackrel{t_1}{\longrightarrow} \mu_1 \stackrel{t_2}{\longrightarrow} \cdots \stackrel{t_n}{\longrightarrow} \mu'.$$  
\end{defn}
\begin{rem}
	Transitions can continue to fire as long as there is an enabled transition. If there is no enabled transition, the execution of the Petri net stops.
\end{rem}
Petri nets were designed for modeling parallelism, modeling distributed systems and modeling flow of information within a system.

\begin{exa}[Task parallelism, data parallelism]

\begin{figure}
	\begin{center}
		\begin{tikzpicture} 
		\node[place,label=below:$P_1$,tokens=3] (P1) at (0,0) {};
		\node[transition,minimum width=8mm, minimum height = 8mm,label=below:$T_1$]  (T1)  at (2,0) {};
		\node[place,label=below:$P_2$,tokens=2] (P2) at (4.4,1) {};
		\node[place,label=below:$P_3$, tokens=2] (P3) at (4.4,-1) {};
		\draw[thick] (P1) edge[post] (T1);
		\draw[thick] (T1) edge[post,bend left=20] (P2);
		\draw[thick] (T1) edge[post,bend right=20] (P3);
		\node[transition,label=below:$T_2$, minimum width=8mm, minimum height = 8mm] (T2) at (6.5,1) {};
		\node[transition,label=below:$T_3$,minimum width=8mm, minimum height=8mm] (T3) at (6.5,-1) {};
		\node[place, label=below:$P_4$] (P4) at (8.5,1) {};
		\node[place, label=below:$P_5$] (P5) at (8.5,-1) {};
		\draw[thick] (P2) edge[post] (T2);
		\draw[thick] (P3) edge[post] (T3);
		\draw[thick] (T2) edge[post] (P4);
		\draw[thick] (T3) edge[post] (P5);
		\node[transition,label=below:$T_4$, minimum width=8mm,minimum height=8mm] (T4) at (11,0) {};
		\node[place, label=below:$P_6$] (P6) at (13,0) {};
		\draw[thick] (P4) edge[post,bend left=20] (T4);
		\draw[thick] (P5) edge[post,bend right=20] (T4);
		\draw[thick] (T4) edge[post] (P6);
		\end{tikzpicture}
	\end{center}
 \caption{Task and data parallelism in a Petri net.}
 \label{fig parallel}
 \end{figure}
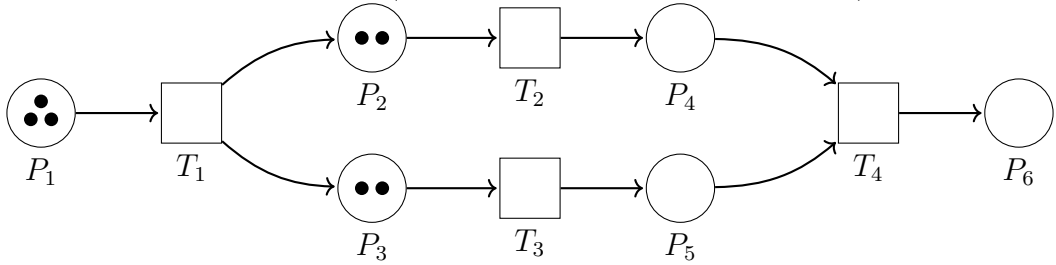
	The Petri net model in Figure \ref{fig parallel} shows that the tasks $T_2$ and $T_3$ can be scheduled in parallel. As transitions they can fire at the same time, realizing task parallelism.
	Moreover, three instances of  $T_1$ can fire in parallel, realizing data parallelism. Similarly two instances of $T_2$ and $T_3$.
\end{exa}

\section[Petri Net for Modular Methods]{Modelling Modular Methods as a Petri Net}\label{sec:5}
In this section, we develop  a generic coordination model for modular algorithms in terms  of Petri nets, as well as its implementation using the \textsc{Singular}/\textsc{GPI-Space} framework, see \cite{modular}. The net is structured into subnets, that are subgraphs and can be thought as inlined into one large Petri net. In the first section, we present the main net for modular algorithms. Subnets are depicted by clouds. In the subsequent three sections, the subnets are described in detail. In Section~\ref{proofalg} we argue why the algorithm is correct and terminates. 

\subsection{Main net}
  We start out by describing the main Petri net modelling the modular algorithm (see Figure \ref{fig:modular}). The Petri net is initialized with the following tokens: a token on the place \texttt{Start}; a token of which all the fields are empty on the place \texttt{Accumulator}; a token containing a boolean field, set to false on the place \texttt{Bal3}; a token containing an integer, set to $0$, on each of the places \texttt{Ct} and \texttt{C} in the main net, as well as the subnets \texttt{Lift} (see Figure \ref{fig:lift}) and \texttt{Lift/Add} (see Figure \ref{fig:liftadd}); a structureless token on the place \texttt{U} in the subnet \texttt{Lift}, the subnet \texttt{Lift/Add}, as well as the subnets \texttt{ParallelFarey} (see Figure \ref{fig:parallelfarey}) and \texttt{ParCompatible} (see Figure \ref{fig:parcompatible}). 
We use \textsc{Singular} as the client to provide the token on the starting place. This token contains the following fields: (i)~an element $P\in R_0\langle m_1,\ldots,m_s\rangle$, provided by the user in \textsc{Singular}; (ii) a finite set of primes $\mathbb P$ (of which the size can be specified), chosen by value (the primes are chosen by decreasing order, starting at the biggest prime \textsc{Singular} can provide); and (iii) the smallest prime $p'$ in the set $\mathbb P$.\\

 \noindent \underline{Transition \texttt{Init}:}\\
 The transition \texttt{Init} parses the information contained in the input token and places a token containing the element $P$ on the place $\texttt{Input}$; for each of the primes $p\in \mathbb P$, places a token containing $p$ on the place $\texttt{Primes}$; places a token containing the prime $p'$ on the place $\texttt{Last}$; and places a set of structureless tokens on the place $\texttt{Bal2}$, using the \texttt{Connect-Out-Many} feature of GPI-space. The number of tokens that $\texttt{Init}$ places on $\texttt{Bal2}$ is optional. The default value for \texttt{Bal2} is $83\%$ of  the total number of cores, assigned to the framework. Note that the place \texttt{Last} will henceforth contain one token that will be continuously updated while the Petri net is running. \\

\noindent \underline{Transition \texttt{Generate}:}\\
\texttt{Generate} fires after reading the token containing the element $P$ from input, and consuming a token containing a prime $q$ from the place \texttt{Primes}. \texttt{Generate} then reduce $P$ modulo $q$, placing a token on the place \texttt{ModIn}, containing the following fields: (i) the reduction $P_{q}$ and (ii) the prime $q$.\\ 

The goal of the next two transitions, the transition \texttt{GenPrime if C $\equiv$ m-1[m]} and the transition \texttt{If C $\not\equiv$ m-1[m]}, is to enlarge the set of tokens containing primes, on the place \texttt{Primes}, by the number of tokens that was placed on the place \texttt{Bal1}, in batches of $m$ primes. The tokens that are added to the place \texttt{Primes} should contain primes that are, and were, not contained in tokens on the place \texttt{Primes} yet. We describe these transitions in detail.\\

\noindent \underline{Transition \texttt{GenPrime if C $\equiv$ m-1[m]}:}\\
The transition \texttt{GenPrime if C $\equiv$ m-1[m]} only fires if the counter \texttt{C} has a value of $m-1$ mod $m$. After consuming a token from \texttt{Bal1} and the token from \texttt{Last}, containing the prime $p'$, \texttt{GenPrime if C $\equiv$ m-1[m]} places $m$ tokens containing the next $m$ primes, respectively, in decreasing order smaller than $p'$, on the place \texttt{Primes}, as well as a token containing the smallest of the $m$ primes on \texttt{Last}. The transition then increments the counter \texttt{C} by $1$. Transition \texttt{GenPrime if C $\equiv$  m-1[m]} hence enlarges the set of primes that the transition \texttt{Generate} has access to by $m$, and updates the value of the prime in the token on the place \texttt{Last}.\\

\noindent \underline{Transition \texttt{If C $\not\equiv$ m-1[m]}:}\\
The transition \texttt{If C $\not\equiv$ m-1[m]} only fires if the counter \texttt{C} has a value that is not equal to $m-1$ mod $m$. The transition consumes a token from the place \texttt{Bal1} and update the value of the counter by adding $1$ to the current value. By this, the transition ensures that the counter \texttt{C} gives a close approximation, that is also a lower bound, of the number of tokens that was placed on \texttt{Bal1} by the transition \texttt{Compute}, minus the number of primes that was placed on \texttt{Primes} by the transition \texttt{If  C $\equiv$ m-1[m]}. \\

\noindent \underline{Transition \texttt{Compute}:}\\
For the transition \texttt{Compute} to fire there should be at least one token on \texttt{Bal2} and the token on \texttt{Bal3} should have boolean value \texttt{False}. After reading the value of the token on \texttt{Bal3}, and consuming a token from \texttt{Bal2}, and a token containing  a reduction $P_q$ and a prime $q$, as fields, from \texttt{ModIn}, it computes the output $Q_q$ as a reduced Gr\"{o}bner basis of the given algorithm $\mathcal A$ over $\mathbb Z/q$. \texttt{Compute} then places a token on \texttt{ModRes1}, containing the following fields: (i) the output $Q_N$, $N=q$, (ii) the product $N=q$, (iii) the number of primes in the product $N=q$ (in this case it is 1), and (iv) $\LM(Q_N)$. We call a token of such type a token of type \texttt{modular data}.\\ 

\noindent \underline{Transitions \texttt{Ct $\not \equiv$ 0 mod 10} and \texttt{Ct $\equiv$ 0 mod 10}:}\\
The two transitions attaching the place \texttt{ModRes1} with the place \texttt{ModRes2}, and with the place \texttt{Test}, respectively, fire sequentially. The transition  \texttt{if  Ct $\not \equiv$ 0 mod 10} fires 9 times, after which the transition with condition \texttt{Ct $\equiv$ 0 mod 10} fires once, repeatedly. This is achieved by the transitions \texttt{ Ct $\not \equiv$ 0 mod 10} and \texttt{Ct $\equiv$ 0 mod 10} consuming a token from the counter $\texttt{Ct}$ and \texttt{ModRes1}, placing the token from \texttt{Ct} back, with a value increased by one, and passing the token from \texttt{ModRes1} to \texttt{ModRes2}, and \texttt{test}, respectively. Since there is only one or zero tokens on \texttt{Ct}, both transitions cannot consume the token at the same time, having a sequential firing as a consequence.\\

\noindent \underline{Subnet \texttt{Lift}:}\\
The subnet \texttt{Lift} in general combines two tokens of type \texttt{modular data}, of which the field (iv), that is $\LM(Q_N)$, agrees, into one token of type \texttt{modular data}. The subnet combines new tokens with new tokens, or already combined tokens. \texttt{Lift} uses the Chinese Remainder Theorem in \textsc{Singular}, as described in Section \ref{rational reconstruction}, to lift the reduced Gr\"{o}bner bases in the fields (i) of two tokens, that is $Q_{N_1}$ and $Q_{N_2}$, into one reduced Gr\"{o}bner basis $Q_{N_1\cdot N_2}$, which is saved in the field (i) of the newly created token. The product of the primes in the products $N_1$ and $N_2$ are saved in the field (ii). The number of primes in the field (iii) is increased to the sum of the number of primes in $N_1$ and $N_2$. \texttt{Lift} stops combining a token with other tokens as soon as the field (iii), that is the number of primes in the product $N$, exceeds $M_1$ (a number that is specified by the user), and places the token on the place \texttt{ModRes3}. To make sure that tokens with many primes in the field (iii) are not hold up for a long time in the subnet \texttt{Lift}, \texttt{Lift} sporadically may also place tokens with less than, or equal to $\frac{M_1}{2}$ primes on \texttt{ModRes3}. Lastly, \texttt{Lift} places a token on the place \texttt{Bal2} every time two tokens are combined into one token, or when a token is placed on \texttt{ModRes3}.\\  

\noindent \underline{Subnet \texttt{Lift/Add}:}\\
Similar to the subnet \texttt{Lift}, the subnet \texttt{Lift/Add} combines two tokens of type \texttt{modular data} into one token of the same type until the number of primes in the field (iii) of the token exceeds $M_2$ primes (a number that is specified by the user). It then combines these tokens with tokens of the same type in \texttt{Accumulator} and places the combined tokens on the place \texttt{LiftedRes}. Note that the subnet \texttt{Lift/Add} will only place a token on the place \texttt{LiftedRes}, if a token was consumed from the place \texttt{Accumulator}. This implies that there will always be one token in total on the places \texttt{LiftedRes}, \texttt{RandRes}, \texttt{RandOut}, \texttt{Result}, and \texttt{Accumulator}. Again, similar to \texttt{Lift}, to make sure that tokens are not held up for a long time in the subnet \texttt{Lift/Add}, the subnet sporadically may combine tokens with less than $M_2$ primes with tokens in the place \texttt{Accumulator}, which is then placed on \texttt{LiftedRes}.\\ 

\noindent \underline{Subnet \texttt{ReconsTest}:}\\
The subnet \texttt{ReconsTest} consumes tokens of type \texttt{modular data} from the place \texttt{LiftedRes}, and tokens of type \texttt{modular data} from the place \texttt{Test}. Firstly, it places a structureless token on the place \texttt{Bal2}. It computes the Farey lift of the result $Q_N$ in the field (i) of the consumed token from \texttt{LiftedRes} (if a coefficient in $Q_N$ is outside the image of the Farey map, the coefficient is mapped to $0$). The result of the Farey lift is then added in a new field (v) to the token. Next, the subnet checks whether all the coefficients in the new field (v) of the token are not divisible by the prime $p$ in the field (iii) of the token consumed from the place \texttt{Test}. If this is not the case, a new token is consumed from \texttt{Test}, and the divisibility test is repeated. If there is no coefficient that is divisible by $p$, the result in (v) is reduced modulo $p$ and compared to the result in the field (i) of the token that was consumed from the place \texttt{Test}. A boolean field (vi) is added to the token to reflect the comparison. Finally token is placed on the place \texttt{RandRes}. The subnet \texttt{ReconsTest} also changes the boolean value in the token on the place \texttt{Bal3} to true every time \texttt{ReconsTest} performs parallel tasks. Then, it turns the value back to false when the parallel tasks are done. This ensures that \texttt{ReconsTest} has access to more resources during the parallel computations to perform the test faster by disabling the transition \texttt{Compute} to fire.
\\

\noindent \underline{Transition \texttt{Break if RandRes.test = True}:}\\
The transition \texttt{Break} \texttt{if RandRes.test} \texttt{=} \texttt{True} consumes the  token from the place \texttt{RandRes} if (there is a token on \texttt{RandRes} and) the field (vi) is \texttt{true} and place it on the place \texttt{RandOut}.\\

\noindent \underline{Transition \texttt{Continue if RandRes.test = False}:}\\
The transition \texttt{Continue} \texttt{if} \texttt{RandRes.test} \texttt{=} \texttt{False} consumes the token from the place \texttt{RandRes} if (there is a token on \texttt{RandRes} and) the field (vi) is \texttt{False}, discards the fields (v) and (vi), that is the Farey lift and the boolean field, and place the token of type \texttt{modular data} on the place \texttt{Accumulator}.\\

\noindent \underline{Transition \texttt{Verify}:}\\
The transition \texttt{Verify} consumes the  token from the place \texttt{RandOut} (if there is a token on \texttt{RandOut}) and does a verification test on the result of the Farey lift in the field (v). Information of the input $P\in R_0\langle m_1,\ldots, m_s\rangle$, that was placed on the place \texttt{Input}, may be needed, and therefore read, for the verification. The boolean in the field (vi) is changed according to the outcome of the test. The token is then placed on the place \texttt{Result}.\\

\noindent \underline{Transition \texttt{Continue if Result.test = False}:}\\
The transition \texttt{Continue if Result.test = False} consumes the token from the place \texttt{Result} if (there is a token on \texttt{Result} and) the field (vi) is \texttt{False}. The fields (v) and (vi), that is the result of the Farey lift and the boolean field, are discarded and the token of type \texttt{modular data} is placed on the place \texttt{Accumulator}.\\

\noindent \underline{Transition \texttt{Break if Result.test = True}:}\\
The transition \texttt{Break if Result.test} \texttt{=} \texttt{True} consumes the token from the place \texttt{Result} if (there is a token on \texttt{RandRes} and) the field (vi) is \texttt{true} and place a token with only the field (v), that is the result of the Farey lift, on the place \texttt{Output}.\\

\begin{figure}[ht]
\centering
\scalebox{0.70}{
	\begin{tikzpicture}
		\node[place,tokens=1,label={90:$\texttt{Start}$}] (singular) at (-5,1) {};
		\node[place] (input) at (0,0) {\texttt{Input}};
		\node[place] (primes) at (-2.5,-1) {\texttt{Primes}};
		\node[transition] (generate) at (0,-2) {\texttt{Generate}};
		\draw[thick] (primes) edge[post,bend left=30] (generate);
		\draw[dashed] (input) edge[post] (generate);
		\node[place] (modinputs) at (0,-4) {\texttt{ModIn}};
		\draw[thick] (generate) edge[post] (modinputs);
		\node[transition] (compute) at (0,-6) {\begin{tabular}{c}
				\texttt{Compute} \\
				if\\
				\texttt{bal3 = False}
			\end{tabular}};
		\draw[thick] (modinputs) edge[post] (compute);
		\node[place] (modres1) at (0,-8.25) {\texttt{ModRes1}};
		\draw[thick] (compute) edge[post] (modres1);
		\node[transition] (init) at (-5,-1) {\texttt{Init}};
		\draw[thick] (init) edge[post, bend left=20] (input);
		\draw[-{Implies}, double,double distance=1.5pt, line width=0.5pt] (init) to (primes);
		\node[transition] (generateprime)  at (-2.5,-4){
		\begin{tabular}{c}
			\texttt{GenPrime} \\
			\texttt{if} \\
			\texttt{C} $\equiv$ \texttt{m-1} [m]
	\end{tabular}};
		\node[place,tokens = 1, label={180:$\texttt{C}$}] (c) at (-4.5,-5.8) {};
		\node[transition] (countbal) at (-3.8,-7.7) {
		\begin{tabular}{c}
			\texttt{if} \\
			\texttt{C} $\not\equiv$ m-1 [m]
		\end{tabular}	
	};
		\draw[-{Implies}, double,double distance=1.5pt, line width=0.5pt] (generateprime) to (primes);
		\node[place] (lastprime) at (-3.5,-2){\texttt{Last}};
		\draw[thick] (init) edge[post] (lastprime);
		\draw[thick] (lastprime) edge[post, bend right=20] (generateprime);
		\draw[thick] (generateprime) edge[post, bend right=20] (lastprime);
		\node[place] (bal1) at (-2.75,-6) {\texttt{Bal1}};
		\draw[thick] (compute) edge[post] (bal1);
		\draw[thick] (bal1) edge[post, bend right = 10] (generateprime);
		\node[place] (bal2) at (-5,-9) {\texttt{Bal2}}; 
		\draw[thick] (bal2) edge[post,bend right=30] (compute);
		\draw[-{Implies},bend right=30, double,double distance=1.5pt, line width=0.5pt] (init) to (bal2);
		\draw[thick] (countbal) edge[post, bend left=20] (c);
		\draw[thick] (c) edge[post, bend left=20] (countbal);
		\draw[thick] (bal1) edge[post, bend left=10] (countbal);
		\draw[thick] (c) edge[post, bend left=20] (generateprime);
		\draw[thick] (generateprime) edge[post, bend left=20] (c);
		\node[transition] (usedtoken) at (-2.5,-10) {
			\begin{tabular}{c}
				\texttt{if}\\
				\texttt{Ct} $ \not\equiv$ \texttt{0} [10]
			\end{tabular}
		};
		\node[transition] (testtoken) at (1.5,-10) {
			\begin{tabular}{c}
				\texttt{if}\\
				\texttt{Ct} $\equiv$ \texttt{0} [10]
			\end{tabular}
		};
		\node[place, tokens = 1, label={270:$\texttt{Ct}$}] (ct) at (-0.5,-12) {};

		\draw[thick] (singular) edge[post] (init);
		\draw[thick] (ct) edge[post,bend left=10] (testtoken);
		\draw[thick] (testtoken) edge[post,bend left=10] (ct);
		\draw[thick] (ct) edge[post,bend right=10] (usedtoken);
		\draw[thick] (usedtoken) edge[post,bend right=10] (ct);
		\draw[thick] (modres1)  edge[post,bend left = 30] (testtoken);
		\draw[thick] (modres1) edge[post,bend right=30] (usedtoken);
		\node[place] (modres2)  at (-2.5,-12.5) {\texttt{ModRes2}};
		\draw[thick] (usedtoken) edge[post] (modres2);
		\node[cloud, draw, fill = gray!10, aspect=2, cloud puffs=16] (lift1) at (-2.5,-15) {\texttt{Lift}};
		\draw[thick] (modres2) edge[post] (lift1);
		\draw[thick] (lift1) edge[post, bend left=30] (bal2);
		\node[place] (modres3) at (-2.5,-17.5)  {\texttt{ModRes3}};
		\draw[thick] (lift1) edge[post] (modres3);
		\node[cloud, draw, fill = gray!10, aspect=2, cloud puffs=16] (subnet1) at (-0,-20) {\texttt{Lift}$/$\texttt{Add}};
		\draw[thick] (modres3) edge[post,bend right=30] (subnet1);

		\node[place] (liftedres) at (3,-18) {\texttt{LiftedRes}};
		\node[place] (accumulator) at (10,-16) {\texttt{Accumulator}};
		\draw[thick] (subnet1) edge[post,bend right=10] (liftedres);
		\node[cloud, draw, fill = gray!10, aspect=2, cloud puffs=16] (reconstest) at (5,-14.5) {\texttt{ReconsTest}};
		\draw[thick] (accumulator) edge[post,bend left=30] (subnet1);
		\draw[thick] (liftedres) edge[post,bend right=20] (reconstest);
		\node[place] (test) at (1.5,-12.5)  {\texttt{Test}};
		\draw[thick] (testtoken) edge[post] (test);
		\draw[thick] (test) edge[post, bend right=20] (reconstest);
		\node[place] (randres) at (5,-11.5) {\texttt{RandRes}};
        \node[place] (bal3) at (3.8,-9) {\texttt{Bal3}};
        \draw[thick](reconstest) edge[post,bend left=20] (bal3);
        \draw[thick] (bal3) edge[post,bend right=30] (reconstest);
        \draw[thick](reconstest) edge[post,bend left=12] (bal3);
         \draw[thick] (bal3) edge[post,bend right=40] (reconstest);
        \draw[thick,dashed](bal3) edge[post,bend right = 10] (compute);
		\draw[thick] (reconstest) edge[post] (randres);
        \draw[thick](reconstest) edge[post, bend left = 50] (bal2);
		\node[transition] (false) at (9,-10.5) {
			\begin{tabular}{c}
				\texttt{Continue} \\
				if\\
				\texttt{RandRes.test = False}
			\end{tabular}
		};
		\node[transition] (true) at (5,-5.5) {
			\begin{tabular}{c}
				\texttt{Break} \\
				if\\
				\texttt{RandRes.test = True}
			\end{tabular}
		};
		\draw[thick] (false) edge[post, bend left=20] (accumulator);
		\draw[thick] (randres) edge[post, bend right = 30] (false);
		\draw[thick] (randres) edge[post] (true);
		\node[place] (randout) at (5,-3) {\texttt{RandOut}};
		\draw[thick] (true) edge[post] (randout);
		\node[transition] (verify) at (4.8,-1) {\texttt{Verify}};
		\draw[thick] (randout) edge[post] (verify);
		\draw[dashed] (input) edge[post,bend left=20] (verify);5
		\node[place] (result) at (7.5,-1) {\texttt{Result}};
		\draw[thick] (verify) edge[post] (result);
		\node[transition] (True) at (9,2) {
			\begin{tabular}{c}
				\texttt{Break} \\
				if\\
				\texttt{Result.test = True}
			\end{tabular}
		};
		\node[transition] (False) at (10,-7.5) {
			\begin{tabular}{c}
				\texttt{Continue} \\
				if\\
				\texttt{Result.test = False}
			\end{tabular}
		};
		\draw[thick] (False) edge[post, bend left=50] (accumulator);
		\draw[thick] (result) edge[post, bend left=30] (False);
		\draw[thick] (result) edge[post,bend right=30] (True);
		\node[place] (output) at (4, 2) {\texttt{Output}};
		\draw[thick] (True)  edge[post] (output);
	\end{tikzpicture}}
	\captionsetup{labelfont=bf}
	\caption{\textbf{Modular algorithm Petri net}}
	\label{fig:modular}
\end{figure}
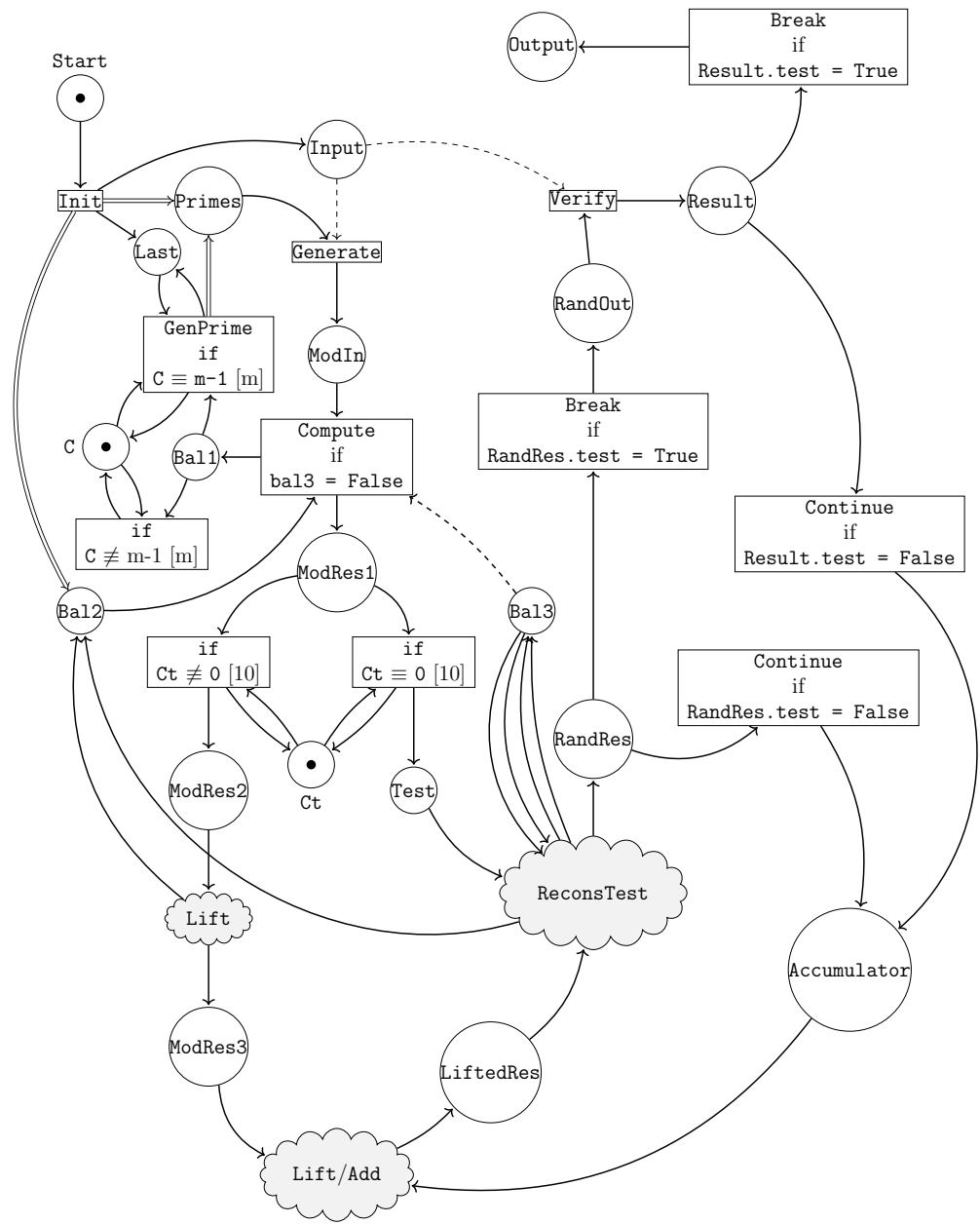

We now describe the different subnets in detail.

\subsection{Subnet Lift}
We first describe the subnet \texttt{Lift}, see Figure \ref{fig:lift}. Note that we initialize the subnet with a structureless token on the place \texttt{U} and the counter \texttt{Ct} set to $0$.\\

\noindent \underline{Transition \texttt{In}:}\\
The transition \texttt{In} consumes a token from \texttt{ModRes2} and places it on \texttt{Copy}. It also places a structureless token on the empty place \texttt{A}, with the purpose to increment the counter \texttt{Ct}.\\

\noindent \underline{Transition \texttt{if nb\_primes} $>$ \texttt{M$_1$ and Ct} $\ge \texttt{1}$:}\\
The transition \texttt{if nb\_primes} $>$ \texttt{M$_1$ and Ct} $\ge \texttt{1}$ consumes tokens from the place \texttt{Copy} if the number in the field (iii) of the token, that is the number of primes in the product $N$, is greater than $ \texttt{M}_1$ and the value of the counter \texttt{Ct} is greater or equal to $ 1$. This ensures that the counter \texttt{Ct} stays positive and that the majority of the tokens that is placed on the place \texttt{E} and consumed by the transition \texttt{Out} has more than $\texttt{M}_1$ primes in the product $N$.\\ 

\noindent \underline{Transition \texttt{if nb\_primes $\le$ M$_1$ and Ct$\ge$1}:}\\
The transition needs a token on the place \texttt{U} and a token on the place \texttt{Copy} with less that or equal to $M_1$ primes in the product $N$ in the field (ii) of the token, as well as a value greater than or  equal to $ 1$ of the counter \texttt{Ct} to fire. (These conditions ensure that the transition \texttt{if nb\_primes $\le$ M$_1$} fires sequentially; only tokens with fewer primes than M$_1$ in the product $N$ is lifted; and the counter \texttt{Ct} will never be negative.)
In this case, the transition consumes the structureless token from the place \texttt{U} and place it on the place \texttt{D}. It also consumes a token from the place \texttt{Copy} of which the number in field (iii) is less than or equal to M$_1$ and places the token on the place \texttt{R}. Lastly, it places a structureless token on the  place \texttt{C}, with the purpose to decrement the counter \texttt{Ct}.\\

\noindent \underline{Transition \texttt{if nb\_primes $\le$ M$_1$}:}\\
The transition \texttt{if nb\_primes $\le$ M$_1$} fires only if there is a token on place \texttt{D} and a token with less than of equal to $M_1$ primes in the product $N$ in the field (ii) of the token. The transition consumes the structureless token from the place \texttt{D} and places it on the place \texttt{U}. It also consumes a token, with less than or equal to $M_1$ primes in the product $N$ in the field (ii) of the token, from the place \texttt{Copy} and places it on the place \texttt{L}.\\

\noindent \underline{Transition \texttt{if Ct $=0$ and nb\_primes $\ge \frac{\texttt{M}_1}{2}$}:}\\
 The purpose of the transition \texttt{if Ct $=0$ and nb\_primes $\ge \frac{\texttt{M}_1}{2}$} is to avoid a situation in which there is a token, with many primes in the product $N$ in the field (ii), on the place \texttt{R} for a long time and the transition \texttt{Lift1} does not fire. It is only desirable, and not a prerequisite, that the value in the field (iii) in the tokens on place \texttt{ModRes3} is greater than  M$_1$ for the rest of the Petri net to work. In fact it is better to use tokens of which the number in the field (iii) is less than M$_1$ in the rest of the Petri net, than to let these tokens lie for a long time in the subnet \texttt{Lift}. The purpose of the  transition \texttt{if Ct $=0$ and nb\_primes $\ge \frac{\texttt{M}_1}{2}$} is to release such tokens from the subnet \texttt{Lift} and put the subnet again in its initial state.

The transition \texttt{if} \texttt{Ct $= 0$ and nb\_primes $\ge \frac{\texttt{M}_1}{2}$} only fires when the counter \texttt{Ct}, which it only reads, has value $0$ and there is a token with a value greater than or equal to $\frac{M_1}{2}$ on the place $\texttt{R}$. Note that the counter \texttt{Ct} always has a value greater than or equal to the number of tokens on the place \texttt{Copy}. Hence, in the situation where the transition \texttt{if Ct $= 0$ and nb\_primes $\ge \frac{M_1}{2}$} will fire, there will be no token on place \texttt{L}, there will be a token on place \texttt{D} and there will be no token on place \texttt{U}. The transition \texttt{if Ct $= 0$ and nb\_primes $\ge \frac{\texttt{M}_1}{2}$} will then consume the token on places \texttt{D} and \texttt{R}, places the token that was on place \texttt{R} on place \texttt{E}, and the structureless token on place \texttt{D} on place \texttt{U}, as well as a structureless token on place $B$. The purpose of the last placement is to increase the number of tokens on the place \texttt{Bal2}.\\

\noindent \underline{Transition \texttt{Lift1}:}\\
The transition \texttt{Lift1} consumes a token from the place \texttt{L} and the place \texttt{R}. Suppose the fields of the token consumed from the places \texttt{L} and \texttt{R} are, respectively, (i) $Q_{N_L}$ (ii) $N_L$ (iii) $M_L$ (iv) $\LM(Q_{N_L})$ and (i) $Q_{N_R}$ (ii) $N_R$ (iii) $M_R$ (iv) $\LM(Q_{N_R})$. \texttt{Lift1} first verifies whether the two tokens can be combined into one token by checking whether the sets in the field (iv), that is, the leading monomials of the reduced Gr\"{o}bner basis of $Q_{N_L}$ and $Q_{N_R}$, respectively, are the same. In case the sets are not the same, one of the tokens with the biggest value in  the field (iii), that is, the number of primes in the product $N_L$ and $N_R$, respectively, is placed on the place \texttt{Copy}, a structureless token is placed on the place \texttt{B}, and the other token is discarded. If the sets in the fields (iv) are the same a new token is created with the following fields (i) the lift $Q_N$ of $Q_{N_L}$ and $Q_{N_R}$, using the Chinese Remainder Theorem (ii) $N=N_L\cdot N_R$ (iii) $M_L+M_R$ (iv) $\LM(Q_{N_L})=\LM(Q_{N_R})=\LM(Q_N)$. The newly created token is placed on the place \texttt{Copy} and a structureless token is placed on the place \texttt{B}. Note that the structureless tokens on the place \texttt{B} will be moved  by a transition to the place \texttt{Bal2}.\\

\noindent \underline{Transitions \texttt{Cancel1}, \texttt{Cancel2}, \texttt{F}, \texttt{Ct+1} and \texttt{Ct-1}:}\\
The purpose of the transitions \texttt{Cancel1} and \texttt{Cancel2} is to make the counter \texttt{Ct} act in a less sequential, and more parallel manner. Consequently the counter \texttt{Ct} gives in general an up-to-date value of the number of tokens on the place \texttt{Copy}. The transition \texttt{Cancel2} consumes a token from the place \texttt{A} and a token from the place \texttt{C} and discard both tokens. Note that the number of tokens on the place \texttt{A} is the number of tokens with which the number of tokens on the place \texttt{Copy} is increased and the number of tokens on the place \texttt{C} is the number of tokens with which the place \texttt{Copy} is decreased. Likewise, the transition \texttt{Cancel1} consumes a token from the place \texttt{D} and a token from the place \texttt{A} and discard the tokens. The transition \texttt{F} consumes tokens from  the place \texttt{C} and places the tokens on the place \texttt{D}. Therefore \texttt{Cancel1} does in essence the same as \texttt{Cancel2}. The purpose of transition \texttt{F} is simply that it takes time to consume and place tokens. Hence the transition \texttt{Ct-1} will fire in a delayed manner with respect to the transition \texttt{Ct+1}. The purpose of the transition \texttt{Ct+1} is to increase the counter by $1$ with every token consumed from the place \texttt{A}, and the purpose of the transition \texttt{Ct-1} is to decrease the counter by $1$ with every token it consumed from the place \texttt{D}. We delay the decreasing of the counter with regard to the increasing of the counter to ensure that the value of the counter is greater than or equal to the exact number of tokens on the place \texttt{Copy}. We do this to avoid the situation where tokens on the place \texttt{R} are unnecessarily taken out of the lifting machinery and placed on the place \texttt{E} by the transition \texttt{if Ct=0 and nb\_primes $\ge \frac{\texttt{M}_1}{2}$}.\\ 

\noindent \underline{Transition \texttt{Out}:}\\
The transition \texttt{Out} consumes tokens from the place \texttt{E} and places them, outside this subnet, on the place \texttt{ModRes3}.\\

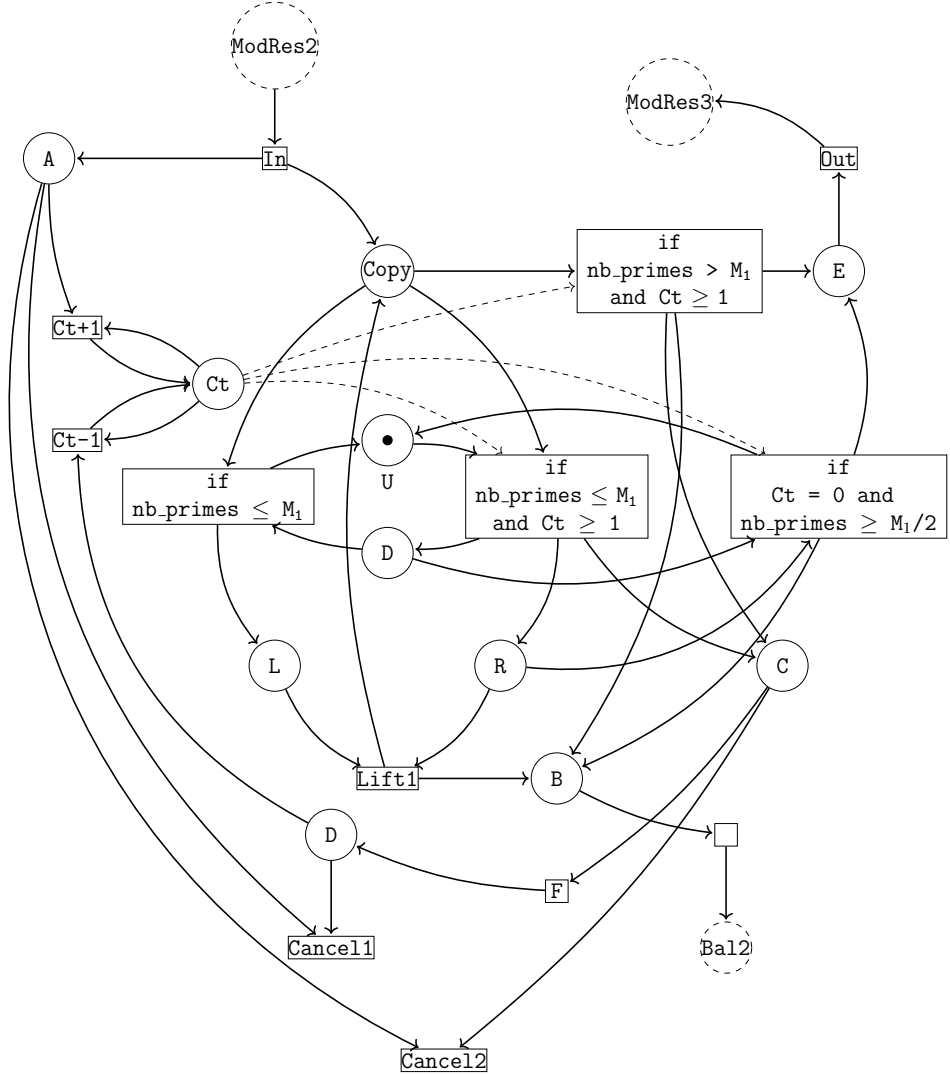
\begin{figure}[ht]
\scalebox{0.75}{
\begin{tikzpicture}
\node[transition] (input) at (-2,0) {\texttt{In}};
\node[place] (copy) at (0,-2) {\texttt{Copy}};
\node[place] (precount)  at (-6,0) {\texttt{A}};
\node[transition] (countup) at (-5.5,-3) {\texttt{Ct+1}};
\node[place] (count) at (-3,-4) {\texttt{Ct}};
\node[transition] (countdown) at (-5.5,-5) {\texttt{Ct-1}};
\node[transition] (preout) at (5,-2) {
\begin{tabular}{c}
\texttt{if}\\
\texttt{nb\_primes > M}$_{\texttt{1}}$\\
\texttt{and Ct} $\geq$ \texttt{1}
\end{tabular}
};
\node[transition] (tright) at (3,-6) {
\begin{tabular}{c}
\texttt{if}\\
\texttt{nb\_primes} $\leq$ \texttt{M}$_{\texttt{1}}$\\
\texttt{and Ct $\geq$ 1}
\end{tabular}
};
\node[transition] (tleft) at (-3,-6){
\begin{tabular}{c}
\texttt{if}\\
\texttt{nb\_primes $\leq$ M}$_\texttt{1}$
\end{tabular}
};
\node[place,tokens=1,label={270:$\texttt{U}$}] (up) at (0,-5) {};
\node[place] (down) at (0,-7) {\texttt{D}};
\node[place] (right) at (2,-9) {\texttt{R}};
\node[place] (left) at (-2,-9) {\texttt{L}};
\node[transition] (lift1) at (0,-11) {\texttt{Lift1}};
\node[place] (out) at (8,-2) {\texttt{E}};
\node[transition] (extraction) at (8,-6) {
\begin{tabular}{c}
\texttt{if} \\
\texttt{Ct = 0 and } \\
\texttt{nb\_primes $\geq $ M$_1$/2}
\end{tabular}
};
\node[place] (bal) at (3,-11){\texttt{B}};
\node[place] (minus1) at (7,-9) {\texttt{C}};
\node[transition] (bridge) at (3,-13) {\texttt{F}};
\node[place] (minus) at (-1,-12) {\texttt{D}};
\node[transition] (cancel1) at (-1,-14) {\texttt{Cancel1}};
\node[transition] (cancel2) at (1,-16) {\texttt{Cancel2}};
\node[transition] (pout) at (8,0) {\texttt{Out}};
\node[transition] (outtobal) at (6,-12){};
\node[place, dashed] (modres1) at (-2,2) {\texttt{ModRes2}};
\node[place, dashed] (bal2) at (6,-14) {\texttt{Bal2}};
\node[place, dashed] (modres2) at (5,1) {\texttt{ModRes3}};

\draw[thick] (input) edge[post, bend left=20] (copy);
\draw[thick] (input) edge[post] (precount);
\draw[thick] (precount) edge[post, bend right=10] (countup);
\draw[thick] (countup) edge[post,bend right=20] (count);
\draw[thick] (count) edge[post, bend right=20] (countup);
\draw[thick] (countdown) edge[post, bend left=20] (count);
\draw[thick] (count) edge[post,bend left=20] (countdown);
\draw[thick] (copy) edge[post, bend right=20] (tleft);
\draw[thick] (copy) edge[post, bend left=20] (tright);
\draw[thick] (copy) edge[post] (preout);
\draw[thick] (preout) edge[post] (out);
\draw[thick] (up) edge[post, bend left=10] (tright);
\draw[thick] (tright) edge[post,bend left=10] (down);
\draw[thick] (down) edge[post, bend left= 10]  (tleft);
\draw[thick] (tleft) edge[post,bend left=10] (up);
\draw[thick] (tright) edge[post, bend left=20] (right);
\draw[thick] (tleft) edge[post, bend right=20] (left);
\draw[thick] (left) edge[post, bend right=20] (lift1);
\draw[thick] (right) edge[post, bend left =20] (lift1);
\draw[thick] (lift1) edge[post] (bal);
\draw[thick] (right) edge[post, bend right=30] (extraction);
\draw[thick] (down) edge[post, bend right=20] (extraction);
\draw[thick] (extraction) edge[post, bend right=20] (up);
\draw[thick] (extraction) edge[post, bend left=20] (bal);
\draw[thick] (extraction) edge[post, bend right=20] (out);
\draw[thick] (preout) edge[post, bend right=20] (minus1);
\draw[thick] (minus1) edge[post, bend left=10] (bridge);
\draw[thick] (bridge) edge[post, bend left=10] (minus);
\draw[thick] (minus1) edge[post, bend left=10] (cancel2);
\draw[thick] (minus) edge[post, bend left=30] (countdown);
\draw[thick] (minus) edge[post] (cancel1);
\draw[thick] (precount) edge[post, bend right=30] (cancel1);
\draw[thick] (precount) edge[post, bend right=40] (cancel2);
\draw[dashed] (count) edge[post, bend left=20] (tright);
\draw[dashed] (count) edge[post, bend left=20] (extraction);
\draw[dashed] (count) edge[post, bend left=5] (preout);
\draw[thick] (out) edge[post] (pout);
\draw[thick] (preout) edge[post, bend left=20] (bal);
\draw[thick] (bal) edge[post, bend right=10] (outtobal); 
\draw[thick] (tright) edge[post, bend right=20] (minus1);
\draw[thick] (lift1) edge[post, bend left=15] (copy);
\draw[thick] (modres1) edge[post] (input);
\draw[thick] (outtobal) edge[post] (bal2);
\draw[thick] (pout) edge[post, bend right=20] (modres2);
\end{tikzpicture}
}
\caption{\textbf{Subnet Lift}}
\label{fig:lift}
\end{figure}

\subsection{Subnet Lift/Add}
We now describe the functioning of the subnet \texttt{Lift/Add} presented in Figure \ref{fig:liftadd}. Note that we initialize the Petri net with a structureless token on place \texttt{U} and the token on the  counter \texttt{Ct} is set to $0$.\\

\noindent \underline{Transition \texttt{Count}:}\\
The transition \texttt{Count} consumes tokens from the place \texttt{ModRes3}, increments the counter \texttt{Ct} by 1, for each token it consumes, and places the tokens on the place \texttt{ModRes4}. The purpose of the counter \texttt{Ct} is to give a lower bound for the number of tokens on the place \texttt{ModRes4}. \\

\noindent \underline{Transition \texttt{if Ct $\ge2$ and nb\_primes $\le$ M$_2$}:}\\
The transition \texttt{if Ct$\ge2$ and nb\_primes $\le$ M$_2$} fires only if there is a token on the place $U$ and the counter \texttt{Ct} has a value $\ge 2$. The transition consumes a token of type \texttt{modular data} with a value $\le M_2$ in the field (iii) of the token from the place \texttt{ModRes4}, and the structureless token from the place \texttt{U}. It then places the token of type \texttt{modular data} on place \texttt{R}, and a structureless token on the place \texttt{D}.\\

\noindent \underline{Transition \texttt{Ct-2}:}\\
The transition \texttt{Ct-2} only fires if there is a token on place \texttt{D}. The transition consumes any token of type \texttt{modular data} from the place \texttt{ModRes4} and the structureless token from the place \texttt{D}. It then places the structureless token on place \texttt{U}, the token of type \texttt{modular data} on place \texttt{L}, and decrements the counter \texttt{Ct} by $2$.

Note that the transitions \texttt{if Ct$\ge$2 and nb\_primes $\le$ M$_2$}, followed by the transition \texttt{Ct-2}, will only fire if there are at least two tokens on \texttt{ModRes4} of which one has  less than or equal to $ M_2$ primes.

Furthermore note that the interaction of the transitions \texttt{if Ct$\ge2$ and nb\_primes $\le$ M$_2$} and \texttt{Ct-2} with the places \texttt{U} and \texttt{D} ensures that the transitions fire sequentially.\\ 

\noindent \underline{Transition \texttt{Lift2}:}\\
The transition \texttt{Lift2} consumes two tokens of type \texttt{modular data}, one from the place \texttt{L} and one from the place \texttt{R}. If the tokens are compatible, that is the sets in the fields (iv) of the tokens are the same, it combines the tokens into a new token in the same way as the transition \texttt{Lift1} in the subnet \texttt{Lift} (see Figure \ref{fig:lift}) combines two tokens into a new token of type \texttt{modular data}. If the two tokens are not compatible, a token with the biggest number in the field (iii) is placed on the place \texttt{ModRes5} and the other token is discarded.\\

\noindent \underline{Transitions \texttt{if nb\_primes$>$M$_2$} and \texttt{if nb\_primes$\le$M$_2$, Ct+1}:}\\
Every token on the place \texttt{ModRes5} is consumed and placed on the place \texttt{ModRes4} by either the transition \texttt{if nb\_primes $>$ M$_2$} or the transition \texttt{if nb\_primes $\le$ M$_2$, Ct+1}, depending on the value in the field (iii) of the token. The transition \texttt{if nb\_primes $\le$ M$_2$, Ct+1}, in addition, increment the counter \texttt{Ct} by $1$. We hence do not count the tokens with a value of more than $M_2$ in the field (iii) of the token. This means that the transitions \texttt{if Ct $\ge$ 2 and nb\_primes $\le$ M$_2$} and \texttt{Ct-2}, and hence \texttt{Lift2}, will eventually stop to fire, if no new tokens are placed on \texttt{ModRes4} by the transition \texttt{Count}. This ensures that the number in the field (iii) of the tokens, that is, the number of primes in the product $N$, does not become so big that it slows down the subnet \texttt{ReconsTest} considerably.\\

\noindent \underline{Transitions \texttt{Lift3, Ct-1}:}\\
The transition \texttt{Lift3, Ct-1} only fires if there is a token on the place \texttt{U}. Once it fires it consumes the token from \texttt{U} and places it back once it is finished with its task. This implies that the transition \texttt{Lift3, Ct-1}, the transition \texttt{if Ct$\ge$2 and nb\_primes $\le$ M$_2$} and the transition \texttt{Ct-2} cannot fire at the same time, which avoids the situation where the transition \texttt{if Ct$\ge$2 and nb\_primes $\le$ M$_2$} and the transition \texttt{Lift3, Ct-1} consumes a token at the same time, with no tokens left on the place \texttt{ModRes4} for transition \texttt{Ct-2} to consume. This could create a situation where a token may be unable to move from place \texttt{R} for a considerable amount of time.\\
The transition \texttt{Lift3, Ct-1} consumes two tokens of type \texttt{modular} \texttt{data}, one from the place \texttt{ModRes4} and another one from the place \texttt{Accumulator}, which, if compatible, it combines into a new token, similar to the transition \texttt{Lift2}, which it then places on the place \texttt{LiftedRes}. If the tokens are not compatible, it places the token that was consumed from \texttt{ModRes4} on the place \texttt{LiftedRes} and discards the token that was consumed from \texttt{Accumulator}. \texttt{Lift3, Ct-1} also decrements the counter \texttt{Ct} by 1 for each token it consumes from \texttt{ModRes4}.\\

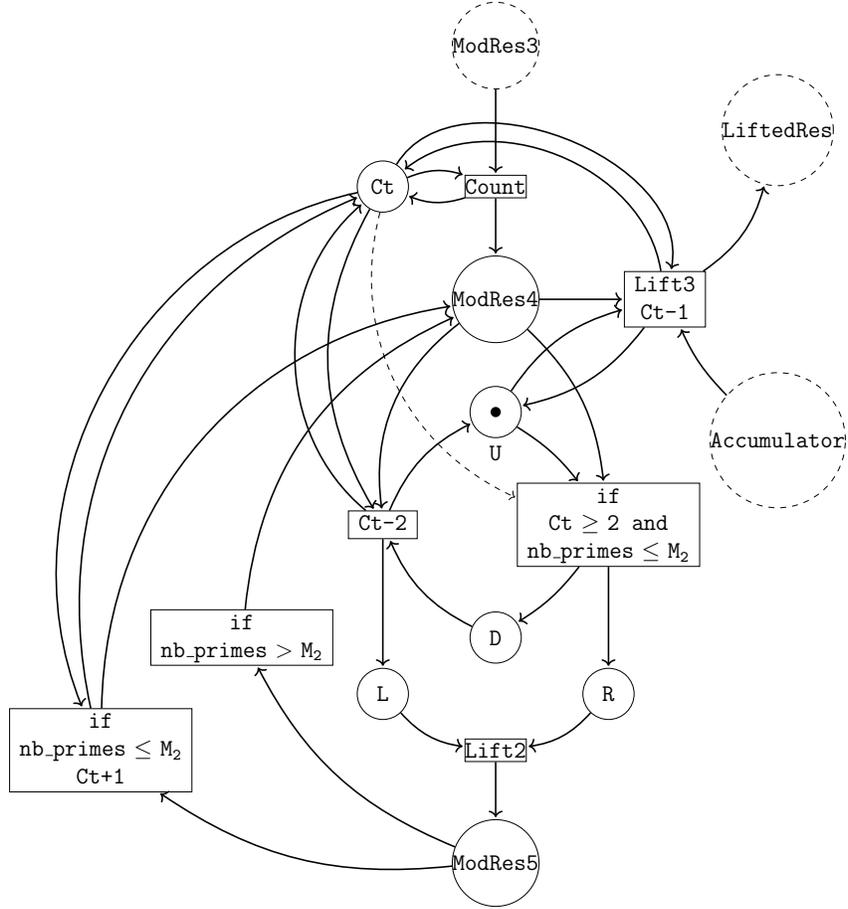
\begin{figure}[ht]
\centering
\scalebox{0.75}{
\begin{tikzpicture}
\node[transition] (count) at (0,-2) {\texttt{Count}};
\node[place,dashed] (modres3) at (0,0.5) {\texttt{ModRes3}};
\node[place] (ct) at (-2,-2) {\texttt{Ct}};
\node[place] (modres4) at (0,-4) {\texttt{ModRes4}};
\node[transition] (tright) at (2,-8) {
	\begin{tabular}{c} 
	\texttt{if}\\
	\texttt{Ct} $\geq $ \texttt{2 and} \\ 
	\texttt{nb\_primes} $\leq$ \texttt{M}$_\texttt{2}$
	\end{tabular}
};
\node[transition] (tleft) at (-2,-8) {
	\begin{tabular}{c}
	\texttt{Ct-2}
	\end{tabular}
};
\node[place,tokens=1,label={270:$\texttt{U}$}] (up)  at (0,-6) {};
\node[place] (down) at (0,-10) {\texttt{D}};
\node[place] (right) at (2,-11) {\texttt{R}};
\node[place] (left) at (-2,-11) {\texttt{L}};
\node[transition] (lift2) at (0,-12) {\texttt{Lift2}};
 \node[place] (modres5) at (0,-14) {\texttt{ModRes5}};
\node[transition] (back) at (-4.5,-10) {
\begin{tabular}{c}
\texttt{if} \\
\texttt{nb\_primes} $>$ \texttt{M}$_\texttt{2}$
\end{tabular}
};

\node[transition] (backincrease) at (-7,-12) {
	\begin{tabular}{c}
	\texttt{if} \\
	\texttt{nb\_primes} $\leq$ \texttt{M}$_\texttt{2}$\\
	\texttt{Ct+1}
	\end{tabular}
};
\node[transition] (lift3) at (3,-4) {
\begin{tabular}{c}
	\texttt{Lift3} \\
	\texttt{Ct-1}
\end{tabular}
};
\node[place, dashed] (liftedres) at (5,-1) {\texttt{LiftedRes}};
\node[place, dashed] (accumulator) at (5,-6.5) {\texttt{Accumulator}};

\draw[thick] (count) edge[post,bend left = 20] (ct);
\draw[thick] (ct) edge[post, bend left = 20]  (count);
\draw[thick] (count) edge[post] (modres4);
\draw[thick] (modres4) edge[post] (lift3);
\draw[thick] (modres4) edge[post, bend right=30] (tleft);
\draw[thick] (modres4) edge[post, bend left=20] (tright);
\draw[thick] (up) edge[post, bend left=10] (tright);
\draw[thick] (tright) edge[post,bend left=10] (down);
\draw[thick] (down) edge[post,bend left=20] (tleft);
\draw[thick] (tleft) edge[post,bend left=20] (up);
\draw[thick] (tright) edge[post] (right);
\draw[thick] (tleft) edge[post] (left);
\draw[thick] (right) edge[post,bend left=20] (lift2);
\draw[thick] (left) edge[post, bend right=20] (lift2);
\draw[thick] (lift2) edge[post] (modres5);
\draw[thick] (modres5) edge[post,bend left=20] (back);
\draw[thick] (modres5) edge[post,bend left=20] (backincrease);
\draw[thick] (back) edge[post,bend left=30] (modres4);
\draw[thick] (backincrease) edge[post, bend left=40] (modres4);
\draw[thick] (backincrease) edge[post, bend left=40] (ct);
\draw[thick] (ct) edge[post, bend right=50] (backincrease);
\draw[thick] (up) edge[post,bend left=20] (lift3);
\draw[thick] (lift3) edge[post,bend left=20] (up);
\draw[thick] (lift3) edge[post, bend right=20] (liftedres);
\draw[thick] (accumulator) edge[post, bend left=10] (lift3);
\draw[thick] (ct) edge[post, bend right=30] (tleft);
\draw[thick] (tleft) edge[post,bend left=50] (ct);
\draw[dashed] (ct) edge[post,bend right=40] (tright);
\draw[thick] (ct) edge[post, bend left=80] (lift3);
\draw[thick] (lift3) edge[post, bend right=60] (ct);
\draw[thick] (modres3) edge[post] (count);
\end{tikzpicture}
}
\caption{\textbf{Subnet Lift/Add}}
\label{fig:liftadd}
\end{figure}

\subsection{Subnet ReconsTest}
Lastly, we describe the functioning of the subnet \texttt{ReconsTest} (see Figure \ref{fig:reconstest}).\\

\noindent \underline{Subnet \texttt{ParallelFarey}:}\\
The subnet \texttt{ParallelFarey}
 consumes tokens of type \texttt{modular} \texttt{data} from the place \texttt{LiftedRes}. It firstly place a structureless token on the place \texttt{Bal2}. The data in the field (i) of a consumed token contains a set of polynomials. The subnet lift each of the coefficients of the polynomials to a preimage under the Farey map. If a coefficient is not in the image of the Farey map, the coefficient is mapped to $0$. The subnet adds a new field (v) to the the consumed token with the lifted set of polynomials. The token is then placed on the place \texttt{FareyRes}. The subnet also  turns the value in the token on the place \texttt{Bal3} to true every time \texttt{ParallelFarey} perfoms parallel tasks. It sets it back to false when parallel processes are done.   \\

 \noindent \underline{Subnet \texttt{ParCompatible}:}\\
 In general, the  subnet \texttt{ParCompatible} consumes tokens from the place \texttt{FareyRes} and the place \texttt{Test}. The subnet determines whether the numerators and denominators of the coefficients of each of the polynomials in the field (v) of the token consumed from the place \texttt{FareyRes} is coprime to the prime $p$ in the field (ii) of the token consumed from the place \texttt{Test}. If this is the case it means that the result in the field (v) of the token consumed from \texttt{FareyRes} can be reduced modulo the prime $p$. The subnet adds the following fields to the consumed token from the place \texttt{FareyRes}: a boolean field (vi) which reflects its compatibility with the prime $p$ as described; a field (vii) containing the data of the field (i) of the token consumed from \texttt{Test}; as well as a field (viii) containing the prime $p$. The token is then placed on the place \texttt{ResTest}. The subnet does the described compatibility test in parallel for each of the polynomials in the field (v) of the token consumed from the place \texttt{FareyRes}. As in the subnet \texttt{ParallelFarey}, the subnet \texttt{ParCompatible} changes the value in the token on the place \texttt{Bal3} depending on the existence of parallel tasks to be executed. If it is the case, it sets it to true. If it is not the case, it sets it back to false.\\

 \noindent \underline{Transition \texttt{if Test = True}:}\\
 The transition \texttt{if Test = True} consumes tokens from the place \texttt{RestTest} of which the boolean field (vi) is true. A consumed token is placed on the place \texttt{Res}.\\

\noindent \underline{Transition \texttt{if Test = False}:}\\
The transition \texttt{if Test = False} consumes tokens from the place \texttt{RestTest} of which the boolean field (vi) is false. The transition removes the fields (vi), (vii) and (viii) of a consumed token and  place it back on the place \texttt{FareyRes}.\\

\noindent \underline{Transition \texttt{Compare}:}\\
The transition  \texttt{Compare} consumes tokens from the place \texttt{Res}. It reduces the polynomials in the field (v) of a consumed token by the prime in the field (viii) and compare it with the data in the field (vii). If the compared data are not equal the boolean in the field (vi) is updated to false. The transition then discard the fields (vii) and (viii) and place the token of type \texttt{boolean modular data} on the place \texttt{RandRes}.\\

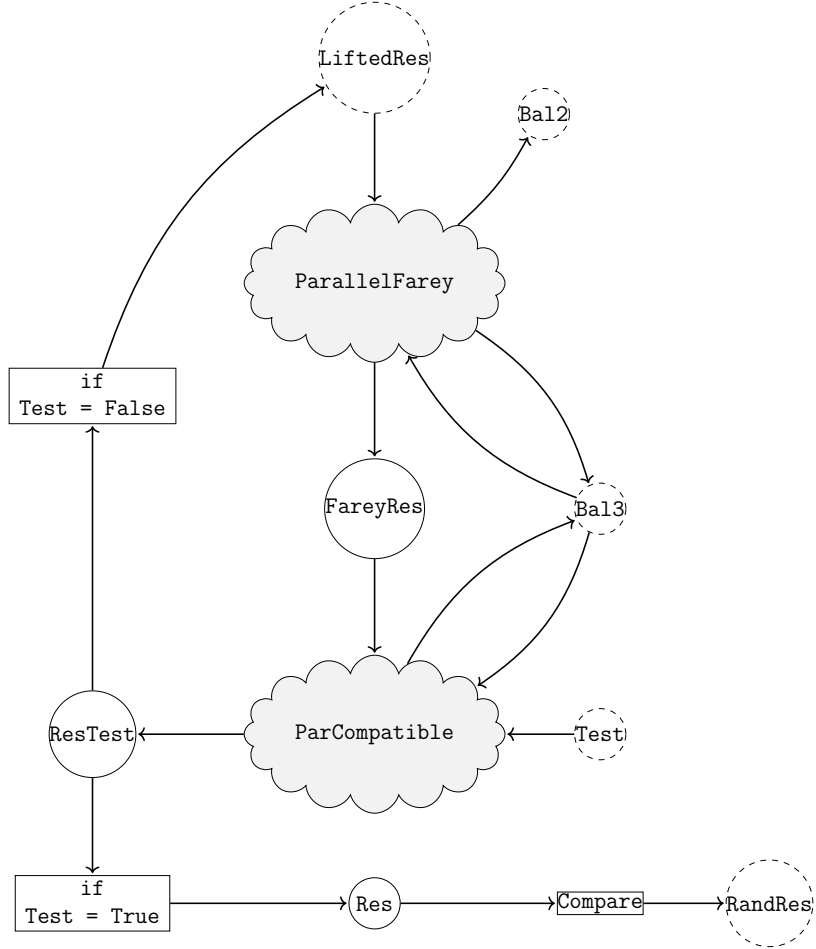
\begin{figure}[ht]
\centering
\scalebox{0.75}{
	\begin{tikzpicture}
		\node[place, dashed] (liftedres) at (0,0) {\texttt{LiftedRes}};
        \node[place,dashed] (bal2) at (3,-1) {\texttt{Bal2}};
		\node[cloud, draw, fill = gray!10, aspect=2, cloud puffs=16] (parallelfarey) at (0,-4) {\texttt{ParallelFarey}};
		\draw[thick] (liftedres) edge[post] (parallelfarey);
		\node[place] (fareyresult) at (0,-8) {\texttt{FareyRes}};
        \node[place,dashed] (bal3) at (4,-8) {\texttt{Bal3}};
        \draw[thick](parallelfarey) edge[post, bend right = 10] (bal2);
        \draw[thick] (bal3) edge[post,bend left=20] (parallelfarey);
         \draw[thick] (parallelfarey) edge[post,bend left=20] (bal3);
		\draw[thick] (parallelfarey) edge[post] (fareyresult);
		\node[cloud, draw, fill = gray!10, aspect=2, cloud puffs=16] (compatible) at (0,-12) {\texttt{ParCompatible}};
		\draw[thick] (bal3) edge[post,bend left=20] (compatible);
        \draw[thick] (compatible) edge[post,bend left=20] (bal3);
        \draw[thick] (fareyresult) edge[post] (compatible);
		\node[place,dashed] (test) at (4, -12) {\texttt{Test}};
		\draw[thick] (test) edge[post] (compatible);
		\node[place] (restest) at (-5,-12) {\texttt{ResTest}};
		\draw[thick] (compatible) edge[post] (restest);
		\node[transition] (False) at (-5,-6) {
		\begin{tabular}{c}
			\texttt{if} \\
			\texttt{Test = False}
		\end{tabular}	
	};
	\draw[thick] (restest)  edge[post, bend left=0] (False);
	\draw[thick] (False) edge[post,bend left=20] (liftedres);
	\node[transition] (True) at (-5,-15) {
		\begin{tabular}{c}
			\texttt{if} \\
			\texttt{Test = True}
		\end{tabular}	
	};
	\draw[thick] (restest) edge[post] (True);
	\node[place] (Res) at (0,-15) {\texttt{Res}};
	\node[transition] (compare) at (4,-15) {\texttt{Compare}};
	\draw[thick] (True) edge[post] (Res);
	\draw[thick] (Res) edge[post] (compare);
	\node[place,dashed] (randres) at (7,-15) {\texttt{RandRes}};
	\draw[thick] (compare) edge[post] (randres);
	\end{tikzpicture}
 }
	\caption{\textbf{Subnet ReconsTest}}
	\label{fig:reconstest}
\end{figure}%

Next, we describe all the subnets that appear in the subnet \texttt{ReconsTest} in detail.

\subsubsection{Subnet ParallelFarey}
The subnet \texttt{ParallelFarey}, see Figure \ref{fig:parallelfarey}, is initialized with a structureless token on the place \texttt{U}.\\

\noindent \underline{Transition \texttt{Split}:}\\
The transition \texttt{Split} consumes tokens from the place \texttt{LiftedRes} of type \texttt{modular data}.
The transition \texttt{Split}, splits the the generators of the result in the field (i) of the consumed token into different tokens with a field containing a single generator. These tokens are then all placed on the place \texttt{Gen} by the GPI-space feature \texttt{Connect-Out-Many}. The generators can be lift in parallel by the transformation \texttt{FareyLift}.
A token with a field containing the number of generators of the field (i) of the consumed token is placed on the place \texttt{NbGen}, as well as, a copy of the consumed token on the place \texttt{LiftInfo}. The transition \texttt{Split} also updates the boolean value in the token on the place \texttt{Bal3} to true. This  disables the transition \texttt{Compute} in the main net Figure \ref{fig:modular} and ensures that the subnet gets enough computational resources. \\

\noindent \underline{Transition \texttt{Farey}:}\\
The transition \texttt{Farey} consumes tokens, containing a field (poly) with a single polynomial, from the place \texttt{Gen}. It maps each of the coefficients of the polynomial to its preimage of the Farey map. If a preimage does not exist, the coefficient is mapped to $0$. The transition also adds a new field (\texttt{Nb}) to each token, giving the number of polynomials in the existing field (poly), which is at this stage $1$. The transition then place the token on the place \texttt{Res}.\\ 

\noindent \underline{Transitions \texttt{if Nb < NbGen}:}\\
The transitions \texttt{if Nb < NbGen} consumes a token from \texttt{Res}, by turn, only if the value in the field \texttt{Nb} is strictly smaller than the value of the token on the place \texttt{NbGen}. It then place the token on the place \texttt{L} or \texttt{R}, respectively.\\

\noindent \underline{Transition \texttt{Merge}:}\\
The transition \texttt{Merge} consumes a token from the place \texttt{L} and a token from the place \texttt{R}. Then, it creates a new token with the union of the sets in the fields (poly) of the two tokens in a field (poly), and a value that is the sum of the values in the fields \texttt{Nb} of the two tokens in a field \texttt{Nb}. The newly created token is placed back on the place \texttt{Res}.\\ 

\noindent \underline{Transition \texttt{if Nb = NbGen}:}\\
The transition \texttt{if Nb = NbGen} consumes the token from \texttt{NbGen} (leaving the place empty as the subnet was initialized), a token from the place \texttt{Res} with the same value in the field \texttt{Nb} as the value of the consumed token from \texttt{NbGen}, and the token from the place \texttt{LiftInfo}. Clearly the transition will not fire if there are no tokens with the same value in the field \texttt{Nb} as the value of the the token on \texttt{NbGen}. Hence the transition will only fire when all the tokens on the place \texttt{Res} is combined into one token. The transition adds the field (poly) of the token consumed from \texttt{Res} to the token consumed from \texttt{LiftInfo}, renaming it to (v). The newly formed token is then placed on the place \texttt{FareyRes}. The transition \texttt{if Nb = NbGen} also sets the value on the place \texttt{Bal3} back to false. \\

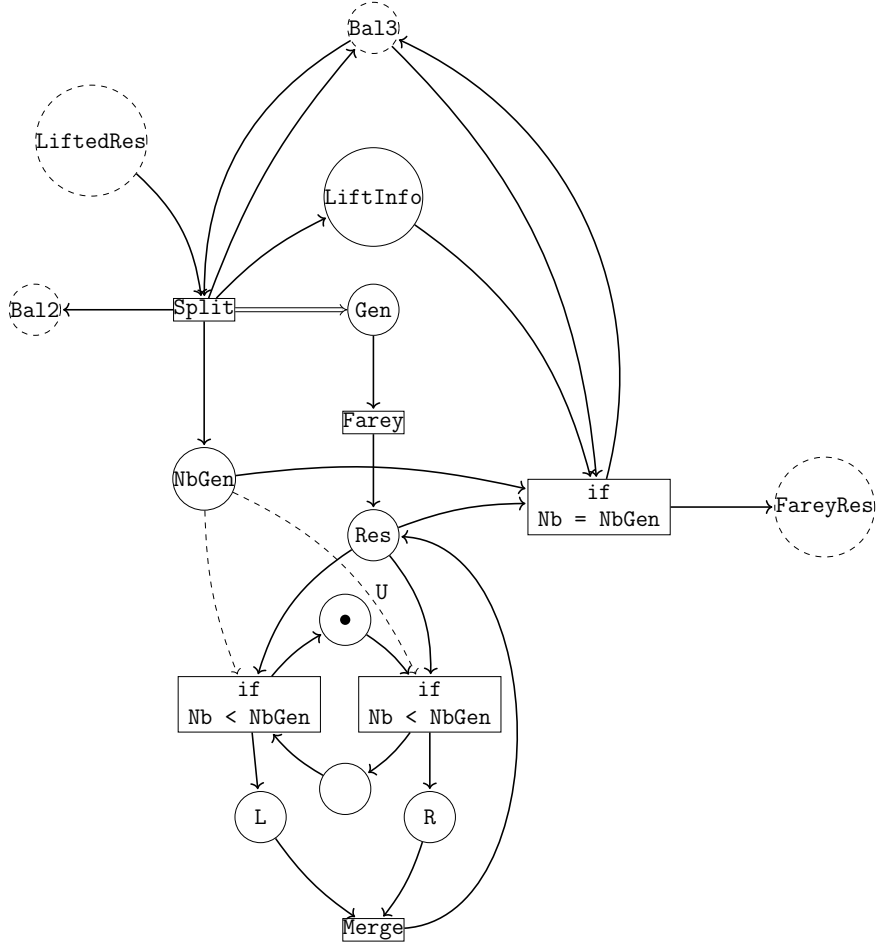
\begin{figure}[ht]
\centering
\scalebox{0.75}{
    	\begin{tikzpicture}
		\node[place, dashed] (liftedres) at (-2,0) {\texttt{LiftedRes}};
        \node[place,dashed](bal2) at (-3,-3) {\texttt{Bal2}};
		\node[transition] (split) at (0,-3) {\texttt{Split}};
        \draw[thick] (split) edge[post] (bal2);
		\draw[thick] (liftedres) edge[post, bend left=20] (split);
		\node[place] (gen) at (3,-3) {\texttt{Gen}};
		\node[place] (nbgen) at (0,-6) {\texttt{NbGen}};
		\node[place] (liftinfo)  at (3,-1) {\texttt{LiftInfo}};
        \node[place,dashed] (bal3) at (3,2) {\texttt{Bal3}};
	    \draw[-{Implies}, double,double distance=1.5pt, line width=0.5pt] (split) to (gen);
		\draw[thick] (split) edge[post] (nbgen);
		\draw[thick] (split) edge[post,bend left=10] (liftinfo);
		\node[transition] (farey) at (3,-5) {\texttt{Farey}};
		\node[place] (res) at (3,-7) {\texttt{Res}};
		\draw[thick] (gen) edge[post] (farey);
		\draw[thick] (farey) edge[post] (res);
        \draw[thick](bal3) edge[post, bend right=30] (split);
        \draw[thick] (split) edge[post,bend left=10] (bal3);
		\node[transition] (left) at (0.8,-10) {
		\begin{tabular}{c}
			\texttt{if} \\
			\texttt{Nb < NbGen}
		\end{tabular}	
	};
	\node[transition] (right) at (4,-10) {
	\begin{tabular}{c}
		\texttt{if} \\
		\texttt{Nb < NbGen}
	\end{tabular}	
};
	 \node[place,tokens=1,label={30:$\texttt{U}$}] (up) at (2.5,-8.5) {};
	 \node[place] (down) at (2.5,-11.5) {};
	 \draw[thick] (res) edge[post,bend right=20] (left);
	 \draw[thick] (res) edge[post,bend left=20] (right);
	 \draw[thick] (up) edge[post,bend left=10] (right);
	 \draw[thick] (right) edge[post,bend left=10] (down);
	 \draw[thick] (down) edge[post,bend left=10] (left);
	 \draw[thick] (left) edge[post,bend left=10] (up);
	 \node[place] (L) at (1,-12) {\texttt{L}};
	 \node[place] (R) at (4,-12) {\texttt{R}};
	 \draw[thick] (left) edge[post] (L);
	 \draw[thick] (right) edge[post] (R);
	 \node[transition] (append) at (3,-14){\texttt{Merge}};
	 \draw[thick] (L) edge[post,bend right=10] (append);
	 \draw[thick] (R) edge[post,bend left=10] (append); 
	 \draw[thick] (append) edge[post, bend right=87] (res);
	 \draw[dashed] (nbgen) edge[post,bend right=10] (left);
	 \draw[dashed] (nbgen) edge[post,bend left=20]  (right);
	 \node[transition] (extract) at (7,-6.5) {
	 	\begin{tabular}{c}
	 		\texttt{if} \\
	 		\texttt{Nb = NbGen}
	 	\end{tabular}	
	 };
  \draw[thick] (bal3) edge[post,bend left = 20] (extract);
  \draw[thick](extract) edge[post,bend right=40] (bal3);
 \draw[thick] (nbgen) edge[post, bend left=10] (extract);
 \draw[thick] (res) edge[post,bend left=10] (extract);
 \node[place,dashed] (fareyres) at (11,-6.5) {\texttt{FareyRes}};
 \draw[thick] (extract) edge[post] (fareyres);
 \draw[thick] (liftinfo) edge[post,bend left=20] (extract);
	\end{tikzpicture}
 }
	\caption{\textbf{Subnet ParallelFarey}}
	\label{fig:parallelfarey}
\end{figure}

\subsubsection{Subnet ParCompatible}
Now, we describe the subnet \texttt{ParCompatible}, see Figure \ref{fig:parcompatible}. The subnet is initialized with a structureless token in the place \texttt{U}.\\

\noindent \underline{Transition \texttt{Split}:}\\
The transition \texttt{Split} consumes tokens from the place \texttt{FareyRes} and the place \texttt{Test}. The transition adds the field (i) and (ii) of the token consumed from \texttt{Test} to the fields of the the token consumed from \texttt{FareyRes} and rename them to (vi) and (vii), respectively. The newly formed token is placed on the place \texttt{LiftInfo}. The transition \texttt{Split}, splits the generators, similar to the subnet \texttt{ParallelFarey} in the field (i) of the consumed token from \texttt{FareyRes} into different tokens with a field containing a single generator. Lastly it place a token containing the number of generators in the field (i) of the consumed token from \texttt{FareyRes} on the place \texttt{NbGen}. The transition Split also updates the boolean
value in the token on the place \texttt{Bal3} to true as seen in the transition \texttt{Split} of the subnet \texttt{ParallelFarey}, see Figure \ref{fig:parallelfarey}.\\

\noindent \underline{Transition \texttt{Compatible}:}\\
The transition \texttt{Compatible} consumes tokens from the place \texttt{Gen} and reads the prime $p$ that is contained in the field (vii) of the token on the place \texttt{LiftInfo}, that is the prime that was in the field (ii) of the token that was consumed by \texttt{Split} from the place \texttt{Test}. The transition then test whether the numerator or denominator of any coefficient of the polynomial in the field (i) of the token consumed from \texttt{Gen} is divisible by $p$. If this is the case the transition creates a token with a boolean field that is \texttt{false}, otherwise it creates a token with a boolean field that is \texttt{true}. A field \texttt{Nb} is added to the newly created token, with the number of generators in the field (i) of the token consumed from \texttt{Gen} (which is in this case $1$). The token is then placed on the place \texttt{Res}.\\

\noindent \underline{Transitions \texttt{if Nb < NbGen}:}\\
Similar to the subnet \texttt{ParallelFarey}, the transition with condition \texttt{ Nb < NbGen} consumes a token from \texttt{Res}, by turn, only if the value in the field \texttt{Nb} is strictly smaller than the value of the token on the place \texttt{NbGen}. It then places the token on the place \texttt{L} or \texttt{R}, respectively.\\

\noindent \underline{Transition \texttt{And}:}\\
The transition combines a token from the place \texttt{L} and a token from the place \texttt{R} into a token with the same fields, by adding the numbers in the fields \texttt{Nb} of the two tokens, and by placing the value $b_L\land b_R$ in the boolean field, where $b_L$ and $b_R$ is the values in the boolean fields of the tokens, consumed from the places \texttt{L} and \texttt{R}, respectively. The token is then placed on the place \texttt{Res}.\\

\noindent \underline{Transition \texttt{if Nb = NbGen}:}\\
The transition \texttt{if Nb = NbGen} only fires when the number contained in the field \texttt{Nb} of a token on \texttt{Res} is equal to the number contained in the token on the place \texttt{NbGen}. If this is the case there will be only one token on the place \texttt{Res}. The transition then consumes the token from the place \texttt{NbGen}, the token from the place \texttt{LiftInfo} and the token from the place \texttt{Res}. The token that was consumed from \texttt{NbGen} is discarded, the boolean field from the token that was consumed from \texttt{Res} is added to the token consumed from \texttt{LiftInfo}, which is then placed on the place \texttt{ResTest}. The transition if \texttt{Nb = NbGen} also sets the value on the place \texttt{Bal3} back to false.

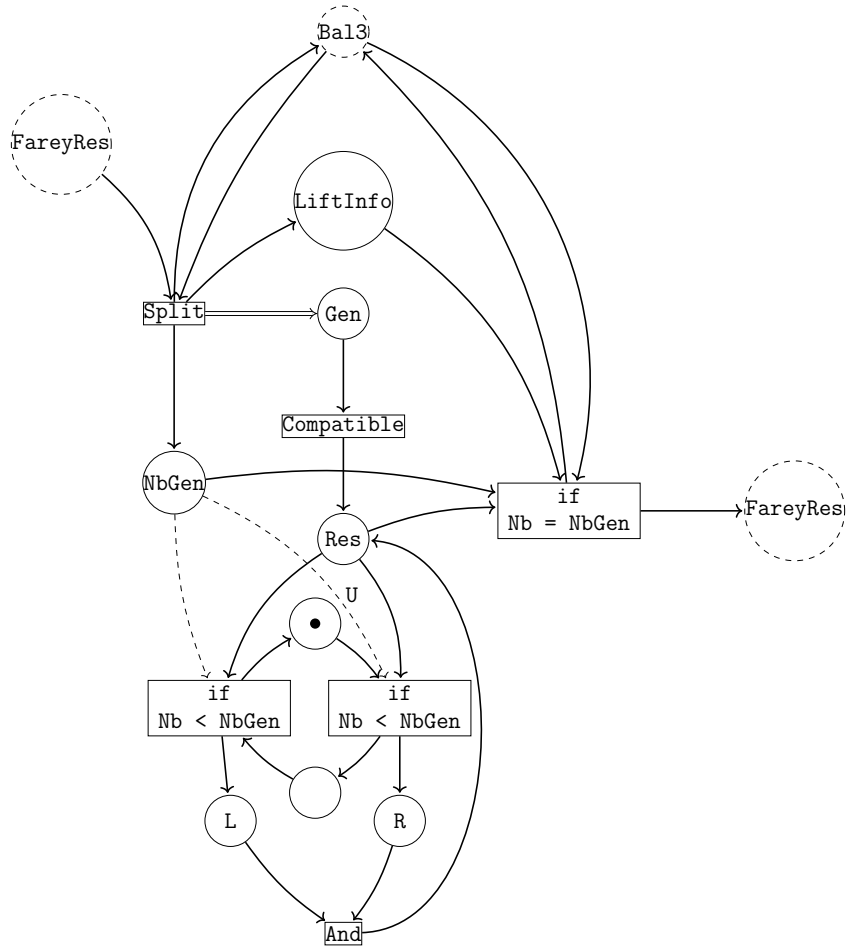
\begin{figure}[ht]
	\centering
 \scalebox{0.75}{
		\begin{tikzpicture}
			\node[place, dashed] (fareyres) at (-2,0) {\texttt{FareyRes}};
			\node[transition] (split) at (0,-3) {\texttt{Split}};
			\draw[thick] (fareyres) edge[post,bend left=20] (split);
			\node[place] (gen) at (3,-3) {\texttt{Gen}};
			\node[place] (nbgen) at (0,-6) {\texttt{NbGen}};
			\node[place] (liftinfo)  at (3,-1) {\texttt{LiftInfo}};
            \node[place,dashed] (bal3) at (3,2) {\texttt{Bal3}};
            \draw[thick](split) edge[post, bend left = 30] (bal3);
            \draw[thick] (bal3) edge[post,bend right= 10] (split);
 			\draw[-{Implies}, double,double distance=1.5pt, line width=0.5pt] (split) to (gen);
			\draw[thick] (split) edge[post] (nbgen);
			\draw[thick] (split) edge[post,bend left=10] (liftinfo);
			\node[transition] (compatible) at (3,-5) {\texttt{Compatible}};
			\node[place] (res) at (3,-7) {\texttt{Res}};
			\draw[thick] (gen) edge[post] (compatible);
			\draw[thick] (compatible) edge[post] (res);
			\node[transition] (left) at (0.8,-10) {
				\begin{tabular}{c}
					\texttt{if} \\
					\texttt{Nb < NbGen}
				\end{tabular}	
			};
			\node[transition] (right) at (4,-10) {
				\begin{tabular}{c}
					\texttt{if} \\
					\texttt{Nb < NbGen}
				\end{tabular}	
			};
			\node[place,tokens=1,label={30:$\texttt{U}$}] (up) at (2.5,-8.5) {};
			\node[place] (down) at (2.5,-11.5) {};
			\draw[thick] (res) edge[post,bend right=20] (left);
			\draw[thick] (res) edge[post,bend left=20] (right);
			\draw[thick] (up) edge[post,bend left=10] (right);
			\draw[thick] (right) edge[post,bend left=10] (down);
			\draw[thick] (down) edge[post,bend left=10] (left);
			\draw[thick] (left) edge[post,bend left=10] (up);
			\node[place] (L) at (1,-12) {\texttt{L}};
			\node[place] (R) at (4,-12) {\texttt{R}};
			\draw[thick] (left) edge[post] (L);
			\draw[thick] (right) edge[post] (R);
			\node[transition] (and) at (3,-14){\texttt{And}};
			\draw[thick] (L) edge[post,bend right=10] (and);
			\draw[thick] (R) edge[post,bend left=10] (and); 
			\draw[thick] (and) edge[post, bend right=87] (res);
			\draw[dashed] (nbgen) edge[post,bend right=10] (left);
			\draw[dashed] (nbgen) edge[post,bend left=20]  (right);
			\node[transition] (extract) at (7,-6.5) {
				\begin{tabular}{c}
					\texttt{if} \\
					\texttt{Nb = NbGen}
				\end{tabular}	
			};
            \draw[thick](extract) edge[post, bend right = 20] (bal3);
            \draw[thick] (bal3) edge[post,bend left= 40] (extract);
			\draw[thick] (nbgen) edge[post, bend left=10] (extract);
			\draw[thick] (res) edge[post,bend left=10] (extract);
			\node[place,dashed] (fareyres) at (11,-6.5) {\texttt{FareyRes}};
			\draw[thick] (extract) edge[post] (fareyres);
			\draw[thick] (liftinfo) edge[post,bend left=20] (extract);
	\end{tikzpicture}
 }
	\caption{\textbf{Subnet ParCompatible}}
	\label{fig:parcompatible}
\end{figure}
\subsection{Correctness and termination of the algorithm}\label{proofalg}
Using  notation from  Section \ref{sub:genericFramework}, we consider a deterministic algorithm $\mathcal A$ which takes as input an element $P\in
R_{0}\left\langle m_{1},\ldots,m_{s}\right\rangle $  and computes an element $$Q(0)=\prod_{i=1}^{t}Q_i\in R_{0}\left\langle
n_{1},\ldots,n_{t}\right\rangle $$ such that each $Q_{i}$, for $i=1,\ldots,t$, is
a reduced Gr\"{o}bner basis with respect to a global monomial ordering $>_{i}$ on
$R_0^{n_{i}}$. The algorithm under consideration also needs to be applicable over $R_{p}$ for $p$ a prime. Moreover, we assume that \emph{for any input $P\in
R_{0}\left\langle m_{1},\ldots,m_{s}\right\rangle $ there are only finitely many bad primes for the algorithm $\mathcal{A}$.}  This assumption is always satisfied if we deal with Buchberger based algorithms, see  \cite{arnold2003modular} and \cite{moller1984upper}. Now, fix an input $P\in
R_{0}\left\langle m_{1},\ldots,m_{s}\right\rangle $.
First, we generate primes that are compatible with  the input $P$ in the transition \texttt{GenPrime}. That is, primes that are coprime with all the numerators and denominators of coefficients in $P$ (in their coprime representation). Then, the transition \texttt{Compute} applies the algorithm $\mathcal{A}$ to the reductions $P_p$ for all generated primes $p$. In the subnet \texttt{Lift}, the modular results are lifted to a $P_N$ for $N \geq \frac{M_1}{2}$ via the Chinese Remainder Theorem. This is done by recursively lifting two modular results with the same lead monomials into one. Thereafter, in the subnet \texttt{Lift/Add}, a modular result $P_N$ for $N \geq \frac{M_1}{2}$ gets lifted, respecting the leading monomials, with the accumulator or  to a reasonable size with another modular result $P_{N'}$ for $N' \geq \frac{M_1}{2}$. Eventually, modular results $P_N$ are combined with the accumulator, and the result is lifted to $Q \in R_{0}\langle n_{1},\ldots,n_{r}\rangle$ via the error tolerant reconstruction in the subnet \texttt{ParallelFarey}. Then, the characteristic zero result $Q$ is tested against a modular result with one prime via the pTest in the subnet \texttt{ReconsTest}. If the  pTest is successful, the result $Q$ is tested with a final test in the  transition \texttt{Verify}.  If $Q$ passes the verification, we return it  and stop the framework. In the case where pTest or the verification fails, we set the accumulator to $P_N$ and the framework continue to run. 

Now, suppose that the pTest fails. So, we set the accumulator to $P_N$. If the lead monomials of $P_N$ disagrees with the lead monomials of a new modular result $P_{N'}$ with $N' \geq \frac{M_1}{2}$, then we throw away the value of the  accumulator and use $P_{N'}$. In this manner, the framework is able to remove accumulated bad primes. This defines a weighted majority vote. As there are only  finitely many bad primes, for a sufficiently large set of primes used, the error tolerant reconstruction will eventually lift a modular result $P_N$ to the expected output $Q(0) \in R_0\langle n_1,\dots,n_t \rangle$, and the framework will terminate.

\section{Timings}\label{sec:6}

In this section, we provide examples from Gröbner bases and birational geometry on which we time the implementation of the coordination layer discussed in Section \ref{sec:5}. 

The implementation is based on the \textsc{Singular}-\textsc{GPI-Space} framework for massively parallel computations in computer algebra \cite{singgspc}. While the task manager \textsc{GPI-Space} (version 23.06) is used as an implementation of a coordination language based on Petri nets, we rely on the computer algebra system \textsc{Singular} (version 4.3.0) as the computation model. Moreover, we have realized a user interface based on \textsc{Singular} interface for our application, so that the computation can be triggered with input data as well as output data given in \textsc{Singular}. 

Unless specified, timings are conducted on the cluster Beehive of the Fraunhofer Institute for Industrial Mathematics ITWM. Each node is outfitted with the following hardware components: two Intel Xeon Gold 6240R CPUs at 2.40 GHz (with a total of 48 cores without hyperthreading), 64 GB RAM, a 480 GB SSD, 10 Gigabit Ethernet and HDR100 Infiniband. The nodes run under  Oracle Linux 8 (which is derived from Red Hat Enterprise Linux), and  are connected via FDR Infiniband. The implementation and examples are available at the GitHub repository \cite{modular}.

All timings are provided in seconds. In all the tables, the term "no. of cores" (respectively "no. of primes") refers  the number of cores  (respectively primes) used during a running process. Additionally, the symbol "---" means no applicable due to the limitation of the implementation on one node, that is lack of coordination layer to run the algorithm on more than one node. 

In the first section, we investigate an application of the generic  framework in the computation of reduced  Gr\"{o}bner bases by comparing its timing with existing algorithms in \textsc{Singular}. Then, in the second section, we give an application in the computation of image of rational maps. The timings are compared with plain \textsc{Singular} implementations using task parallelism if available. Moreover, we present how the implementation scales with the number of nodes.

All examples are chosen to emphasize the performance of the coordination separate computation implementation.

\subsection{Gr\"{o}bner Bases}\label{sec:grobner}
In this section, we compute Gr\"{o}bner bases of ideals over characteristic $0$ via $3$ differents implementations: the massively parrallel  implementation \texttt{gspc\_modstd}, the implementation of modular algorithm of \textsc{Singular} \texttt{modStd} from the library \texttt{modstd.lib} and the direct computation over $\mathbb{Q}$ via the procedure \texttt{std}.  
The interface \texttt{gspc\_modstd} has the following signature 
\begin{verbatim}
    gspc_modstd(ideal I, configToken gc, int bal1, int bal2, int M1, int M2,
                int nb_primes)
\end{verbatim}
 where \texttt{I} is the input ideal; \texttt{gc} contains data required for the startup of \textsc{GPI-Space}; the integers \texttt{bal1}, \texttt{bal2}, and \texttt{nb\_primes} represent in the modular net Figure~\ref{fig:modular} the value of \texttt{m}, the initial number of structureless tokens on the place \texttt{Bal2}, and  the intial number of tokens on the place \texttt{Primes}, respectively; the integer \texttt{M1} (resp. \texttt{M2}) corresponds to the value of \texttt{M}$_1$ in the subnet \texttt{Lift} Figure~\ref{fig:lift} (resp. the value of \texttt{M}$_2$ in the subnet \texttt{Lift/Add} Figure~\ref{fig:liftadd}). The last five arguments are optional. For the timing, we use the default values given by 
 $$
 \label{eq:configuration}
(\texttt{bal1}, \texttt{bal2}, \texttt{M1}, \texttt{M2}, \texttt{nb\_primes}) = \left(n-1, \lceil 0.83 n \rceil, n-1, n-1, 2n-2 \right) $$
 with $n$  the number of cores  assigned to the framework and $\lceil. \rceil$ the ceiling function.
For timing purpose, we consider the following  examples. 
\begin{exa}
   We considered a homogeneous ideal generated by $5$ quatrics in a $5$ variables polynomial ring over $\mathbb{Q}$. The ideal is generated by \textsc{Sinigular} via the procedure \texttt{randomid} in the library \texttt{random.lib} in such a  way that all generators coefficients are integers in the interval $[-20,20]$. We use the lexicographic ordering \texttt{lp} as monomial ordering. The running time are summarized in Table~\ref{tab:randomstd}. Note that the $48$ cores timing for \texttt{gspc\_modstd} is more than the corresponding timing for \texttt{modStd}. This can be explained by the fact that\texttt{gspc\_modstd}  uses more primes during the computation. In this example, the computation of the reduced Gr\"{o}bner bases in characteristic a prime $p$ is  faster  compared to the subsequent two examples, so the bottle neck of the computation is acually the lifting and testing.
\end{exa}
\begin{table}[ht]
	\begin{center}
		\begin{tabular}{c|c c|cc|c}
			\multirow{2}{4em}{no. of cores} & \multicolumn{2}{|c|}{\texttt{gspc\_modstd}} & \multicolumn{2}{|c|}{\texttt{modStd}} & \texttt{std} \\
			 \cline{2-6}
    & time & no. of primes&  time &no. of primes& time \\  \hline
   --- &---&---&---&---& 3201 \\
   48 & 81&503 &64 & 384 & --- \\
   96 & 38 & 503 & ---& ---& --- \\
   \hline
			
		\end{tabular}
	\caption{Running time in seconds of  a lexicographic Gr\"{o}bner basis for the complete intersection of $5$ homogeneous quartics in 5 variables with  coefficient size 20 over $\mathbb{Q}$. }
    \label{tab:randomstd}
	\end{center}
 \end{table}
\begin{exa} 
Consider the cyclic $8$ ideal in an 8 variable polynomial ring over $\mathbb{Q}$. The timings for the computation of its reduced Gr\"{o}bner basis with respect the degree reverse lexicographical ordering \texttt{dp} are given in  Table \ref{tab:cyclic}. Both modular algorithms relatively use the same number of prime during their executions but the implementation \texttt{gspc\_modstd} shows better performance. Indeed, the implementation \texttt{modStd} does not accumulate modular results throughout the characteristic $p$ Gröbner basis computation. That is, it has to finish lifing all the current characteristic $p$  results in order to generate the next batch of characteristic $p$ results, in case the algorithm does not stop. In this example, we note that a characterstic $p$ result takes  a longer time to compute compared to the previous example. 
\end{exa}

\begin{table}[ht]
	\begin{center}
		\begin{tabular}{c|c c|cc|c}
			\multirow{2}{4em}{no. of cores} & \multicolumn{2}{|c|}{\texttt{gspc\_modstd}} & \multicolumn{2}{|c|}{\texttt{modStd}} & \texttt{std} \\
			 \cline{2-6}
    & time & no. of primes&  time &no. of primes& time \\  \hline
      1& --- & ---& ---& --- &  $>3d$\\
    48& 146 & 73& 207& 78 &  ---\\
   96 & 80 & 75  & ---& ---& --- \\
   \hline
			
		\end{tabular}
	\caption{Running time in seconds for a Gr\"{o}bner basis of the cyclic $8$ ideal in $8$ variables over $\mathbb{Q}$ with respect to the degree reverse lexicographical ordering  \texttt{dp}. }
 \label{tab:cyclic}
	\end{center}
 \end{table}
\begin{exa}
    Consider the kastura 11 ideal over $\mathbb{Q}$. The timings for the computation of  its reduced Gr\"{o}bner basis with respect the degree reverse lexicographical ordering \texttt{dp} are provided in  Table~\ref{tab:katsura}.  The implementation \texttt{gspc\_modstd} demonstrates superior performance when compared to the other two implementations \texttt{modStd} and \texttt{std}. In this example, we observe that a characteristic $p$ result requires a longer time for computation.
\end{exa}
 \begin{table}[ht]
	\begin{center}
		\begin{tabular}{c|c c|cc|c}
			\multirow{2}{4em}{no. of cores} & \multicolumn{2}{|c|}{\texttt{gspc\_modstd}} & \multicolumn{2}{|c|}{\texttt{modStd}} & \texttt{std} \\
			 \cline{2-6}
    & time & no. of primes&  time &no. of primes& time \\  \hline
       1& --- &--- & ---& --- &  1589\\
    48& 326 &72 & 405& 78 &  ---\\
   96 &  187& 72  & ---& ---& -- \\
   \hline
			
		\end{tabular}
	\caption{Running time in seconds for a Gr\"obner basis of the Katsura $11$ ideal  over $\mathbb{Q}$  with respect to the degree reverse lexicographical ordering  \texttt{dp}. }
 \label{tab:katsura}
	\end{center}
 \end{table}

\subsection{Computations for Rational Maps}
In this section, we time the computation of the image of rational maps using the new massively parallel modular algorithm implementation \texttt{gspc\_modimage} and compare the results with the modular algorithm \texttt{ModImageMap} in textsc{Singular} and the  direct computation implementation \texttt{ImageMap} from \cite{modularimage}. 
In all the example, we fix the quintic plane curve ideal 
$$
\begin{array}{c}
I = t_1^5+10t_1^4t_2+20t_1^3t_2^2+130t_1^2t_2^3-20t_1t_2^4+20t_2^5-2t_1^4t_0 \\
-40t_1^3t_2t_0-150t_1^2t_2^2t_0-90t_1t_2^3t_0-40t_2^4t_0+t_1^3t_0^2\\
+30t_1^2t_2t_0^2+110t_1t_2^2t_0^2+20t_2^3t_0^2
\end{array}$$
of $\mathbb{Q}[t_0,t_1,t_2]$ and consider various rational maps $$\Phi :  V(I) \subset \mathbb{P}_\mathbb{Q}^2  \dashrightarrow \mathbb{P}_\mathbb{Q}^n \text{ for } n \geq 1.$$ The function \texttt{gspc\_modimage} has the signature
\begin{verbatim}
    gspc_modimage(ideal phi, ideal I, configToken gc, int bal1, int bal2,
                  int M1, int M2, int nb_primes)
\end{verbatim} 
where \texttt{phi} is the ideal generated by the homogeneous polynomials defining the rational map; \texttt{I} is the homogeneous ideal defining the source projective variety; \texttt{gc} is a configuration data required for the startup of \textsc{GPI-Space}. The remaining parameters have the same meaning  as in Section \ref{sec:grobner}.

\begin{exa}
For $5 \leq d \leq 7$. Consider the degree $d$ Veronese embedding  rational maps $\Phi_d: V(I) \subset \mathbb{P}_\mathbb{Q}^2 \dashrightarrow \mathbb{P}_\mathbb{Q}^{n}$, with $n= \binom{2+d}{2}$, defined by all the monomials of degree $d$ in the variables $t_0,t_1$ and  $t_2$.
We do timings for this example on a typical user machine, an  Intel i7-1265U with 10 cores, max. 4.8GHz, 32GB RAM, 2 performance cores, 8 efficient cores and 12 threads.  The timings are summarized in Table \ref{tab:veronese}. Computing the image under the Veronese becomes very difficult with increasing degree, so any improvement is worthwhile. Although the monomial map is not very well suited for modular methods, we nevertheless see a twofold speedup for the degree $7$ Veronese. 

\end{exa}

\begin{table}[ht]
	\begin{center}
		\begin{tabular}{c|c|c c|cc|c}
  &\multirow{2}{0.7em}{} & \multicolumn{2}{|c|}{\texttt{gspc\_modimage}} & \multicolumn{2}{|c|}{\texttt{modimage}} & \texttt{image$/\mathbb{Q}$} \\
			&\multirow{2}{0.7em}{$n$ \,\,\,} & \multicolumn{2}{|c|}{(10 cores)} & \multicolumn{2}{|c|}{\texttt{(10 cores)}} & \texttt{(1 core)} \\
			 \cline{3-7}
 deg &   & time & no. of primes&  time &no. of primes& time \\  \hline
 5 &  20& 42 & 5 & 42& 11 &  41\\ 
 6 & 27&491 &5 &465 & 11&742 \\ 
 7 &35&3560&5&3360&11&6990 \\
   \hline	
		\end{tabular}
	\caption{Running time in seconds for the computation of the images of degree~$d$ Veronese embeddings of a quintic plane curve.}
 \label{tab:veronese}
	\end{center}
 \end{table}
 \begin{exa}
We consider two rational maps $V(I) \subset \mathbb{P}_\mathbb{Q}^2\dashrightarrow \mathbb{P}_\mathbb{Q}^2$  defined by  random homogeneous polynomials with integer coefficients in the interval $[-2000, 2000]$. The first is defined by homogeneous polynomials of degree $2$ and the second by homogeneous polynomials of degree $3$. The timings are given in Table~\ref{tab:randommap}.
 \end{exa}

 We observe that for small examples, as well as examples which only require a small number of primes (or are run on limited resources) and, hence, only can make use of a limited amount of parallelism, the Petri net based scalable implementation is as fast as the one which is based on Singular task parallelism and forking and thus is limited to a single machine. The modular approach is significantly faster than the computation over the rationals. For examples with non-trivial coefficients in the map, we see a speedup of $1000$ or more.

\begin{table}[ht]
	\begin{center}
		\begin{tabular}{c|c|c c|cc|c}
  &\multirow{2}{0.7em}{\,\, \,\,\,\,\,$d$ } & \multicolumn{2}{|c|}{\texttt{gspc\_modimage}} & \multicolumn{2}{|c|}{\texttt{modimage}} & image$/\mathbb{Q}$ \\
	$n$		&\multirow{2}{0.7em}{} & \multicolumn{2}{|c|}{\texttt{(10 cores)}} & \multicolumn{2}{|c|}{\texttt{(10 cores)}} & \texttt{(1 core)} \\
			 \cline{3-7}
 &   & time & no. of primes&  time &no. of primes& time \\  \hline
 3 & 2& 2.9 & 18 & 5.5& 20 &  2445\\ 
 3 & 3&358 &35 &417& 36& $>$6h \\ 
   \hline	
		\end{tabular}
	\caption{Running time in seconds for the computation of the images of  quintic plane curve under rational map defined by $n$ homogeneous polynomials of degree $d$ and coefficient size less or equal to 2000.}
 \label{tab:randommap}
	\end{center}
 \end{table}

 \begin{exa}
    Lastly, we investigate how the new implementation \texttt{gspc\_modimage} scales with the number of cores. We consider a rational map $\Phi: V(I) \subset \mathbb{P}_K^2 \dashrightarrow \mathbb{P}^4$ given by homogeneous polynomials of degree 3  and  integer coefficients in the interval $[-2000, 2000].$ The timings are given in Table~\ref{tab:scale}. We notice that the 144 cores timing terminates with 216 primes accumulated. This is considerably less than the number of primes used by the 96 cores timing or the 192 cores timing. This is explained by the feature of the subnet \texttt{Lift} (see Figure \ref{fig:lift}) to release token containing less that $144$ primes.\\
    \begin{table}[ht]
	\begin{center}
		\begin{tabular}{c|c c}
			\multirow{2}{4em}{Number of cores} & \multicolumn{2}{|c}{\texttt{gspc\_modimage}}  \\
			 \cline{2-3}
    & time & number of primes\\ \hline
    48& 116 &216 \\
   96&46&286\\
   144&25&216\\
    192&17 &288
			
		\end{tabular}
	\caption{Scaling of the run-times of \texttt{gspc\_modimage} for the computation of the image  of a quintic plane curve by  a rational map into $\mathbb{P}^4 $ given by random polynomials of degree $3$ and absolute coefficient size $\leq 2000$.}
 \label{tab:scale}
	\end{center}
 \end{table}

    In Figure \ref{fig:runtime}, we show the run-times from Table~\ref{tab:scale} against the number cores as well as in red the  least-square exponential curve that fits the run-times.  Then, Figure~\ref{fig:speedup} visualizes the speedup factor (relative to the  timing using only one node with $48$ cores, which we normalize to speedup $48$) against the number of cores. Moreover, the  $y$-value is calibrated in such a way that the 48 cores speedup corresponds to 48. We notice a linear speed up. A visualization of the parallel efficiency of the timings from Table \ref{tab:scale} is given in Figure~\ref{fig:parallelefficiency}. We also plot the least-square fitting curve of the parallel efficiency  using a parabola. We recall that the parallel efficiency is defined as the ratio of the speedup to the number of cores.

    We observe that the implementation shows a parallel efficiency of greater than one, which is noteworthy since parallel efficiencies, in particular for fine grain parallelism, are typically below one. Similar effects have been observed in other setting for example in \cite{boehm2018massively}. We attribute the observed effect to the ability of the Petri net to intertwine the different algorithmic components (generate, compute, compatibility check, Chinese remainder lift, Farey lift, comparison) more efficiently, and hence find a faster way to the solution.  
\end{exa}

 \begin{figure}[ht]
     \centering
     \includegraphics[width=\textwidth]{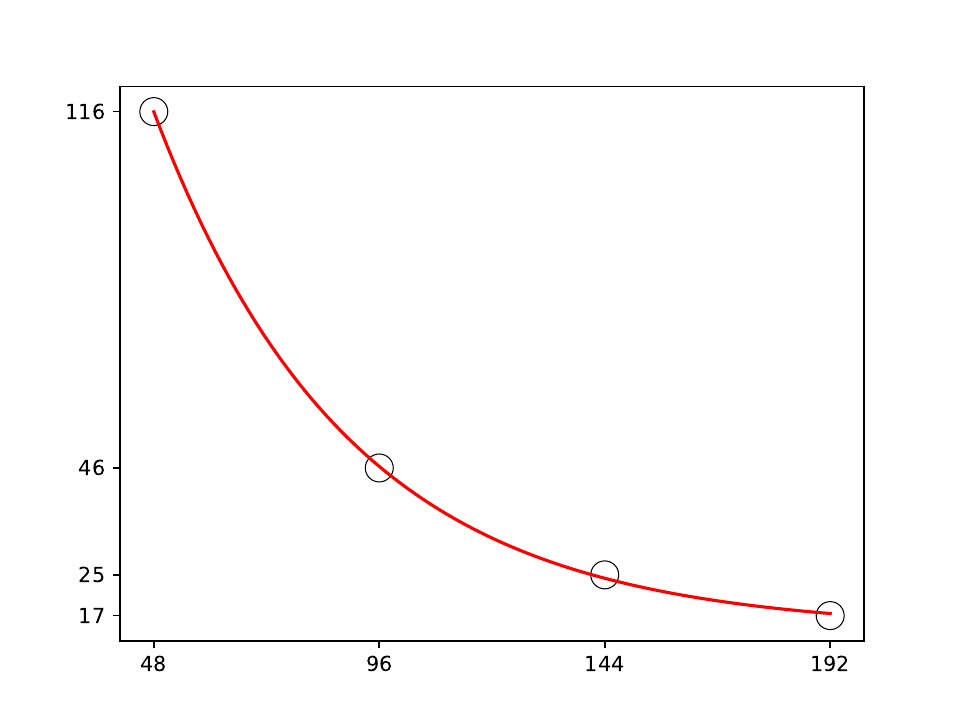}
     \caption{Run-times from Table~\ref{tab:scale} (in seconds) against the number of cores for the computation of the image  of a quintic plane curve by  a rational map into $\mathbb{P}^4 $ given by random polynomials of degree $3$ and absolute coefficient size $\leq 2000$.}
     \label{fig:runtime}
 \end{figure}

\begin{figure}[ht]
     \centering
     \includegraphics[width=\textwidth]{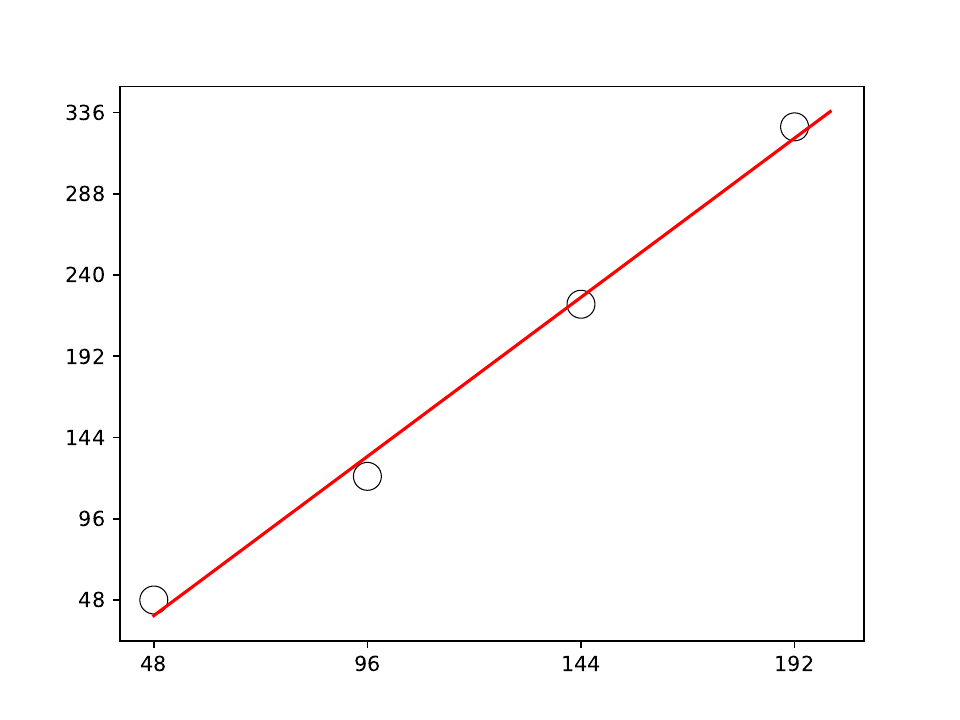}
     \caption{Speedup against the number of cores for the computation of the image  of a quintic plane curve by  a rational map into $\mathbb{P}^4 $ given by random polynomials of degree $3$ and absolute coefficient size $\leq 2000$.}
     \label{fig:speedup}
 \end{figure}

 \begin{figure}[ht]
     \centering
     \includegraphics[width=\textwidth]{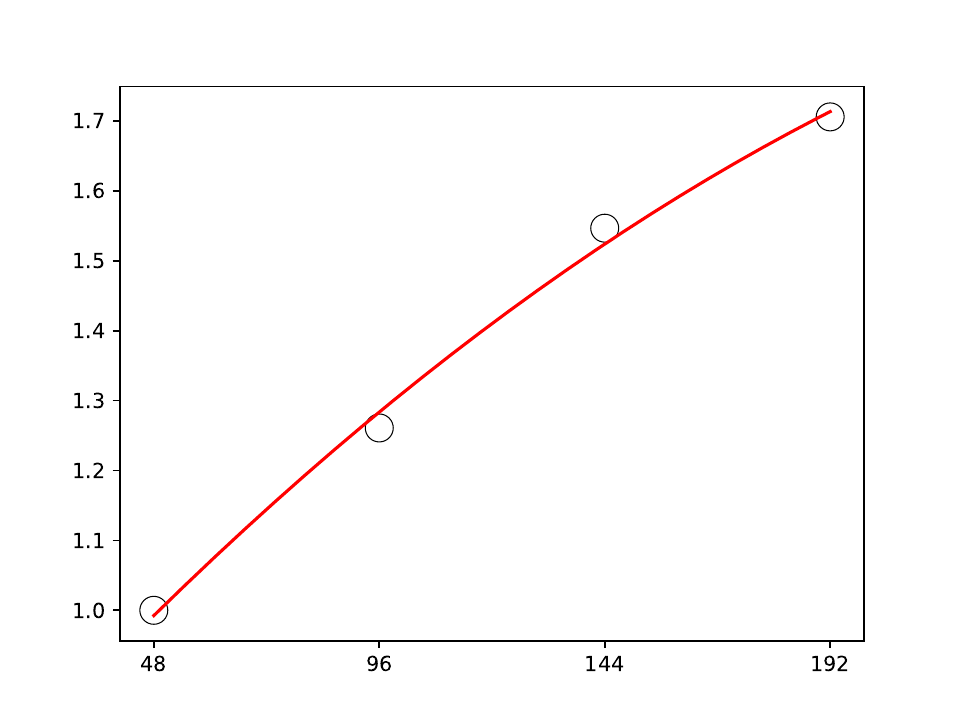}
     \caption{Parallel efficiency obtained from the run-times in Table~\ref{tab:scale} against the number of cores for the computation of the image of a quintic plane curve by  a rational map into $\mathbb{P}^4 $ given by random polynomials of degree $3$ and absolute coefficient size $\leq 2000$.} 
     \label{fig:parallelefficiency}
 \end{figure}

\section{Future Developments}
Current work in progress includes the application of the framework for the computation of syzygies, Schreyer resolutions, ideal and module operations such as intersections and quotients, complete intersection covers of smooth schemes, as well as global sections of divisors.

\bibliography{reference}

\begin{thebibliography}{}

\bibitem[Arnold, 2003]{arnold2003modular}
Arnold, E.~A. (2003).
\newblock Modular algorithms for computing gr{\"o}bner bases.
\newblock {\em Journal of Symbolic Computation}, 35(4):403--419.

\bibitem[Bendle et~al., 2024a]{boehmPFD}
Bendle, D., Boehm, J., Heymann, M., Ma, R., Rahn, M., Ristau, L., Wittmann, M.,
  Wu, Z., Xu, H., and Zhang, Y. (2024a).
\newblock {pfd-parallel, a Singular/GPI-Space package for massively parallel
  multivariate partial fractioning}.
\newblock {\em Comput. Phys. Commun.}, 294:108942.

\bibitem[Bendle et~al., 2024b]{bendle2020parallel}
Bendle, D., Böhm, J., Ren, Y., and Schröter, B. (2024b).
\newblock Massively parallel computation of tropical varieties, their positive
  part, and tropical grassmannians.
\newblock {\em Journal of Symbolic Computation}, 120:102224.

\bibitem[Boehm, 2020]{boehm2017}
Boehm, J. (2020).
\newblock {\em {Computer Algebra}}.
\newblock Lecture Notes.

\bibitem[B{\"o}hm et~al., 2015]{Boehm2012}
B{\"o}hm, J., Decker, W., Fieker, C., and Pfister, G. (2015).
\newblock The use of bad primes in rational reconstruction.
\newblock {\em Mathematics of Computation}, 84(296):3013--3027.

\bibitem[B\"{o}hm et~al., 2021]{boehm2018massively}
B\"{o}hm, J., Decker, W., Fr\"{u}hbis-Kr\"{u}ger, A., Pfreundt, F.-J., Rahn,
  M., and Ristau, L. (2021).
\newblock Towards massively parallel computations in algebraic geometry.
\newblock {\em Found. Comput. Math.}, 21(3):767–806.

\bibitem[B\"{o}hm and Fr\"{u}hbis-Kr\"{u}ger, 2021]{boehmissac}
B\"{o}hm, J. and Fr\"{u}hbis-Kr\"{u}ger, A. (2021).
\newblock Massively parallel computations in algebraic geometry.
\newblock In {\em Proceedings of the 2021 on International Symposium on
  Symbolic and Algebraic Computation}, ISSAC '21, page 11–14, New York, NY,
  USA. Association for Computing Machinery.

\bibitem[Bourbaki, 1998]{Bourbaki1-3}
Bourbaki, N. (1998).
\newblock {\em Algebra I Chapters 1-3}.
\newblock Springer.

\bibitem[Collins and Encarnación, 1995]{COLLINS1995287}
Collins, G.~E. and Encarnación, M.~J. (1995).
\newblock Efficient rational number reconstruction.
\newblock {\em Journal of Symbolic Computation}, 20(3):287--297.

\bibitem[Cox et~al., 2007]{David2007}
Cox, D., Little, J., and O'Shea, D. (2007).
\newblock {\em {Ideals, Varieties, and Algorithms: An Introduction to
  Computational Algebraic Geometry and Commutative Algebra}}.
\newblock Springer, 4 edition.

\bibitem[Decker et~al., 2022]{Singular}
Decker, W., Greuel, G.-M., Pfister, G., and Sch\"onemann, H. (2022).
\newblock {\sc Singular} {4-3-0} --- {A} computer algebra system for polynomial
  computations.
\newblock \url{http://www.singular.uni-kl.de}.

\bibitem[Fraunhofer~ITWM, 2024]{gpispace}
Fraunhofer~ITWM, C. C. H. P.~C. (2024).
\newblock {GPI}-space.
\newblock https://www.gpi-space.de.

\bibitem[Gelernter and Carriero, 1992]{gelernter1992coordination}
Gelernter, D. and Carriero, N. (1992).
\newblock Coordination languages and their significance.
\newblock {\em Communications of the ACM}, 35(2):96.

\bibitem[Greuel et~al., 2008]{greuel2008singular}
Greuel, G.-M., Pfister, G., Bachmann, O., Lossen, C., and Sch{\"o}nemann, H.
  (2008).
\newblock {\em A Singular introduction to commutative algebra}, volume 348.
\newblock Springer.

\bibitem[Idrees et~al., 2011]{idrees2011parallelization}
Idrees, N., Pfister, G., and Steidel, S. (2011).
\newblock Parallelization of modular algorithms.
\newblock {\em Journal of Symbolic Computation}, 46(6):672--684.

\bibitem[Kornerup and Gregory, 1983]{Kornerup1983}
Kornerup, P. and Gregory, T. (1983).
\newblock {Mapping integers and Hensel codes onto Farey fractions}.
\newblock {\em BIT Numerical Mathematics}, 23:9--20.

\bibitem[Lazić, 2023]{lazic2023programming}
Lazić, V. (2023).
\newblock Programming the minimal model program: a proposal.

\bibitem[Liu, 2006]{Liu2006}
Liu, Q. (2006).
\newblock {\em {Algebraic geometry and arithmetic curves}}.
\newblock Oxford graduate texts in mathematics, 6. Oxford University Press.

\bibitem[Micali, 1964]{micali1964algebres}
Micali, A. (1964).
\newblock Sur les algebres universelles.
\newblock In {\em Annales de l'institut Fourier}, volume~14, pages 33--87.

\bibitem[M{\"o}ller and Mora, 1984]{moller1984upper}
M{\"o}ller, H.~M. and Mora, F. (1984).
\newblock Upper and lower bounds for the degree of gr{\"o}bner bases.
\newblock In {\em International Symposium on Symbolic and Algebraic
  Manipulation}, pages 172--183. Springer.

\bibitem[Peterson, 1981]{peterson1981}
Peterson, J.~L. (1981).
\newblock {\em Petri net theory and the modeling of systems}.
\newblock Prentice Hall PTR.

\bibitem[Rakotoarisoa, 2023]{modularimage}
Rakotoarisoa, H.~P. (2023).
\newblock essay-code.
\newblock \url{https://bitbucket.org/hobihasina/singular-code/src/master/}.

\bibitem[Simis, 2004]{Simis2004}
Simis, A. (2004).
\newblock {Cremona transformations and some related algebras}.
\newblock {\em Journal of Algebra}, 280(1):162--179.

\bibitem[{The gspc-modular team}, 2023]{modular}
{The gspc-modular team} (2023).
\newblock gspc-modular.
\newblock \url{https://github.com/singular-gpispace/modular}.

\bibitem[{The Singular/{GPI-S}pace framework team}, 2023]{singgspc}
{The Singular/{GPI-S}pace framework team} (2023).
\newblock Singular/{GPI-S}pace framework.
\newblock \url{https://agag-jboehm.math.rptu.de/~boehm/singulargpispace/}.

\end{thebibliography}
\bibliographystyle{apalike}
\end{document}